\documentclass[12pt]{amsart}
\usepackage{ amsmath, amsthm, amsfonts, amssymb, color}
 \usepackage{mathrsfs}
\usepackage{amsfonts, amsmath}
 \usepackage{amsmath,amstext,amsthm,amssymb,amsxtra}
 \usepackage{txfonts}
 \usepackage[colorlinks, citecolor=blue,pagebackref,hypertexnames=false]{hyperref}
 \allowdisplaybreaks
 \usepackage{pgf,tikz}
 \usepackage{fancybox}
\usepackage{tikz}
\usepackage{graphicx}
\usepackage{graphics}

\begin{document}
\baselineskip 17pt
\hfuzz=6pt

\newtheorem{theorem}{Theorem}[section]
\newtheorem{proposition}[theorem]{Proposition}
\newtheorem{coro}[theorem]{Corollary}
\newtheorem{lemma}[theorem]{Lemma}
\newtheorem{definition}[theorem]{Definition}
\newtheorem{example}[theorem]{Example}
\newtheorem{remark}[theorem]{Remark}
\newcommand{\ra}{\rightarrow}
\renewcommand{\theequation}
{\thesection.\arabic{equation}}
\newcommand{\ccc}{{\mathcal C}}
\newcommand{\one}{1\hspace{-4.5pt}1}

\newtheorem*{TheoremA}{Theorem A}

\newtheorem*{TheoremB}{Theorem B}

 \def \Lips  {{   \Lambda}_{L}^{ \alpha,  s }(X)}
\def\BL {{\rm BMO}_{L}(X)}
\def\HAL { H^p_{L,{at}, M}(X) }
\def\HML { H^p_{L, {mol}, M}(X) }
\def\HM{ H^p_{L, {mol}, 1}(X) }
\def\Ma { {\mathcal M} }
\def\MM { {\mathcal M}_0^{p, 2, M, \epsilon}(L) }
\def\dMM { \big({\mathcal M}_0^{p,2, M,\epsilon}(L)\big)^{\ast} }
  \def\RR {  {\mathbb R}^n}
\def\HSL { H^p_{L, S_h}(X) }
\newcommand\mcS{\mathcal{S}}
\newcommand\mcB{\mathcal{B}}
\newcommand\D{\mathcal{D}}
\newcommand\C{\mathbb{C}}
\newcommand\N{\mathcal{N}}
\newcommand\R{\mathbb{R}}
\newcommand\Rf{\mathfrak{R}}
\newcommand\G{\mathbb{G}}
\newcommand\T{\mathbb{T}}
\newcommand\Z{\mathbb{Z}}
\newcommand\lp{L^p(M,\mu)}

\newcommand\CC{\mathbb{C}}
\newcommand\NN{\mathbb{N}}
\newcommand\ZZ{\mathbb{Z}}

\newcommand\bmo{{\rm BMO}_L(M)}

\renewcommand\Re{\operatorname{Re}}
\renewcommand\Im{\operatorname{Im}}

\newcommand{\mc}{\mathcal}

\def\SL{\sqrt[m] L}
\newcommand{\la}{\lambda}
\def \l {\lambda}
\newcommand{\eps}{\varepsilon}
\newcommand{\pl}{\partial}
\newcommand{\supp}{{\rm supp}{\hspace{.05cm}}}
\newcommand{\x}{\times}
\newcommand{\mar}[1]{{\marginpar{\sffamily{\scriptsize
        #1}}}}

\newcommand\wrt{\,{\rm d}}

\title[BMO space associated with operators on manifolds with ends ]
{BMO spaces associated to operators with generalised Poisson bounds on non-doubling manifolds with ends}

\author[P. Chen, X.T. Duong, J. Li, L. Song and L.X. Yan]{Peng Chen, \ Xuan Thinh Duong, \
 Ji Li, \ Liang Song \ and \ Lixin Yan
}
\address{Peng Chen, Department of Mathematics, Sun Yat-sen (Zhongshan)
University, Guangzhou, 510275, P.R. China}
\email{achenpeng1981@163.com}
\address{Xuan Thinh Duong, Department of Mathematics, Macquarie University, NSW 2109, Australia}
\email{xuan.duong@mq.edu.au}
\address{Ji Li, Department of Mathematics, Macquarie University, NSW, 2109, Australia}
\email{ji.li@mq.edu.au}
\address{Liang Song, Department of Mathematics, Sun Yat-sen (Zhongshan)
University, Guangzhou, 510275, P.R. China}
\email{songl@mail.sysu.edu.cn}
\address{
Lixin Yan, Department of Mathematics, Sun Yat-sen (Zhongshan) University, Guangzhou, 510275, P.R. China}
\email{mcsylx@mail.sysu.edu.cn
}

\date{\today}

\begin{abstract}
	Consider  a non-doubling manifold with ends $M = \mathfrak{R}^{n}\sharp\,  {\R}^{m}$
	where $\mathfrak{R}^n=\mathbb{R}^n\times \mathbb{S}^{m-n}$ for $m> n \ge 3$.
	We say that an operator $L$ has a generalised Poisson kernel if $\sqrt{ L}$ generates a semigroup $e^{-t\sqrt{L}}$
	whose kernel $p_t(x,y)$ has an upper bound similar to the kernel of $e^{-t\sqrt{\Delta}}$ where $\Delta$ is the Laplace-Beltrami
	operator on $M$. An example for operators with generalised Gaussian bounds is the Schr\"odinger operator $L = \Delta + V$
	where $V$ is an arbitrary non-negative locally integrable potential. In this paper, our aim is to introduce the BMO space $\bmo$ associated
	to operators with generalised Poisson bounds which serves as an appropriate setting for certain singular integrals with
	rough kernels to be bounded from $L^{\infty}(M)$ into this new $\bmo$. On our $\bmo$ spaces, we show that the
	John--Nirenberg inequality holds  and
	we show an interpolation theorem for a holomorphic family of operators which interpolates between	 $L^q(M)$
	 and $\bmo$. As an application, we show that the holomorphic functional calculus $m(\sqrt{L})$ is bounded
	from $L^{\infty}(M)$ into $\bmo$, and bounded on $L^p(M)$ for $1 < p < \infty$.

\end{abstract}

\maketitle


\section{Introduction}
\setcounter{equation}{0}

\subsection{Background and statement of main results}
The space BMO of functions of bounded mean oscillation   on ${\RR}$,  which was
 originally introduced by John and Nirenberg \cite{JN}  in the context
of partial differential equations, has been identified as the dual of classical Hardy space $H^1$ in the celebrated
work by Fefferman and Stein \cite{FS}.
Since then  the BMO function space and its predual $H^1$
are considered as the natural substitutions for the Lebesgue spaces $L^{\infty}$ and $L^{1}$ respectively in the
study of singular integrals and they
  are well established for spaces of homogeneous type $(X, d, \mu)$, i.e.   the underlying measure  $\mu$ satisfies the doubling
  (volume) property
\begin{equation}\label{Doubing}
\mu(B(x, 2r)) \le C \mu (B(x,r))
\end{equation}
for all the balls $B(x, r)$  with centre $x$ and radius $r$.
In this case,
assume that a singular integral operator $T$ is bounded on $L^2(X)$ and its associated kernel $k(x,y)$ satisfies the well known
H\"ormander condition, i.e. there exist constants $C>0$ and $c>1$ so that
\begin{equation} \label{Hormander2}
\int_{d(x,y) \ge c d(x,x_1)} |k(x,y) - k(x_1,y)| \,d\mu (y) \le C\nonumber
\end{equation}
for all $x, x_1\in X$ and
\begin{equation} \label{Hormander1}
\int_{d(x,y) \ge c d(y,y_1)} |k(x,y) - k(x,y_1)| \,d\mu (x) \le C\nonumber
\end{equation}
for all $y, y_1\in X$,  then $T$ is bounded from $L^{\infty}(X)$ into the  space ${\rm BMO}(X)$  and from its predual $H^1(X)$
into $L^1(X)$. By  interpolation, $T$ is bounded on $L^p(X)$ for
all $1 < p < \infty$ (see \cite{CW, Chr}). For various applications of BMO,
we refer the reader to     Stein's monograph  \cite[Chapter IV]{St2}  and references therein.

In the last two decades, the study of singular integrals beyond the Calder\'on--Zygmund operators has been extensive
and carried out in two directions: singular integrals on non-doubling spaces and singular integrals with rough kernels.

 \smallskip

{\it (i) Singular integrals on non-doubling spaces.}
The doubling property of the
underlying measure is a basic condition in the classical Calder\'on-Zygmund
theory of harmonic analysis. Recently, more attention has been paid to non-doubling
measures. The works of Nazarov, Treil, Volberg, Tolsa, H\"ytonen and others
have shown that a number of estimates for singular integrals for doubling spaces are still true for non-homogeneous spaces,
i.e. when the space $(X,d, \mu)$ might be non-doubling but it satisfies a polynomial bound on volume growth:
\begin{equation}\label{Non-homo}
\mu(B(x, r)) \le C r^m.
\end{equation}
 See for example, \cite{NTV1, NTV2, Tol6}.

The   BMO space for non-homogeneous spaces was introduced and studied in \cite{NTV, MMNO, Tol, Hyt, HYY, BD}.
However, we note that for the regularized BMO spaces introduced and studied by Tolsa, Hytonen, Bui, Duong and others,
a sufficient condition for an $L^2$ bounded operator $T$  to be bounded from $L^{\infty}(X)$ into the regularized $BMO(X)$
is that the associated kernel $k(x,y)$ of $T$ satisfies the  upper bound
$$ | k(x,y) |  \le {C \over d(x,y)^m }$$
where $m$ is the upper bound on the dimension in (\ref{Non-homo}) and that $k(x,y)$ is H\"older continuous in variable $x$, i.e.
$$| k(x,y) - k(x_1, y) |  \le {C d(x, x_1)^{\alpha} \over d(x,y)^{m+\alpha}} $$
for some $\alpha > 0$ when $d(x,y) \ge c d(x, x_1)$.

Note that the above conditions with $d(x,y)^m$ and $d(x,y)^{m+\alpha}$ on the upper bounds of $|k(x,y)|$ and $| k(x,y) - k(x_1, y) |$ respectively,
are quite strong since in general $d(x,y)^m$ is strictly greater than the volume of the ball with radius $r = d(x,y)$ as in
the standard case of spaces of homogeneous type. Indeed, these required estimates do not hold for large classes of singular
operators on certain non-homogeneous spaces.


 \smallskip

{\it(ii) Singular integrals with rough kernels.}
It is now understood that there
are important situations in which
the classical Calder\'on-Zygmund theory is not
applicable, and these situations are tied to the theory of partial differential
operators   generalizing the Laplacian (e.g., the Schr\"odinger operators $L=-\Delta +V$).
Operators based on the operator $L$, such as the Riesz transform, the pure imaginary
powers $L^{is}$, $s\in {\mathbb R}$ or square functions, may lie beyond the scope of the Calder\'on-Zygmund
theory,   whose kernels do not
satisfy
  the H\"ormander condition.
  Weak type $(1,1)$ estimate was obtained for some of these operators under weaker condition than the
H\"ormander condition, see \cite{CD, DM}.
The study of singular integrals with rough kernels also
lead to the   BMO  space  associated to operators which has been a very successful approach
in recent progress of harmonic analysis.  The main feature of this  BMO space is that
it is  adapted to the operator $L$ through the heat semigroup $e^{-t {L}}$ (or the Poisson semigroup $e^{-t\sqrt{L}}$)
which plays the role of the (generalised) approximation to the identity.
This topic  has attracted a lot of attention in the last   decades,
 and  has been a  very active research topic  in harmonic analysis -- see for example,
    \cite{Aus, ADM, ACDH, AMR, CD, CD2, CL, DL, DM, DY1, DY2, HLMMY, HM} and the references therein.

	 \smallskip
	
	 The present paper can be viewed as a continuation of the above body of work to
	 introduce a space of type
 BMO which is adapted to    operators to
	  study certain singular integrals with rough kernels on some non-doubling spaces.
 Our model of non-homogeneous spaces is  to   consider finite connected sums of the $\mathfrak{R}^n$
and $\R^m$
$$
    M= \mathfrak{R}^{n}\sharp\,  {\R}^{m}.
$$
 for $m> n \ge 3$, where   the manifold $\mathfrak{R}^n$  is given by
$$
 \mathfrak{R}^n=\mathbb{R}^n\times \mathbb{S}^{m-n}.
$$
Here $\mathbb{S}^{m-n}$ is the unit sphere in $\R^{m-n}$. On such a manifold, there is a compact set $K$ with
smooth boundary which connects $\mathfrak{R}^n$ and $\R^m$. It is called the center of $M$. In \cite{GS1,  GS}, Grigor'yan and Saloff-Coste
started a project on the heat kernel bounds for the  heat semigroup $e^{-t\Delta}$ generated by the
Laplace-Beltrami operator $\Delta$ on $M $
  and
obtained  sharp upper bound and lower estimates of heat kernels on $M$, see also  \cite{ GS2}.
The bound of the heat kernel is essentially different from the classical Gaussian upper bound and it depends heavily on the distance of the points $x,y$ from the center part $K$  (see Theorem A in Section 2).
 However, no further information is known on the heat kernels such as estimates on the time derivatives
or the spacial derivatives of the heat kernels.

 Our model case of a differential  operator on underlying space $M$   is the Schr\"odinger operator $L = \Delta + V$ where
 $V$ is a non-negative locally integrable potential.   From the work of  Grigor'yan and Saloff-Coste \cite{GS}, the semigroup $e^{-tL}$
 has the same upper bound as the semigroup $e^{-t {\Delta}}$,
however,  the kernel of $e^{-t {L}}$ can be discontinuous due to the effect of the
 potential $V$. Therefore standard pointwise estimates on the spacial derivatives of the kernel of $e^{-t {L}}$
  are not true and operators like $L^{is}, s\in \mathbb R$ do have rough kernels.


 Throughout the paper,  let $e^{-t \sqrt L}$ be the Poisson semigroup of $L$ on $M$, and
   $A_t f(x)=e^{-t \sqrt L}f(x)$.  Let $\mathcal T$ be the set of functions defined as
$$ \mathcal T:=\{ f\in L^1_{loc}(M): \sup_{x\in M, t>0} |A_t(f)(x)|<\infty \} .
$$
We are now ready to introduce the space  $\bmo$
associated to the operator $L$ on $M$ in the following way:

\begin{definition}\label{def-BMO0}
We say that $f\in \mathcal T$ is in $\bmo$, the {\rm BMO} space associated with $L$, if there exists constant $C>0$ such that

\begin{itemize}
\item[(i)] For every ball $B(x_B,r_B)\subset \R^m$ or $B(x_B,r_B)\subset \Rf^n$,
\begin{eqnarray}\label{BMO-def10}
\frac{1}{\mu(B)}\int_B |f(x)-A_{r_B}f(x)| \,d\mu(x)\leq C;
\end{eqnarray}

\item[(ii)] For all $x\in M$ and $s,t>0$,
\begin{eqnarray}\label{BMO-def20}
|A_s(f)(x)-A_t(f)(x)|\leq C\left(1+\big|\log\frac{s}{t}\big|\right).
\end{eqnarray}
\end{itemize}
When $f\in \bmo$, we define the $\bmo$ norm by the infimum of all the constants $C$ such that \eqref{BMO-def10} and \eqref{BMO-def20} hold.
\end{definition}
Note that $A_t$ acts as an approximation to identity, so the condition~\eqref{BMO-def10} is what we can expect.
The condition~\eqref{BMO-def20} is also quite natural on homogeneous spaces (with doubling measure) if we want to get the John-Nirenberg inequality and other properties for BMO spaces associated with operators; see \cite{DY1}. See \cite{Tol} for a similar definition for BMO spaces with non-doubling measure, where a similar condition as \eqref{BMO-def20} was proposed with $A_t$ replaced by the average over balls.

We also note that our definition of $\bmo$ mainly focuses on the behaviour of Poisson kernel, but not the non-doubling property. So
comparing with the original definition of BMO space on non-doubling spaces introduced in \cite{Tol}, we do not use $\mu(\alpha B)$ (for some $\alpha>1$) in
\eqref{BMO-def10}.
%
%
Then the natural question is that how  we deal with the non-doubling measure.  In fact, we handle the non-doubling
measure by making good use of the upper bound of Poisson kernel (as in Proposition \ref{prop2.1}) and by a new classification of the balls in $M$ as follows.

We classify all balls into two classes.
Denote by $r_{B^m}$ the radius of $B\cap \R^m$. For a fixed $0<\rho<n$, we define a set $\mathcal{B}^\rho_0$ of balls
\begin{eqnarray}\label{b1}
\mathcal{B}^\rho_0:=\{B(x_B,r_B): x_B\in \Rf^n, r_B\geq 2, K\subset B,  r_B^{\frac{n-\rho}{m-\rho}}<r_{B^m}<r_B \}.
\end{eqnarray}
Denote $\mathcal{B}^\rho_1$ the set of all other balls, that is,
\begin{eqnarray}\label{b2}
\mathcal{B}^\rho_1:=\{B(x_B,r_B): B\notin \mathcal{B}^\rho_0\}.
\end{eqnarray}
Throughout the paper  we write $\mathcal{B}_0$ and $\mathcal{B}_1$ in place of  $\mathcal{B}^1_0$ and $\mathcal{B}^1_1$, respectively.
It can be seen that  the classification   of these balls  $\mathcal{B}^\rho_0$ and $\mathcal{B}^\rho_1$ plays an important  role
in our approach.

With this classification of the balls in $M$, we establish the following equivalent characterization of the space $\bmo$
by  providing another definition of $\bmo$ where a new and explicit term ``$\log r_B$'' is  introduced for the balls in  $\mathcal{B}^\rho_0$.
To be more precise, we have
\begin{definition}\label{def-BMO}  Suppose $0<\rho<n$.
We say that $f\in \mathcal T$ is in ${\rm BMO}_{L}^{\rho}(M)$, the {\rm BMO} space associated with $L$, if there exists constant $C>0$ such that:\\
for all $B\in \mathcal{B}^\rho_1$,
\begin{eqnarray}\label{BMO-def1}
\frac{1}{\mu(B)}\int_B |f(x)-A_{r_B}f(x)| \,d\mu(x)\leq C;
\end{eqnarray}
for all $B\in \mathcal{B}^\rho_0$,
\begin{eqnarray}\label{BMO-def2}
\frac{1}{\mu(B)\log r_B}\int_B |f(x)-A_{r_B}f(x)| \,d\mu(x)\leq C.
\end{eqnarray}
When $f\in {\rm BMO}_L^\rho(M)$, we define the ${\rm BMO}_L^\rho(M)$ norm by
$$
\|f\|_{{\rm BMO}_L^\rho(M)}:=\max\left\{\sup_{B\in \mathcal{B}^\rho_1} \frac{1}{\mu(B)}\int_B |f(x)-A_{r_B}f(x)| \, d\mu(x),
\sup_{B\in \mathcal{B}^\rho_0} \frac{1}{\log r_B\mu(B)}\int_B |f(x)-A_{r_B}f(x)| \, d\mu(x)\right\}
$$
with $r_B$  the radius of the ball $B\subset M$.
\end{definition}

Then we prove that the versions of BMO spaces as in Definitions \ref{def-BMO0} and \ref{def-BMO} are equivalent.
\begin{theorem}\label{main0}
For every $0<\rho<n$, the spaces
 $  {\rm BMO}_L^\rho(M) $ and $ \bmo $  coincide
and they have  equivalent norms.
\end{theorem}

%

With the  above characterization of the space $\bmo$, we can prove
  the John--Nirenberg type inequality for $\bmo$.
\begin{theorem}\label{main1} Let $\mathcal{B}_0$ and $\mathcal{B}_1$ be the set of balls defined in
\eqref{b1} and \eqref{b2}, respectively.
If $f\in {{\rm BMO}_L(M)}$, then there exist positive constants $c_1$ and $c_2$ such that

\begin{itemize}
\item[(i)]
For every ball $B\in \mathcal{B}_1$ and every $\alpha>0$, we have
\begin{eqnarray}\label{eq44.1}
\mu\big(\{x\in B:|f(x)-A_{r_B}f(x)|>\alpha\}\big)\leq c_1\mu(B)\exp\left(-\frac{c_2\alpha}{\|f\|_{\bmo}}\right);
\end{eqnarray}

\item[(ii)]For every ball $B\in \mathcal{B}_0$ and every $\alpha>0$, we have
\begin{eqnarray}\label{eq44.11}
\mu\big(\{x\in B:|f(x)-A_{r_B}f(x)|>\alpha\}\big)\leq c_1\mu(B)\exp\left(-\frac{c_2\alpha}{\big(\log r_B\big)\|f\|_{\bmo}}\right).
\end{eqnarray}
\end{itemize}

As a consequence,  we    have  that for  all $1\leq p<\infty$,
\begin{align*}
&\|f\|_{\bmo}\\
&\sim\max\left\{\sup_{B\in \mathcal{B}_1} \left(\frac{1}{\mu(B)}\int_B |f(x)-A_{r_B}f(x)|^p \, d\mu(x)\right)^{1/p},
\sup_{B\in \mathcal{B}_0} \frac{1}{\log r_B}\left(\frac{1}{\mu(B)}\int_B |f(x)-A_{r_B}f(x)|^p \, d\mu(x)\right)^{1/p}\right\}.
\end{align*}
\end{theorem}

\medskip

To establish the interpolation between $L^p(M)$ and $\bmo$, we introduce   the following new
version of sharp maximal function
\begin{eqnarray}\label{sharp}\\
{\mathcal M}^\sharp f(x):=\max\Big\{\sup_{B\ni x,B\in \mathcal B_1} \frac{1}{\mu(B)}\int_B|f(y)-A_{r_B}f(y)| \, dy,
\sup_{B\ni x,B\in \mathcal B_0} \frac{1}{(\log r_B)\mu(B)}\int_B|f(y)-A_{r_B}f(y)| \, dy\Big\}. \nonumber
\end{eqnarray}
As in the case of the classical BMO,  we will establish a version of good-$\lambda$ inequality. Just recall that in
 the classical case \cite{Gra} and \cite{DY1}, to prove the interpolation, they established the good-$\lambda$ inequality
 with respect to sharp maximal function and Hardy--Littlewood maximal function. However, in our setting, the Hardy--Littlewood
 maximal function is not the suitable candidate.  Hence, we now introduce a new nontangential maximal function in terms
 of the Poisson semigroup.  To do it, first we define two sets of dyadic cubes:
$$
\mathcal{I}_1=\{\mbox{dyadic cubes $Q\subset \R^m$ or $Q\subset \Rf^n$ such that $r_Q\geq 2$ and dist$(Q,K)\leq r_Q$}\}
$$
and
$$
\mathcal{I}_2=\{\mbox{dyadic cubes $Q\subset \R^m$ or $Q\subset \Rf^n$ such that $r_Q<2$ or dist$(Q,K)> r_Q$}\}.
$$
Then we introduce our ``cube'' system on $M$ in the following way: the set $\mathcal{D}_1$ includes the following cubes:
\begin{itemize}
\item[(1)]
 all dyadic cubes in $\Rf^n$ including $K$;

\item[(2)]  all dyadic cubes in $\R^m$ such that none of the corners of the cube is around $K$, that is, all dyadic dubes belong to $\mathcal{I}_2$;

\item[(3)] a new``cube" $Q$ defined by $Q=Q_m\cup Q_n$, where $Q_m$ is a dyadic cube in $ \R^m$ and $Q_m\in \mathcal{I}_1$,  and $Q_n$ is a dyadic cube $ \Rf^n$ and $Q_n\in \mathcal{I}_1$ with  $\mu(Q_n)\geq\mu(Q_m)$. Define the ``side length'' of $Q$ as the side length $Q_n$, that is, $r_Q=r_{Q_n}$;
\end{itemize}
the set $\mathcal{D}_2$ includes the following cubes:
\begin{itemize}
\item[(4)]
all the dyadic cubes in $\R^m$ such that one of the corners of the cube is around $K$, that is, all dyadic dubes belong to $\mathcal{I}_1$.
\end{itemize}
We define the non-tangential maximal function on $M$ by
\begin{equation}\label{maximal}
\N_L f(x):=\max\left\{\sup_{Q\ni x, Q\in \mathcal{D}_1}\sup_{y\in Q} |\exp (-r_Q\sqrt L)f(y)|,\sup_{Q\ni x, Q\in \mathcal{D}_2}\sup_{y\in Q,|y|\geq r_Q^{\frac{m-n}{m-2}}/2} |\exp (-r_Q\sqrt L)f(y)|\right\}.
\end{equation}
{We can see for $Q\in \mathcal{I}_1$, in the definition of $\N_L$, we skip a corner of $Q$. In order to have the information of this corner, we need cubes in (3) of $\mathcal{D}_1$.}
With this new non-tangential maximal function, we establish the following good-$\lambda$ inequality:
There exits small enough $\gamma>0$ and large enough $K>0$ such that for all $\lambda>0$  and all locally integrable functions $f$, we have
\begin{eqnarray}\label{ess}
\mu\big(\{  x\in M: |f(x)|>K\lambda,\ \mathcal M^\sharp f(x) \leq \gamma\lambda \}\big)
\leq C\gamma\mu\big(\{x\in M: \N_L  f(x)>\lambda\}\big).
\end{eqnarray}

Based on the estimate \eqref{ess}, we can show  the following interpolation result for a holomorphic family of operators.

\begin{theorem}\label{main2}
Assume that $T_z$ is a holomorphic family of linear operators for $z = s + i t$ with $0 \le s \le 1$ and $-\infty < t < \infty$.
Also assume that $T_{it}$ is uniformly bounded on $L^q(M)$ for some $1 < q < \infty$ and $T_{1+it}$ is uniformly bounded
from $L^{\infty}(M)$ to $\bmo$. Then $T_{\theta}$ is bounded on $L^p(M)$ whenever $0 \le \theta = 1-q/p < 1$.
 \end{theorem}

Note that in the special case that the family $T_z = T$ for all $z$, then we obtain the following:
Assume that $T$  is a sub-linear operator which is
bounded on $L^q(X)$, $1 \leq q < \infty$ and bounded from
$L^\infty(M)$ to $\bmo$.
Then $T$ is bounded on $L^p(M)$ for all $q<p<\infty$.
See Theorem \ref{Theorem-Interpolation} for details.

As an application, we obtain endpoint boundedness of the Laplace transform for the operator $\sqrt{L}$. For more details about the $L^p$ boundedness of the Laplace transform, we refer to Corollary 3 in \cite[p. 121]{Stein}.

\begin{theorem}\label{main3}
Let $\Delta$ be the Laplace-Beltrami
operator on $M:=\Rf^n\sharp\mathbb R^m$ with $m>n\geq 3$ and $L=\Delta+V$ be the Schr\"odinger
operator with non-negative potential $V$.
Let $\tilde m(\sqrt L)$ be the holomorphic functional calculus of Laplace transform type of $\sqrt{L}$ defined by
$$\tilde m(\sqrt L)f  = \int_0^\infty \left[\sqrt L \exp(-t\sqrt L) f \right] m(t) \, dt\, $$
in which $m(t)$ is a bounded function on $[0, \infty)$, i.e.,
$ |m(t)|\leq C_0, $ where $C_0$ is a constant.  Then $\tilde m(\sqrt{L})$ is bounded from $L^\infty(M)$ to $\bmo$. Hence by interpolation and duality, the operator $\tilde m(\sqrt{L})$ is bounded on $L^p(M)$ for $1 < p < \infty$.
\end{theorem}

This result implies directly that $L^{is}$, $s\in \mathbb{R}$, which is one of the natural singular integrals associated with $L$,  is bounded from
$L^\infty(M)$ to $\bmo$ and on $L^p(M)$ for $1 < p < \infty$.
We note that in the work \cite{BDLW}, it was shown that the holomorphic functional calculus $\tilde  m(\sqrt{L})$ is of weak type $(1,1)$.

Another singular integral operator associated to $L$ which has attracted lots of attention is the Riesz transform $\nabla L^{-{1/2}}$.
We point out that when $L=\Delta$, i.e.,  the Laplace-Beltrami
operator on $M=\Rf^n\sharp\mathbb R^m$, Carron \cite{Car} first proved that  Riesz transform $\nabla \Delta^{-{1/2}}$
is bounded on $L^p(M)$ for $p\in ({n\over n-1},n)$ when $n>3$. Recently, Hassell and Sikora \cite{HS}
proved the full range of boundedness for $\nabla \Delta^{-{1/2}}$ by showing that $\nabla \Delta^{-{1/2}}$ is of weak type $(1,1)$,
 bounded on $L^p(M)$ for $1<p<n$ with $n>2$. and unbounded for $p\geq n$. Hence the Riesz transform $\nabla \Delta^{-{1/2}}$ is not
the suitable operator for the study of the structure of our BMO spaces. Indeed, we can see easily that the Riesz transform
is {\it not} bounded from $L^\infty(M)$ to $\bmo$; otherwise by interpolation we will get $L^p$ boundedness of the Riesz transform
for $1 < p < \infty$ which is a contradiction to the result in \cite{HS}.

%

\medskip

\subsection{Structure and main techniques}

To obtain our results above, we mainly use the idea and framework from \cite{DY1}, where ${\rm BMO}_L$ was first introduced and established.
However, the main difficulties in this paper are still very substantial and we list them in the following:

\smallskip
(1) The upper bound of Poisson kernel $p_t(x,y)$ of the semigroup $e^{-t\sqrt{L}}$ is essentially different from the classical
upper bound, i.e.
$${1\over \mu(B(x,t))+ \mu(B(x, d(x,y)))}\left({ t \over {t} +d(x,y)}\right)^\epsilon,
$$
 and it depends heavily on the distance of the points $x,y$ from the center part $K$, and the terms in the denominator of the Poisson kernel
do not usually match the volume of the ball $B(x, {t})$ (or $B(y,{t})$), see Proposition~\ref{prop2.1} in Section 2.

(2) The underlying space $M$ has a non-doubling measure, which satisfies only the polynomial growth. If we just think of this point
only, then it is not new and there are already a few nice techniques and decompositions due to \cite{Tol,NTV,Hyt} and so on. However,
in order to get our results, we need to handle the non-doubling measure by adapting to the Poisson kernel upper bounds as mentioned in (1)
above. Hence, this leads to
a new technique and decomposition of $M$, which is different from  \cite{Tol,NTV,Hyt}.

\smallskip

To be more specific about the connections between (1) and (2) as we addressed above, we need to have a delicate argument
which takes into account the geometry of the manifold and the behaviour of the heat kernel.
Consider a ball $B\subset M$ centered at $x_B$ with radius $r_B$ and the Poisson kernel $p_s(x,y)$ with $s\approx r_B$ and $x$
(or $y$) close to $x_B$, for example, $d(x,y)\leq 2r_B$, and
we say that $B$ matches the Poisson kernel $p_s(x,y)$ if
the denominator of the upper bound of $p_s(x,y)$ is equivalent to the volume of $B$. We also say that $B$ is doubling if $\mu(2B)\leq C\mu(B)$ with $C$ an absolute constant. Then we have five cases
in which the first two cases are close to the classical setting of \cite{DY1}.

\smallskip

\noindent
{\it Case (i):}  $r_B\leq1$. In this case, wherever $x_B$ is, the ball $B$ is doubling and it matches $p_s(x,y)$.

\smallskip
\noindent
{\it Case (ii):}  $|x_B|\geq2 r_B$. In this case, the ball $B$ is away from the center part $K$. Hence $B$ is doubling

\hspace{1cm} and it matches $p_s(x,y)$.

 \smallskip

\noindent
 However, then there are the other three cases left when $r_B>1$, which requires new techniques.

\smallskip

\noindent
{\it Case (iii):} $B\subset\R^m$ and $|x_B|\leq 2r_B$. In this case,  the ball $B$ is doubling, however, it does not match

\hspace{1cm} $p_s(x,y)$.

\smallskip

\noindent
{\it Case (iv):} $B\subset\Rf^n$, $|x_B|\leq 2r_B$ and $B\not\in \mathcal B_0$. In this case,  the ball $B$  matches $p_s(x,y)$, however, it is

\hspace{1cm} not doubling.

\smallskip

\noindent
{\it Case (v):} $B\in \mathcal B_0$. In this case,  the ball $B$ does not  match $p_s(x,y)$, and it is not doubling either.

\hspace{1cm} This is the case where the term $\log r_B$ arises in the definition of $\bmo$ so that we

\hspace{1cm} can obtain the boundedness of certain
singular integrals from $L^\infty(M)$ to $\bmo$.

\smallskip

We would like to mention that the main technique in the proof of Theorem \ref{main1},  the John--Nirenberg type inequality for $\bmo$,
 is to provide suitable version of dyadic decompositions and
split the dyadic cubes into two groups, and then repeat the process infinitely many times. To be more specific, given a ball $B\subset M$,
 if is it not in {\it Case (v)} above,
then we consider two parts
$$ B\cap \Rf^n\ \ \ {\rm and}\ \ \ B\cap \R^m.  $$
Then for each part we consider the dyadic decomposition within the part itself, such that in each level of the dyadic cubes, only a finite
 number of dyadic cubes are close to the center $K$ and others are away. For the dyadic cubes away from $K$, we can handle the proof using the
method as in \cite{DY1}. For the dyadic cubes close to $K$, they could be non-doubling or they do not match the Poisson kernel upper bound, and
 hence there is not enough condition to handle that. However, the number of these cubes is up to a finite upper bound for all levels of dyadic
  cubes, so we can just handle them directly. Repeat this process infinitely many times, we obtain the usual form of John--Nirenberg inequality,
   i.e., \eqref{eq44.1} in Theorem \ref{main1}. Given a ball $B\subset M$, if it is in {\it Case (v)} above, then from a similar process we obtain the
    new version of John--Nirenberg inequality with an extra term $\log r_B$, i.e. , \eqref{eq44.11} in Theorem \ref{main1}.
Also in the proof of of Theorem \ref{main3},
	the main method   is to consider the ball $B$ in {\it Case (i)}---{\it Case (v)} as listed above. The first four cases
can be handled by decomposing the underlying space related to the Poisson kernel upper bounds. In the last case, we have
no information from the kernel upper bound and hence the term $\log r_B$ plays an important role.


The layout of the paper is as follows. In Section~2,   we will prove  some preliminaries, which we need later, mostly on
the  kernel estimates of the heat   and Poisson semigroups of $\sqrt{L}$,
and establish $L^p$ bounds for non-tangential maximal function in terms of the Poisson semigroup.
In  Section 3,   we will prove the equivalence of two definitions of the BMO spaces.
With this, we can show  our main result  Theorem \ref{main1}, the John--Nirenberg inequality for $\bmo$ in Section 4.
In Section 5, we will prove Theorem \ref{main2}, the interpolation between $L^p(M)$ and $\bmo$.  In Section  6, we will  show   Theorem \ref{main3}, the boundedness of singular integral $L^{is}$ from $L^\infty(M)$
to our adapted space $\bmo$.

\bigskip

\section{Preliminaries on manifold with ends}
\setcounter{equation}{0}

Concerning the structure of manifolds with ends $M$, we refer readers to \cite{GS}. The manifold $M$ is basically a copy of $\mathbb R^m$ connected  to  $\mathfrak{R}^n$  smoothly by a compact set $K$ of length $1$ where $\mathfrak{R}^n=\mathbb R^n \times S^{m-n}$ and
$S^{m-n}$ denotes the unit sphere in $\mathbb R^{m-n}$.

For any $x\in M$, define
$
   |x|:=\sup_{z\in K}d(x,z),
$
where $d=d(x,y)$ is the geodesic distance in $M$.
One can see that $|x|$ is separated from zero on $M$ and
\begin{align*}
|x|\approx 1+d(x,K).
\end{align*}

For $x\in M$, let
$
   B(x,r):=\{y\in M: d(x,y)<r\}
$
be the geodesic ball with center $x\in M$ and radius $r>0$ and let
$
   V(x,r)=\mu(B(x,r))
$
where $\mu$ is the Riemannian measure on $M$. We also point out that the function $V(x,r)$ satisfies

\begin{itemize}
\item[(a)]  $V(x,r)\thickapprox r^m$ for all $x\in M$, when $r\leq 1$;

\item[(b)]  $V(x,r)\thickapprox r^n$ for $B(x,r)\subset \Rf^n$, when $r> 1$; and

\item[(c)] $V(x,r)\thickapprox r^m$ for $x\in \Rf^n\backslash K$, $r>2|x|$, or $x\in \R^m$, $r>1$.

\end{itemize}
It is not difficult to check that $M$ does not satisfy the doubling condition.  Indeed, consider a
sequence of balls $B(x_k,r_k)\subset \Rf^n$ such that $r_k = |x_k| > 1$ and $r_k \rightarrow \infty$ as $k \rightarrow \infty$.
Then $V(x_k,r_k)\thickapprox (r_k)^n$. However, $V(x_k,2r_k)\thickapprox (r_k)^m$ and the doubling condition fails.

Let $\Delta$  be the Laplace-Beltrami operator on $M$ and $e^{-t\Delta}$ the heat semi-group generated by $\Delta$.
We denote by $h_t(x,y)$ the heat kernel associated to $e^{-t\Delta}$.
In \cite{GS},  Grigor'yan and Saloff-Coste  obtained the following result.
\begin{TheoremA}[\cite{GS}] Let $M=\Rf^n\sharp\R^m$ with $3\leq n < m$. Then the heat kernel $h_t(x,y)$ satisfies the following estimates.
\begin{itemize}
\item[1.]
 For $t\leq 1$ and all $x,y\in M$,
\begin{eqnarray*}
h_t(x,y)\approx {C \over V(x,\sqrt{t})}\exp\Big( -c {d(x,y)^2\over t} \Big).
\end{eqnarray*}

\item[2.]   For $x,y\in K$ and all $t>1$,
\begin{eqnarray*}
h_t(x,y)\approx {C \over t^{n/2}}\exp\Big( -c {d(x,y)^2\over t} \Big).
\end{eqnarray*}

\item[3.]  For $x\in \R^m\backslash K$, $y\in K$ and all $t>1$,
\begin{eqnarray*}
h_t(x,y)\approx C \Big({1 \over t^{n/2} |x|^{m-2}  }+ {1\over t^{m/2}}\Big)\exp\Big( -c {d(x,y)^2\over t} \Big).
\end{eqnarray*}

\item[4.]  For $x\in \Rf^n\backslash K$, $y\in K$ and all $t>1$,
\begin{eqnarray*}
h_t(x,y)\approx C \Big({1 \over t^{n/2} |x|^{n-2}  }+ {1\over t^{n/2}}\Big)\exp\Big( -c {d(x,y)^2\over t} \Big).
\end{eqnarray*}

\item[5.]  For $x\in \R^m\backslash K$, $y\in \Rf^n\backslash K$ and all $t>1$,
\begin{eqnarray*}
h_t(x,y)\approx C \Big( {1\over t^{n/2}|x|^{m-2}} + {1\over t^{m/2}|y|^{n-2}} \Big)\exp\Big( -c {d(x,y)^2\over t} \Big)
\end{eqnarray*}

\item[6.]  For $x,y\in \R^m\backslash K$ and all $t>1$,
\begin{eqnarray*}
h_t(x,y)\approx  {C\over t^{n/2}|x|^{m-2}|y|^{m-2}}\exp\Big( -c {|x|^2+|y|^2\over t} \Big) +{C\over t^{m/2}}\exp\Big( -c {d(x,y)^2\over t} \Big)
\end{eqnarray*}

\item[7.]  For $x,y\in \Rf^n\backslash K$ and all $t>1$,
\begin{eqnarray*}
h_t(x,y)\approx  {C\over t^{n/2}|x|^{n-2}|y|^{n-2}}\exp\Big( -c {|x|^2+|y|^2\over t} \Big) +{C\over t^{n/2}}\exp\Big( -c {d(x,y)^2\over t} \Big).
\end{eqnarray*}
\end{itemize}
\end{TheoremA}

Let  $e^{-t\sqrt{L}}$ the Poisson semi-group generated by $L=\Delta+V$, where is an arbitrary non-negative potential.
The following proposition was proved in    \cite[Theorem 2.2]{BDLW}.
\begin{proposition}\label{prop2.1}
Let $k\in \mathbb{N}$, we denote by $p_{t,k}(x,y)$ the kernel of $ (t\sqrt{L})^k e^{-t\sqrt{L}}$.  For $k=0$, we write $p_t(x,y)$ instead of $p_{t,0}(x,y)$.
For $k\in \mathbb{N}$, set $k\vee 1=\max\{k,1\}$. Then the kernel  $p_{t,k}(x,y)$ satisfies the following estimates:

\begin{itemize}
\item[1.]   For $x,y\in K $,
$$ |p_{t,k}(x,y)|\leq \frac{C}{t^m}\Big(\frac{t}{t+d(x,y)}\Big)^{m+k\vee 1}+ \frac{C}{t^n}\Big(\frac{t}{t+d(x,y)}\Big)^{n+k\vee 1}; $$

\item[2.]  For $x\in \mathbb{R}^m\backslash K $, $y\in K$,
$$ |p_{t,k}(x,y)|\leq \frac{C}{t^m}\Big(\frac{t}{t+d(x,y)}\Big)^{m+k\vee 1} +\frac{C}{t^n|x|^{m-2}}\Big(\frac{t}{t+d(x,y)}\Big)^{n+k\vee 1}; $$

\item[3.]  For  $x\in \Rf^n\backslash K $, $y\in K$,
$$ |p_{t,k}(x,y)|\leq \frac{C}{t^m}\Big(\frac{t}{t+d(x,y)}\Big)^{m+k\vee 1}+ \frac{C}{t^n}\Big(\frac{t}{t+d(x,y)}\Big)^{n+k\vee 1}; $$

\item[4.]  For  $x\in \mathbb{R}^m\backslash K $, $y\in \Rf^n\backslash K $,
\begin{align*}
|p_{t,k}(x,y)|&\leq \frac{C}{t^m}\Big(\frac{t}{t+d(x,y)}\Big)^{m+k\vee 1}+ \frac{C}{t^n|x|^{m-2}}\Big(\frac{t}{t+d(x,y)}\Big)^{n+k\vee1}\\
&\quad+\frac{C}{t^m|y|^{n-2}}\Big(\frac{t}{t+d(x,y)}\Big)^{m+k\vee 1};
\end{align*}

\item[5.]  For  $x,y\in \mathbb{R}^m\backslash K $,
$$ |p_{t,k}(x,y)|\leq \frac{C}{t^m}\Big(\frac{t}{t+d(x,y)}\Big)^{m+k\vee 1} +\frac{C}{t^n|x|^{m-2}|y|^{m-2}}\Big(\frac{t}{t+|x|+|y|}\Big)^{n+k\vee1};$$

\item[6.]  For $x,y\in \Rf^n\backslash K $,
$$ |p_{t,k}(x,y)|\leq \frac{C}{t^m}\Big(\frac{t}{t+d(x,y)}\Big)^{m+k\vee 1}+ \frac{C}{t^n}\Big(\frac{t}{t+d(x,y)}\Big)^{n+k\vee 1}. $$
\end{itemize}
\end{proposition}


Let us recall next the standard definition of uncentered Hardy--Littlewood maximal function.
\begin{definition}[\cite{DLS}]\label{HLMaximal}
For any  $p \in [1,\infty]$ and any function $f\in L^p$
let
$$
\mathcal M f(x) :=\sup_{y\in M,\ r>0} \left\{ \frac{1}{V(y,r)}\int_{B(y,r)} |f(z)| \, d\mu(z)\colon x\in B(y,r) \right\}.
$$
\end{definition}
\begin{lemma}[\cite{DLS}]\label{max}
The maximal function $\mathcal M(f)$ is of weak type $(1,1)$ and bounded on all $L^p$ spaces for $1< p \le \infty$.
\end{lemma}

Next we show $L^p$-bounds for non-tangential maximal function in terms of the Poisson semigroup,
which will be used in the sequel.
Precisely, we have

\begin{theorem}\label{th nontangential 2}
The non-tangential maximal function $\N_L$ as defined in \eqref{maximal}
 is of weak type $(1,1)$ and bounded on $L^p(M)$ for all $1<p\leq\infty$.
\end{theorem}
\begin{proof}
First, it is easy to very that for all $0<t<\infty$ and $x\in M$
$$
\int_M p_t(x,y) \,d\mu(y)\leq C<\infty
$$
where $C$ is independent on $t$ and $x$. This implies that $\N_L$ is bounded on $L^\infty(M)$. So what remains to prove is $\N_L$ is of weak type $(1,1)$.

Next we define the restriction Hardy-Littlewood  maximal function   $\mathcal{M}_n$ on $\Rf^n$ as:
\begin{eqnarray}
\mathcal{M}_n(f)(x):=\sup_{B\ni x}\frac{1}{\mu(B\cap \Rf^n)}\int_{B\cap \Rf^n} |f(y)|\, d\mu(y),\quad f {\rm\ with\ } \supp f\subset \Rf^n.\nonumber
\end{eqnarray}
It is essentially the classical Hardy-Littlewood  maximal function on $\Rf^n$ and  it is of weak type $(1,1)$.

We now begin to prove $\N_L$ is of weak type $(1,1)$. Fixed $x\in M$ and $Q\in \mathcal{D}_1$.

{\bf Case I}: $Q\in \mathcal{I}_2$ and $Q\subset \R^m$. Let $f_m:=f\chi_{\R^m}$ and $f_n:=f\chi_{\Rf^n}$. For each $y\in Q$,
\begin{align}
|\exp (-r_Q\sqrt L)f(y)|&\leq \int_{\R^m} p_{r_Q}(y,z)|f_m(z)| \, d\mu(z)+\int_{\Rf^n} p_{r_Q}(y,z)|f_n(z)| \, d\mu(z)\nonumber.
\end{align}
For every $y\in\R^m$, no matter where is $z$, we have
$$
p_{r_Q}(y,z)\leq \frac{C}{r_Q^m}\bigg(1+\frac{d(y,z)}{r_Q}\bigg)^{-m-1}+\frac{C}{r_Q^n|y|^{m-2}}\bigg(1+\frac{|y|+|z|}{r_Q}\bigg)^{-n-1}.
$$
Then
\begin{align*}
\int_{\R^m} \frac{C}{r_Q^m}\bigg(1+\frac{d(y,z)}{r_Q}\bigg)^{-m-1}|f_m(z)|\, d\mu(z)\leq C\mathcal{M}(f)(x)
\end{align*}
and note that $Q\in \mathcal{I}_2$ and $x,y\in Q$ implies that $|x|\sim |y|>r_Q$
\begin{align*}
\int_{\R^m} \frac{C}{r_Q^n|y|^{m-2}}\bigg(1+\frac{|y|+|z|}{r_Q}\bigg)^{-n-1}|f_m(z)|\, d\mu(z)&\leq
\int_{\R^m} \frac{C}{r_Q^n|y|^{m-2}}\bigg(\frac{|y|}{r_Q}\bigg)^{-n}|f_m(z)| \, d\mu(z)\\
&\leq C\frac{\|f\|_{L^1(M)}}{|x|^m}.
\end{align*}

{\bf Case II}: $Q\in \mathcal{I}_2$ and $Q\subset \Rf^n$. Let $f_m:=f\chi_{\R^m}$ and $f_n:=f\chi_{\Rf^n}$. For each $y\in Q$,
\begin{align}
|\exp (-r_Q\sqrt L)f(y)|&\leq \int_{\R^m} p_{r_Q}(y,z)|f_m(z)|\, d\mu(z)+\int_{\Rf^n} p_{r_Q}(y,z)|f_n(z)| \, d\mu(z)\nonumber.
\end{align}
For every $y,z\in\Rf^n$, we have
$$
p_{r_Q}(y,z)\leq \frac{C}{r_Q^m}\bigg(1+\frac{d(y,z)}{r_Q}\bigg)^{-m-1}+\frac{C}{r_Q^n}\bigg(1+\frac{d(x,y)}{r_Q}\bigg)^{-n-1}
\leq \frac{C}{\mu(Q)}\bigg(1+\frac{d(x,y)}{r_Q}\bigg)^{-n-1}.
$$
Then
\begin{align*}
\int_{\Rf^n} p_{r_Q}(y,z)|f_n(z)| \, d\mu(z)\leq \int_{\Rf^n} \frac{C}{\mu(Q)}\bigg(1+\frac{d(x,y)}{r_Q}\bigg)^{-n-1}|f_n(z)| \, d\mu(z)\leq C\mathcal{M}_n(f_n)(x).
\end{align*}
For every $y\in\Rf^n$ and $z\in\R^m$, we have
$$
p_{r_Q}(y,z)\leq \frac{C}{r_Q^m}\bigg(1+\frac{d(y,z)}{r_Q}\bigg)^{-m-1}+\frac{C}{r_Q^n|z|^{m-2}}\bigg(1+\frac{|y|+|z|}{r_Q}\bigg)^{-n-1}.
$$
Then
\begin{align*}
\int_{\R^m} \frac{C}{r_Q^m}\bigg(1+\frac{d(y,z)}{r_Q}\bigg)^{-m-1}|f_m(z)| \,d\mu(z)\leq C\mathcal{M}(f)(x)
\end{align*}
and note that $Q\in \mathcal{I}_2$ and $x,y\in Q$ implies that $|x|\sim |y|>r_Q$
\begin{align*}
\int_{\R^m} \frac{C}{r_Q^n|z|^{m-2}}\bigg(1+\frac{|y|+|z|}{r_Q}\bigg)^{-n-1}|f_m(z)| \, d\mu(z)&\leq
\int_{\R^m} \frac{C}{r_Q^n}\bigg(\frac{|y|}{r_Q}\bigg)^{-n}|f_m(z)| \, d\mu(z)\\
&\leq C\frac{\|f\|_{L^1(M)}}{|x|^n}.
\end{align*}

{\bf Case III}: $Q\in \mathcal{I}_1$ and $Q\subset \Rf^n$. For each $y\in Q$,
\begin{align}
|\exp (-r_Q\sqrt L)f(y)|&\leq \int_{\R^m} p_{r_Q}(y,z)|f_m(z)| \,d\mu(z)+\int_{\Rf^n} p_{r_Q}(y,z)|f_n(z)| \, d\mu(z)\nonumber.
\end{align}
For every $y,z\in\Rf^n$, we have
$$
p_{r_Q}(y,z)\leq \frac{C}{r_Q^m}\bigg(1+\frac{d(y,z)}{r_Q}\bigg)^{-m-1}+\frac{C}{r_Q^n}\bigg(1+\frac{d(x,y)}{r_Q}\bigg)^{-n-1}
\leq \frac{C}{\mu(Q)}\bigg(1+\frac{d(x,y)}{r_Q}\bigg)^{-n-1}.
$$
Then
\begin{align*}
\int_{\Rf^n} p_{r_Q}(y,z)|f_n(z)| \, d\mu(z)\leq \int_{\Rf^n} \frac{C}{\mu(Q)}\bigg(1+\frac{d(x,y)}{r_Q}\bigg)^{-n-1}|f_n(z)| \, d\mu(z)\leq C\mathcal{M}_n(f_n)(x).
\end{align*}
For every $y\in\Rf^n$ and $z\in\R^m$, we have
$$
p_{r_Q}(y,z)\leq \frac{C}{r_Q^m}\bigg(1+\frac{d(y,z)}{r_Q}\bigg)^{-m-1}+\frac{C}{r_Q^n|z|^{m-2}}\bigg(1+\frac{|y|+|z|}{r_Q}\bigg)^{-n-1}.
$$
Then
\begin{align*}
\int_{\R^m} \frac{C}{r_Q^m}\bigg(1+\frac{d(y,z)}{r_Q}\bigg)^{-m-1}|f_m(z)| \, d\mu(z)\leq C\mathcal{M}(f)(x)
\end{align*}
It remains to control
\begin{align*}
&\int_{\R^m} \frac{C}{r_Q^n|z|^{m-2}}\bigg(1+\frac{|y|+|z|}{r_Q}\bigg)^{-n-1}|f_m(z)|\, d\mu(z)\\
&\leq
\int_{\{|z|>r_Q\}\cap \R^m} \frac{C}{r_Q^n|z|^{m-2}}\bigg(1+\frac{|z|}{r_Q}\bigg)^{-n-1}|f_m(z)| \, d\mu(z)+\int_{\{|z|\leq r_Q\}\cap\R^m} \frac{C}{r_Q^n|z|^{m-2}}\bigg(1+\frac{|z|}{r_Q}\bigg)^{-n-1}|f_m(z)| \, d\mu(z)\\
&=:I+II.
\end{align*}
For the term $I$,
\begin{align*}
&\int_{\{|z|>r_Q\}\cap \R^m} \frac{C}{r_Q^n|z|^{m-2}}\bigg(1+\frac{|y|+|z|}{r_Q}\bigg)^{-n-1}|f_m(z)| \, d\mu(z)\\
&\leq \int_{\{|z|>r_Q\}\cap \R^m} \frac{C}{r_Q^n|z|^{m-2}}\bigg(\frac{|z|}{r_Q}\bigg)^{-n}|f_m(z)| \, d\mu(z)\\
&\leq \int_{\{|z|>r_Q\}\cap \R^m} \frac{C}{|z|^{n+m-2}}|f_m(z)| \, d\mu(z)\\
&\leq \int_{\{|z|>r_Q\}\cap \R^m} \frac{C}{|x|^{n}}|f_m(z)|\, d\mu(z)\\
&\leq C\frac{\|f\|_{L^1(M)}}{|x|^n}.
\end{align*}
For  the term $II$, choose a nature number $k_0$ such that $(2^{k_0}r_Q)^{m}\sim r_Q^{n}$.
\begin{align*}
II&= \int_{\{|z|\leq r_Q\}\cap \R^m} \frac{C}{r_Q^n|z|^{m-2}}\left(1+\frac{|z|}{r_Q}\right)^{-n-1}|f(z)| \, d\mu(z)\\
&\leq \sum_{k_0\leq k<0}\int_{\{2^kr_Q\leq |z|<2^{k+1}r_Q\}\cap \R^m} \frac{C}{r_Q^n|z|^{m-2}}|f(z)| \, d\mu(z)\\
&\quad+\int_{\{|z|\leq 2^{k_0}r_Q\}\cap\R^m} \frac{C}{r_Q^n|z|^{m-2}}|f(z)|\, d\mu(z)\\
&=:II_{1}+II_{2}.
\end{align*}
To continue, we first point out that for $k\geq k_0$ and $\{|z|<2^kr_Q\}$, we can choose a ball $B(y_k,2r_Q)\subset M$ with radius $2r_Q$ such that
 $y_k\in \Rf^n$, $(\{|z|<2^kr_Q\}\cap \R^m )\cup Q\subset B(y_k,2r_Q)$ and $\mu(B(y_k,2r_Q))\sim (2^kr_Q)^m$.
 Then note that $r_Q\geq 1$ and $n\geq 3$,
\begin{align*}
II_{1}&\leq \sum_{k_0\leq k<0}\int_{\{2^kr_Q\leq |z|<2^{k+1}r_Q\}\cap \R^m} \frac{C}{r_Q^n|z|^{m-2}}|f(z)| \, d\mu(z)\\
&\leq  \sum_{k_0\leq k<0} \frac{C(2^kr_Q)^2}{r_Q^n(2^kr_Q)^m}\int_{B(y_k,2r_Q)} |f(z)|\, d\mu(z)\\
&\leq  \sum_{k_0\leq k<0} 2^{2k}\frac{C}{r_Q^{n-2}}\frac{1}{\mu(B(y_k,2r_Q))}\int_{B(y_k,2r_Q)} |f(z)| \, d\mu(z)\\
&\leq  \sum_{k_0\leq k<0} 2^{2k} \mathcal{M}f(x)\\
&\leq  C \mathcal{M}f(x).
\end{align*}
For $II_{2}$, we can choose a ball $B(y_0,2r_Q)\subset M$ with radius $2r_Q$ such that
 $y_0\in \Rf^n$, $(\{|z|<2^{k_0}r_Q\}\cap \R^m) \cup Q\subset B(y_0,2r_Q)$ and $\mu(B(y_0,2r_Q))\sim r_Q^n$.
\begin{align*}
II_{2}&\leq \int_{\{|z|\leq 2^{k_0}r_Q\}\cap\R^m} \frac{C}{r_Q^n|z|^{m-2}}|f(z)|\, d\mu(z)\\
&\leq C  \int_{B(y_0,2r_Q)} \frac{C}{r_Q^n}|f(z)| \, d\mu(z)\\
&\leq C  \mathcal{M}f(x).
\end{align*}

{\bf Case IV}: $Q\in\mathcal {D}$ and $K\subset Q$. That is, $Q=Q_m\cup Q_n$ where a dyadic cube $Q_m\subset \R^m$ and $Q_m\in \mathcal{I}_1$ and dyadic cube $Q_n\subset \Rf^n$ and $Q_n\in \mathcal{I}_1$  and $\mu(Q_n)>\mu(Q_m)$. The proof of this case is quite
similar to that of Case III.

For each $y\in Q$,
\begin{align}
|\exp (-r_Q\sqrt L)f(y)|&\leq \int_{\R^m} p_{r_Q}(y,z)|f_m(z)| {\,} d\mu(z)+\int_{\Rf^n} p_{r_Q}(y,z)|f_n(z)| {\,} d\mu(z)\nonumber.
\end{align}
For every $z\in\Rf^n$, we have
$$
p_{r_Q}(y,z)\leq \frac{C}{r_Q^m}\bigg(1+\frac{d(y,z)}{r_Q}\bigg)^{-m-1}+\frac{C}{r_Q^n}\bigg(1+\frac{d(y,z)}{r_Q}\bigg)^{-n-1}
\leq \frac{C}{\mu(Q)}\bigg(1+\frac{d(y,z)}{r_Q}\bigg)^{-n-1}.
$$
Then if $x\in Q_n$,
\begin{align*}
\int_{\Rf^n} p_{r_Q}(y,z)|f_n(z)| {\,} d\mu(z)\leq \int_{\Rf^n} \frac{C}{\mu(Q)}\bigg(1+\frac{d(y,z)}{r_Q}\bigg)^{-n-1}|f_n(z)| {\,} d\mu(z)\leq C\mathcal{M}_n(f_n)(x).
\end{align*}
If $x\in Q_m$, for each annuls $\{2^kr_Q\leq |z|< 2^{k+1}r_Q\}\cap \Rf^n$ or $\{|z|\leq 2r_Q\}\cap \Rf^n$,  we  choose one or finite balls $\widetilde B_k$ or $\widetilde B_0$ centered in $\Rf^n$ with radius $2^{k+1}r_Q$ or $4r_Q$, which cover the set
$\{2^kr_Q\leq |z|< 2^{k+1}r_Q\}\cap \Rf^n$ or $\{|z|\leq 2r_Q\}\cap \Rf^n$ and covers $Q_m$ in the large end.  We obtain that
\begin{align*}
&\int_{\Rf^n} p_{r_Q}(y,z)|f_n(z)| {\,} d\mu(z)\\
&\leq \int_{\{|z|\leq 2r_Q\}\cap \Rf^n} \frac{C}{r_Q^n}\bigg(1+\frac{d(y,z)}{r_Q}\bigg)^{-n-1}|f_n(z)| {\,} d\mu(z)\\
&\quad+
\sum_{k\geq 1}\int_{\{2^kr_Q\leq |z|< 2^{k+1}r_Q\}\cap \Rf^n} \frac{C}{r_Q^n}\bigg(1+\frac{d(y,z)}{r_Q}\bigg)^{-n-1}|f_n(z)| {\,} d\mu(z)\\
&\leq \frac{C}{\mu(\widetilde B_0)}\int_{\widetilde B_0} |f_n(z)| {\,} d\mu(z)+
\sum_{k\geq 1}2^{-k(n+1)}\frac{C}{r_Q^n}\int_{\widetilde B_k} |f_n(z)| {\,} d\mu(z)\\
&\leq C\mathcal{M}f(x)+\sum_{k\geq 1}2^{-k} \frac{C}{\mu(\widetilde B_k)}\int_{\widetilde B_k} |f_n(z)| {\,} d\mu(z)\\
&\leq C\mathcal{M}f(x).
\end{align*}

For $z\in\R^m$, we have
$$
p_{r_Q}(y,z)\leq \frac{C}{r_Q^m}\bigg(1+\frac{d(y,z)}{r_Q}\bigg)^{-m-1}+\frac{C}{r_Q^n|z|^{m-2}}\bigg(1+\frac{|z|}{r_Q}\bigg)^{-n-1}.
$$
Then
\begin{align*}
\int_{\R^m} \frac{C}{r_Q^m}\bigg(1+\frac{d(y,z)}{r_Q}\bigg)^{-m-1}|f_m(z)| {\,} d\mu(z)\leq C\mathcal{M}(f)(x)
\end{align*}
It remains to control
\begin{align*}
&\int_{\R^m} \frac{C}{r_Q^n|z|^{m-2}}\bigg(1+\frac{|y|+|z|}{r_Q}\bigg)^{-n-1}|f_m(z)| {\,} d\mu(z)\\
&\leq
\int_{\{|z|>r_Q\}\cap \R^m} \frac{C}{r_Q^n|z|^{m-2}}\bigg(1+\frac{|z|}{r_Q}\bigg)^{-n-1}|f_m(z)| {\,} d\mu(z)+\int_{\{|z|\leq r_Q\}\cap\R^m} \frac{C}{r_Q^n|z|^{m-2}}\bigg(1+\frac{|z|}{r_Q}\bigg)^{-n-1}|f_m(z)| {\,} d\mu(z)\\
&=:I+II.
\end{align*}
For $I$, note that $|z|>r_Q>|x|$,
\begin{align*}
&\int_{\{|z|>r_Q\}\cap \R^m} \frac{C}{r_Q^n|z|^{m-2}}\bigg(1+\frac{|y|+|z|}{r_Q}\bigg)^{-n-1}|f_m(z)| {\,} d\mu(z)\\
&\leq \int_{\{|z|>r_Q\}\cap \R^m} \frac{C}{r_Q^n|z|^{m-2}}\bigg(\frac{|z|}{r_Q}\bigg)^{-n}|f_m(z)| {\,} d\mu(z)\\
&\leq \int_{\{|z|>r_Q\}\cap \R^m} \frac{C}{|z|^{n+m-2}}|f_m(z)| {\,} d\mu(z)\\
&\leq \int_{\{|z|>r_Q\}\cap \R^m} \frac{C}{|x|^{m}}|f_m(z)| {\,} d\mu(z)\\
&\leq C\frac{\|f\|_{L^1(M)}}{|x|^m}.
\end{align*}
For $II$, Fixed a nature number $k_0$ such that $(2^{k_0}r_Q)^{m}\sim r_Q^{n}$.
\begin{align*}
II&= \int_{\{|z|\leq r_Q\}\cap \R^m} \frac{C}{r_Q^n|z|^{m-2}}\left(1+\frac{|z|}{r_Q}\right)^{-n-1}|f(z)| {\,} d\mu(z)\\
&\leq \sum_{k_0\leq k<0}\int_{\{2^kr_Q\leq |z|<2^{k+1}r_Q\}\cap \R^m} \frac{C}{r_Q^n|z|^{m-2}}|f(z)| {\,} d\mu(z)\\
&\quad+\int_{\{|z|\leq 2^{k_0}r_Q\}\cap\R^m} \frac{C}{r_Q^n|z|^{m-2}}|f(z)| {\,} d\mu(z)\\
&=:II_{1}+II_{2}.
\end{align*}
To continue, we first point out that for $k\geq k_0$ and $\{|z|<2^kr_Q\}$, we can choose a ball $B(y_k,2r_Q)\subset M$ with radius $2r_Q$ such that
 $y_k\in \Rf^n$, $(\{|z|<2^kr_Q\}\cap \R^m )\cup Q\subset B(y_k,2r_Q)$ and $\mu(B(y_k,2r_Q))\sim (2^kr_Q)^m$.
 Note that $r_Q\geq 1$ and $n\geq 3$, we have
\begin{align*}
II_{1}&\leq \sum_{k_0\leq k<0}\int_{\{2^kr_Q\leq |z|<2^{k+1}r_Q\}\cap \R^m} \frac{C}{r_Q^n|z|^{m-2}}|f(z)| {\,} d\mu(z)\\
&\leq  \sum_{k_0\leq k<0} \frac{C(2^kr_Q)^2}{r_Q^n(2^kr_Q)^m}\int_{B(y_k,2r_Q)} |f(z)| {\,} d\mu(z)\\
&\leq  \sum_{k_0\leq k<0} 2^{2k}\frac{C}{r_Q^{n-2}}\frac{1}{\mu(B(y_k,2r_Q))}\int_{B(y_k,2r_Q)} |f(z)| {\,} d\mu(z)\\
&\leq  \sum_{k_0\leq k<0} 2^{2k} \mathcal{M}f(x)\\
&\leq  C \mathcal{M}f(x).
\end{align*}
For $II_{2}$, we can choose a ball $B(y_0,2r_Q)\subset M$ with radius $2r_Q$ such that
 $y_0\in \Rf^n$, $(\{|z|<2^{k_0}r_Q\}\cap \R^m) \cup Q\subset B(y_0,2r_Q)$ and $\mu(B(y_0,2r_Q))\sim r_Q^n$.
\begin{align*}
II_{2}&\leq \int_{\{|z|\leq 2^{k_0}r_Q\}\cap\R^m} \frac{C}{r_Q^n|z|^{m-2}}|f(z)| {\,} d\mu(z)\\
&\leq C  \int_{B(y_0,2r_Q)} \frac{C}{r_Q^n}|f(z)| {\,} d\mu(z)\\
&\leq C  \mathcal{M}f(x).
\end{align*}

For the following case, fixed $x\in M$ and $Q\in \mathcal{D}_2$.

{\bf Case V}: $Q\in \mathcal{I}_1$ and $Q\subset \R^m$. For each $y\in Q$, noting that $Q\in \mathcal{I}_1$ implies $r_Q>2$ and then $|y|\geq r_Q^{\frac{m-n}{m-2}}/2$ implies $y\in \R^m\backslash K$. So wherever $z$ is, we have
$$
p_{r_Q}(y,z)\leq \frac{C}{r_Q^m}.
$$
And $Q\in \mathcal{I}_1$ and $x\in Q$ also imply that $|x|<2 r_Q$. Then
$$
|\exp (-r_Q\sqrt L)f(y)|\leq \int_M p_{r_Q}(y,z)|f(z)| {\,} d\mu(z)\leq \int_M \frac{C}{r_Q^m} |f(z)| {\,} d\mu(z)\leq C\frac{\|f\|_{L^1}}{|x|^m}.
$$

The proof of Theorem \ref{th nontangential 2} is complete.
 \end{proof}

\medskip

\section{An equivalent characterization of the space $\bmo$:  proof of Theorem \ref{main0}}
\setcounter{equation}{0}



\begin{proof}[Proof of Theorem \ref{main0}]
First, we show that
\begin{align}\label{main0 inclusion1}
\bmo \subset \text{BMO}_L^\rho(M)
\end{align}
with $ \|f\|_{\text{BMO}_L^\rho(M)}\leq C\|f\|_{\bmo}$ for all $f\in \bmo.$

To see this, for any $f\in \bmo$, we need to show that
\eqref{BMO-def10} and \eqref{BMO-def20} implies \eqref{BMO-def1} and \eqref{BMO-def2}.

We divide all balls $B(x_B, r_B)$ to four classes:

1) $B\subset \Rf^n$ or $B\subset \R^m$ or $r_B\leq 2$;

2)  $K\subset B$, $x_B\in \R^m$ and $r_B>2$;

3) $K\subset B$, $x_B\in \Rf^n$ and $r_{B^m}\leq r_{B}^{\frac{n-\rho}{m-\rho}}$;

4) $B\in \mathcal{B}_0$, that is, $K\subset B$, $x_B\in \Rf^n$ and $r_{B^m}>r_{B}^{\frac{n-\rho}{m-\rho}}$.

For class 1) it is  easy to check that \eqref{BMO-def1} holds.

For class 2), denote $B^n=B\cap \Rf^n$ and $B^m=B\cap \R^m$. Then there exists a fixed number $C(M)$ which depends on the manifold $M$ only such that we can have $C(M)$ balls $B^m_k\subset \R^m$  with $r_{B^m_k}=r_B$ satisfying that $B^m\subset \bigcup\limits_{k=1}^{C(M)} B^m_k$, and $C(M)$ balls $B^n_\ell\subset \Rf^n$  with $r_{B^n_\ell}=r_{B^n}<r_B$ satisfying that $B^n\subset \bigcup\limits_{\ell=1}^{C(M)} B^n_\ell$.
Based this fundamental fact on covering, we have
\begin{eqnarray*}
&&\frac{1}{\mu(B)}\int_B |f(x)-A_{r_B}f(x)| {\,} d\mu(x)\\
&&\leq \frac{1}{\mu(B)}\int_{B^m} |f(x)-A_{r_B}f(x)| {\,} d\mu(x)+\frac{1}{\mu(B)}\int_{B^n} |f(x)-A_{r_B}f(x)| {\,} d\mu(x)\\
&&\leq \sum_{k=1}^{C(M)} \frac{1}{\mu(B)}\int_{B^m_k} |f(x)-A_{r_B}f(x)| {\,} d\mu(x)+ \sum_{\ell=1}^{C(M)} \frac{1}{\mu(B)}\int_{B^n_\ell} |f(x)-A_{r_{B^n}}f(x)| {\,} d\mu(x)\\
&&\quad\quad+ \sum_{\ell=1}^{C(M)} \frac{1}{\mu(B)}\int_{B^n_\ell} |A_{r_B}f(x)-A_{r_{B^n}}f(x)| {\,} d\mu(x)\\
&&=: I+II+III.
\end{eqnarray*}

%
%
%
For term $I$ and $II$, since $\mu(B^m_k)\leq \mu(B)$, $\mu(B^n_\ell)\leq \mu(B)$ and $B^m_k\subset \R^m$, $B^n_\ell\subset \Rf^n$, it follows from \eqref{BMO-def10} that
$$
I+II\leq \sum_{k=1}^{C(M)} \|f\|_{\bmo}+\sum_{\ell=1}^{C(M)} \|f\|_{\bmo}\leq C\|f\|_{\bmo}.
$$
For term $III$, note that $\mu(B)\approx r_B^m$ and $\mu(B^n_\ell)=r_{B^n}^n$. It follows from the fact that $r_{B^n}<r_B$ and from \eqref{BMO-def20}    that
$$
III\leq \sum_{\ell=1}^{C(M)} \frac{r_{B^n}^n}{r_B^m}\left(1+\left|\log\frac{r_{B^n}}{r_B}\right|\right)\|f\|_{\bmo}\leq \sum_{\ell=1}^{C(M)} \bigg(\frac{r_{B^n}}{r_B}\bigg)^n\left(1+\left|\log\frac{r_{B^n}}{r_B}\right|\right)\|f\|_{\bmo}\leq C\|f\|_{\bmo}.
$$
Combining the estimates for these three terms, we obtain that \eqref{BMO-def1} holds.

For class 3), similar to the above estimate, we have at most $C(M)$ balls $B^m_k\subset \R^m$  with $r_{B^m_k}=r_{B^m}$ such that $B^m\subset \bigcup\limits_{k=1}^{C(M)} B^m_k$ and at most $C(M)$ balls $B^n_\ell\subset \Rf^n$  with $r_{B^n_\ell}=r_B$ such that $B^n\subset \bigcup\limits_{\ell=1}^{C(M)} B^n_\ell$.
By these facts,
\begin{eqnarray*}
&&\frac{1}{\mu(B)}\int_B |f(x)-A_{r_B}f(x)| {\,} d\mu(x)\\
&&\leq \sum_{k=1}^{C(M)} \frac{1}{\mu(B)}\int_{B^m_k} |f(x)-A_{r_{B^m}}f(x)| {\,} d\mu(x)+ \sum_{\ell=1}^{C(M)} \frac{1}{\mu(B)}\int_{B^n_\ell} |f(x)-A_{r_{B}}f(x)| {\,} d\mu(x)\\
&&\quad\quad+ \sum_{k=1}^{C(M)} \frac{1}{\mu(B)}\int_{B^m_k} |A_{r_B}f(x)-A_{r_{B^m}}f(x)| {\,} d\mu(x)\\
&&=: I+II+III.
\end{eqnarray*}
For the terms $I$ and $II$, 
similar to the argument in class 2) we have
$
I+II\leq C\|f\|_{\bmo}.
$

For term $III$, note that $\mu(B)\geq r_B^n$ and $\mu(B^m_k)=r_{B^m}^m$. It follows from \eqref{BMO-def20},
$$
III<\sum_{\ell=1}^{C(M)}  \frac{r_{B^m}^m}{r_B^n}\left(1+\left|\log\frac{r_{B^m}}{r_B}\right|\right)\|f\|_{\bmo}.
$$
To continue, we consider the following two cases. For $r_{B^m}\leq 1$, it is direct that
$$
III<\sum_{\ell=1}^{C(M)} \bigg( \frac{r_{B^m}}{r_B}\bigg)^n\left(1+\left|\log\frac{r_{B^m}}{r_B}\right|\right)\|f\|_{\bmo}\leq C\|f\|_{\bmo}.
$$
For $r_{B^m}>1$, in class 3), $r_{B^m}\leq r_{B}^{\frac{n-\rho}{m-\rho}}$ and it leads
$$
III<\sum_\ell r_{B}^{-\frac{(m-n)\rho}{m-\rho}}\left(1+\log r_B\right)\|f\|_{\bmo}\leq C\|f\|_{\bmo}.
$$
Combining the estimates for these three terms, we obtain that \eqref{BMO-def1} holds.

For class 4), again we  have at most $C(M)$ balls $B^m_k\subset \R^m$  with $r_{B^m_k}=r_{B^m}$ such that $B^m\subset \bigcup\limits_{k=1}^{C(M)} B^m_k$ and finite number of balls $B^n_\ell\subset \Rf^n$  with $r_{B^n_\ell}=r_B$ such that $B^n\subset \bigcup\limits_{\ell=1}^{C(M)} B^n_\ell$.
By these facts,
\begin{eqnarray*}
&&\frac{1}{\mu(B)}\int_B |f(x)-A_{r_B}f(x)| {\,} d\mu(x)\\
&&\leq \sum_k \frac{1}{\mu(B)}\int_{B^m_k} |f(x)-A_{r_{B^m}}f(x)| {\,} d\mu(x)+ \sum_\ell \frac{1}{\mu(B)}\int_{B^n_\ell} |f(x)-A_{r_{B}}f(x)| {\,} d\mu(x)\\
&&\quad\quad+ \sum_k \frac{1}{\mu(B)}\int_{B^m_k} |A_{r_B}f(x)-A_{r_{B^m}}f(x)| {\,} d\mu(x)\\
&&=: I+II+III.
\end{eqnarray*}
Again, for the terms $I$ and $II$,
similar to the argument in class 2) we have
$
I+II\leq C\|f\|_{\bmo}.
$

For term $III$, note that $\mu(B^m)\leq \mu(B)$. It follows from \eqref{BMO-def20} that
$$
III<\sum_\ell \frac{\mu(B^m)}{\mu(B)}\left(1+\left|\log\frac{r_{B^m}}{r_B}\right|\right)\|f\|_{\bmo}\leq C\log r_B \|f\|_{\bmo}.
$$
This implies that  \eqref{BMO-def2} holds.

Combining all estimates in class 1)---class 4), we obtain that \eqref{main0 inclusion1} holds and it is clear that
$ \|f\|_{\text{BMO}_L^\rho(M)}\leq C\|f\|_{\bmo}$ for all $f\in \bmo$.

\bigskip

Next, we show that
\begin{align}\label{main0 inclusion2}
\bmo \supset \text{BMO}_L^\rho(M)
\end{align}
with $\|f\|_{\bmo} \leq C\|f\|_{\text{BMO}_L^\rho(M)}$ for all $f\in \text{BMO}_L^\rho(M) .$

To see this, we will show that \eqref{BMO-def1} and \eqref{BMO-def2} implies \eqref{BMO-def10} and \eqref{BMO-def20}  for $f\in \text{BMO}_L^\rho(M) .$

In fact, suppose $f\in \text{BMO}_L^\rho(M) $, then \eqref{BMO-def10} follows direcly from \eqref{BMO-def1} with the constant $\|f\|_{\text{BMO}_L^\rho(M)}$. Hence it remains to prove \eqref{BMO-def20}.
We point out that it suffices to
 prove that
 there exists a positive constant $C$ such that for every $f\in \text{BMO}_L^\rho(M) $, $x\in M$  and for every $s,t>0$ with  $s/4<t\leq s$,
\begin{align}\label{key step 4.1 pre}
|A_sf(x)-A_{s+t}f(x)|\leq C\|f\|_{\text{BMO}_L^\rho(M)}.
\end{align}
In fact, if \eqref{key step 4.1 pre} holds, then
for every $f\in \text{BMO}_L^\rho(M)$, $x\in M$  and for every $s,t>0$ with $0<t<s/4$, we have
\begin{eqnarray*}
|A_sf(x)-A_{s+t}f(x)|\leq |A_sf(x)-A_{s+s}f(x)|+|A_{s+s}f(x)-A_{s+t}f(x)|\leq 2C\|f\|_{\text{BMO}_L^\rho(M)}.
\end{eqnarray*}
Combining this estimate and the one in \eqref{key step 4.1 pre} we have obtain that for  every $s,t>0$ with $t<s$,
\begin{eqnarray*}
|A_sf(x)-A_{s+t}f(x)|\leq 2C\|f\|_{\text{BMO}_L^\rho(M)}.
\end{eqnarray*}
As a consequence, for any $\mathcal K>1$, by chooisng $\ell$ to be the integer satisfying $2^\ell \leq \mathcal K < 2^\ell+1$ (hence $\ell \leq 2\log \mathcal K$) we obtain that
\begin{align*}
|A_sf(x)-A_{\mathcal Ks}f(x)|&\leq \sum_{k=0}^{\ell-1}|A_{2^ks}f(x)-A_{2^{k+1}s}f(x)|+|A_{2^{\ell}s}f(x)-A_{\mathcal Ks}f(x)|\\
&\leq C(1+\log \mathcal K)\|f\|_{\text{BMO}_L^\rho(M)},
\end{align*}
which implies that \eqref{BMO-def20} holds.

We now prove \eqref{key step 4.1 pre}.
First note that
\begin{align*}
|A_sf(x)-A_{s+t}f(x)|=|A_s(f-A_tf)(x)|
\leq \int_M p_s(x,y)|f(y)-A_tf(y)| {\,} d\mu(y).
\end{align*}
Note that for the following proof of this lemma, we only need the upper bound of the Poisson kernel. For the Poisson kernel, because of
$s/4<t<s$, we can control the upper bound of $p_s(x,y)$ by $4^{m+n+1}p_t(x,y)$. So for what following, we write
\begin{align}\label{AtAs integral}
|A_sf(x)-A_{s+t}f(x)|=|A_s(f-A_tf)(x)|
\leq C\int_M p_t(x,y)|f(y)-A_tf(y)| {\,} d\mu(y).
\end{align}

To begin with, for $x\in M$, $t>0$ and $k\geq1$, we set
annulus
\begin{align}\label{annuli}
B(x,t,k):=\{ z\in M: 2^{k-1}t\leq d(x,z)<2^kt \}.
\end{align}

{\bf Case I}: $0<t\leq 1$.

By  writing the integral in the right-hand side of \eqref{AtAs integral} into two parts, we have
\begin{align*}
|A_sf(x)-A_{s+t}f(x)|
&\leq C\int_{\Rf^n} p_t(x,y)|f(y)-A_tf(y)| {\,} d\mu(y)+C\int_{\R^m\backslash K} p_t(x,y)|f(y)-A_tf(y)| {\,} d\mu(y)\\
&=:I+II.
\end{align*}

We first estimate the term $I$.

From Proposition~\ref{prop2.1}, we get that in this case, for $x\in M$, the Poisson kernel is bounded by
$$
p_t(x,y)\leq \frac{C}{t^m}\left(1+\frac{d(x,y)}{t}\right)^{-m-1}+\frac{C}{t^n}\left(1+\frac{d(x,y)}{t}\right)^{-n-1}.
$$
Then
\begin{align*}
I&\leq C\int_{\Rf^n} \frac{C}{t^m}\left(1+\frac{d(x,y)}{t}\right)^{-m-1}|f(y)-A_tf(y)| {\,} d\mu(y)\\
&\quad+
C\int_{\Rf^n} \frac{C}{t^n}\left(1+\frac{d(x,y)}{t}\right)^{-n-1}|f(y)-A_tf(y)| {\,} d\mu(y)\\
&=:I_1+I_2.
\end{align*}

Then by using the annuli \eqref{annuli} we have
\begin{align*}
I_1
&\leq \int_{B(x,t)\cap \Rf^n} \frac{C}{t^m}|f(y)-A_tf(y)| {\,} d\mu(y)\\
&\quad+\sum_{k\geq 1}\int_{B(x,t,k)\cap \Rf^n} \frac{C}{t^{m}}\left(1+\frac{d(x,y)}{t}\right)^{-m-1}|f(y)-A_tf(y)| {\,} d\mu(y)\\
&=: I_{11}+I_{12}.
\end{align*}
It follows from \eqref{BMO-def1} directly that the first term $I_{11}$ is bounded by
$
\|f\|_{\text{BMO}_L^\rho(M)}
$ since when $t\leq1$, $B(x,t)$ is in $\mathcal B_1^\rho$.
%
%
As for the second term, for each $k\geq 1$, we first note that for $y\in B(x,t,k)\cap \Rf^n$,
$$
\frac{C}{t^{m}}\left(1+\frac{d(x,y)}{t}\right)^{-m-1}\leq  \frac{C}{t^{m}}2^{-k(m+1)}.
$$
Next, if $2^kt<1$, we can choose at most $(2^{k}t)^{m}/t^m$ of balls $B_\ell(t)$ in $\Rf^n$ with radius $t$ to cover $B(x,t,k)\cap \Rf^n$.
If $2^kt\geq 1$, we can choose at most $(2^{k}t)^{n}/t^m$ of balls $B_\ell(t)$ in $\Rf^n$ with radius $t$ to cover $B(x,t,k)\cap \Rf^n$.
In each case, we can choose at most $2^{km}$ of balls $B_\ell(t)$ with radius $t$ to cover $B(x,t,k)\cap \Rf^n$. Thus, the second term
\begin{align*}
I_{12}
&\leq \sum_{k\geq 1}\int_{B(x,t,k)\cap \Rf^n} \frac{C}{t^m}2^{-k(m+1)}|f(y)-A_tf(y)| {\,} d\mu(y)\\
&\leq C \sum_{k\geq 1}2^{km}2^{-k(m+1)} \sup_\ell \frac{1}{t^m}\int_{B_\ell(t)} |f(y)-A_tf(y)| {\,} d\mu(y)\\
&\leq C \sup_\ell\frac{1}{\mu(B_\ell(t))}\int_{B_\ell(t)}|f(y)-A_tf(y)| {\,} d\mu(y)\\
&\leq C \|f\|_{\text{BMO}_L^\rho(M)},
\end{align*}
where the last inequality follows from \eqref{BMO-def1} since each $B_\ell(t)$ is in $\mathcal B_1^\rho$.

For the term $I_2$, using the annuli again, we have
\begin{align*}
I_2 
&\leq \int_{B(x,t)\cap \Rf^n} \frac{C}{t^n}|f(y)-A_tf(y)| {\,} d\mu(y)\\
&\quad+\sum_{k\geq 1,2^{k}t<1}\int_{B(x,t,k)\cap \Rf^n} \frac{C}{t^{n}}\left(1+\frac{d(x,y)}{t}\right)^{-n-1}|f(y)-A_tf(y)| {\,} d\mu(y)\\
&\quad+\sum_{k\geq 1, 2^{k}t\geq 1}\int_{B(x,t,k)\cap \Rf^n} \frac{C}{t^{n}}\left(1+\frac{d(x,y)}{t}\right)^{-n-1}|f(y)-A_tf(y)| {\,} d\mu(y)\\
&=: I_{21}+I_{22}+I_{23}.
\end{align*}
For the first term, since $B(x,t)$ is in $\mathcal B_1^\rho$ and ${1\over t^n} \leq  {C\over t^m}$,  from \eqref{BMO-def1}, we get that
$
I_{21}\leq C\|f\|_{\text{BMO}_L^\rho(M)}.
$

For the term $I_{22}$, we  choose at most $2^{km}$ of balls $B_\ell(t)$ in $\Rf^n$ with radius $t$ to cover $B(x,t,k)\cap \Rf^n$. Thus,
\begin{align*}
I_{22}
&\leq \sum_{k\geq 1,2^kt\leq 1}\int_{B(x,t,k)\cap \Rf^n} \frac{C}{t^n}2^{-k(n+1)}|f(y)-A_tf(y)| {\,} d\mu(y)\\
&\leq C \sum_{k\geq 1,2^kt\leq 1}2^{km}2^{-k(n+1)}t^{m-n} \sup_\ell \frac{1}{t^m}\int_{B_\ell(t)} |f(y)-A_tf(y)| {\,} d\mu(y)\\
&\leq C \sup_\ell\frac{1}{\mu(B_\ell(t))}\int_{B_\ell(t)}|f(y)-A_tf(y)| {\,} d\mu(y)\\
&\leq C \|f\|_{\text{BMO}_L^\rho(M)},
\end{align*}
where the last inequality follows from \eqref{BMO-def1} since each $B_\ell(t)$ is in $\mathcal B_1^\rho$.

For $I_{23}$, we can choose at most $(2^{k}t)^n/t^m$ of balls $B_\ell(t)$ with radius $t$ to cover $B(x,t,k)\cap \Rf^n$. Thus,
\begin{align*}
I_{23}
&\leq \sum_{2^kt> 1}\int_{B(x,t,k)\cap \Rf^n} \frac{C}{t^n}2^{-k(n+1)}|f(y)-A_tf(y)| {\,} d\mu(y)\\
&\leq C \sum_{2^kt> 1}\frac{(2^{k}t)^n}{t^m}2^{-k(n+1)}t^{m-n} \sup_\ell \frac{1}{t^m}\int_{B_\ell(t)} |f(y)-A_tf(y)| {\,} d\mu(y)\\
&\leq C \sup_\ell\frac{1}{\mu(B_\ell(t))}\int_{B_\ell(t)}|f(y)-A_tf(y)| {\,} d\mu(y)\\
&\leq C \|f\|_{\text{BMO}_L^\rho(M)},
\end{align*}
where the last inequality follows from \eqref{BMO-def1} since each $B_\ell(t)$ is in $\mathcal B_1^\rho$.

We now estimate the term $II$.

From Proposition~\ref{prop2.1}, we get that in this case, for $x\in M$, the Poisson kernel is bounded by
$$
p_t(x,y)\leq \frac{C}{t^m}\left(1+\frac{d(x,y)}{t}\right)^{-m-1}+\frac{C}{t^n|y|^{m-2}}\left(1+\frac{|y|}{t}\right)^{-n-1}.
$$
Then
\begin{align*}
II&\leq C\int_{\R^m\backslash K} \frac{C}{t^m}\left(1+\frac{d(x,y)}{t}\right)^{-m-1}|f(y)-A_tf(y)| {\,} d\mu(y)\\
\quad&+
C\int_{\R^m\backslash K} \frac{C}{t^n|y|^{m-2}}\left(1+\frac{|y|}{t}\right)^{-n-1}|f(y)-A_tf(y)| {\,} d\mu(y)\\
&=:II_1+II_2.
\end{align*}

Similar to $I_1$, we can get that  $II_1$ is bounded by $C\|f\|_{\text{BMO}_L^\rho(M)}$.

We now consider $II_2$.
\begin{align*}
II_2&\leq \int_{\R^m\backslash K} \frac{C}{t^n|y|^{m-2}}\left(1+\frac{|y|}{t}\right)^{-n-1}|f(y)-A_tf(y)| {\,} d\mu(y)\\
&\leq \sum_{2^{k}t\geq 1}\int_{\{y\in \R^m\backslash K:\ 2^kt\leq |y|<2^{k+1}t\}} \frac{C}{t^n|y|^{m-2}}\left(1+\frac{|y|}{t}\right)^{-n-1}|f(y)-A_tf(y)| {\,} d\mu(y).
\end{align*}
We can choose at most $2^{km}$ of balls $B_\ell(t)$ with radius $t$ to cover ${\{y\in \R^m\backslash K:\ 2^kt\leq |y|<2^{k+1}t\}}$. Thus, we have
\begin{align*}
II_{2}
&\leq \sum_{2^kt> 1}\int_{\{y\in \R^m\backslash K:\ 2^kt\leq |y|<2^{k+1}t\}} \frac{C}{t^n|y|^{m-2}}2^{-k(n+1)}|f(y)-A_tf(y)| {\,} d\mu(y)\\
&\leq C \sum_{2^kt> 1}2^{km}2^{-k(n+1)}t^{m-n}(2^kt)^{2-m} \sup_\ell \frac{1}{t^m}\int_{B_\ell(t)} |f(y)-A_tf(y)| {\,} d\mu(y)\\
&\leq C \sum_{2^kt> 1}(2^kt)^{2-n}2^{-k} \sup_\ell \frac{1}{t^m}\int_{B_\ell(t)} |f(y)-A_tf(y)| {\,} d\mu(y)\\
&\leq C \sup_\ell\frac{1}{\mu(B_\ell(t))}\int_{B_\ell(t)}|f(y)-A_tf(y)| {\,} d\mu(y)\\
&\leq C \|f\|_{\text{BMO}_L^\rho(M)},
\end{align*}
where the last inequality follows from \eqref{BMO-def1} since each $B_\ell(t)$ is in $\mathcal B_1^\rho$.

Combining the estimates of the terms $I$ and $II$ we obtain that \eqref{key step 4.1 pre} holds for $t\leq1$.

{\bf Case II}: $t>1$.  

Again, by  writing the integral in the right-hand side of \eqref{AtAs integral} into two parts, we have
\begin{align*}
|A_tf(x)-A_{s+t}f(x)|
&\leq \int_{\Rf^n} p_t(x,y)|f(y)-A_tf(y)| {\,} d\mu(y)+\int_{\R^m\backslash K} p_t(x,y)|f(y)-A_tf(y)| {\,} d\mu(y)\\
&=:I+II.
\end{align*}

We first estimate $I$.

From Proposition~\ref{prop2.1}, we see that in this case, for $x\in M$, the Poisson kernel is bounded by
$$
p_t(x,y)\leq \frac{C}{t^n}\left(1+\frac{d(x,y)}{t}\right)^{-n-1}.
$$
Then we have
\begin{align*}
I&\leq
\int_{\Rf^n} \frac{C}{t^n}\left(1+\frac{d(x,y)}{t}\right)^{-n-1}|f(y)-A_tf(y)| {\,} d\mu(y).
\end{align*}

By decomposing $\Rf^n$ into annuli as in \eqref{annuli},  we have
\begin{align*}
I
&\leq \int_{B(x,t)\cap \Rf^n} \frac{C}{t^n}|f(y)-A_tf(y)| {\,} d\mu(y)\\
&\quad+\sum_{k\geq 1}\int_{B(x,t,k)\cap\Rf^n} \frac{C}{t^{n}}\left(1+\frac{d(x,y)}{t}\right)^{-n-1}|f(y)-A_tf(y)| {\,} d\mu(y)\\
&=: I_{1}+I_{2}.
\end{align*}
For the  term $I_{1}$,  we can choose at most $2^{n}$ of balls $B_\ell(t)\subset \Rf^n$ with radius $t$
to cover $B(x,t)\cap \Rf^n$.  Then
$$
I_{1}\leq
 C\sup_{\ell}\frac{1}{\mu(B_\ell(t))}\int_{B_\ell(t)}|f(y)-A_tf(y)| {\,} d\mu(y)\leq C\|f\|_{\text{BMO}_L^\rho(M)},
$$
where the last inequality follows from \eqref{BMO-def1} since each $B_\ell(t)$ is in $\mathcal B_1^\rho$.

As for the second term, for each $k\geq 1$, we first note that for $y\in B(x,t,k)\cap\Rf^n$,
$$
\frac{C}{t^{n}}\left(1+\frac{d(x,y)}{t}\right)^{-n-1}\leq  \frac{C}{t^{n}}2^{-k(n+1)}.
$$
Next, we can choose at most $2^{kn}$ of balls $B_\ell(t)\subset \Rf^n$ with radius $t$ to cover $B(x,t,k)\cap\Rf^n$. Thus,
\begin{align*}
I_{2}
&\leq \sum_{k\geq 1}\int_{B(x,t,k)\cap\Rf^n} \frac{C}{t^n}2^{-k(n+1)}|f(y)-A_tf(y)| {\,} d\mu(y)\\
&\leq C \sum_{k\geq 1}2^{kn}2^{-k(n+1)} \sup_\ell \frac{1}{\mu(B_\ell(t))}\int_{B_\ell(t)} |f(y)-A_tf(y)| {\,} d\mu(y)\\
&\leq C \|f\|_{\text{BMO}_L^\rho(M)},
\end{align*}
where the last inequality follows from \eqref{BMO-def1} since each $B_\ell(t)$ is in $\mathcal B_1^\rho$.

We now consider the term $II$.
From Proposition~\ref{prop2.1}, we see that in this case, for $x\in M$, the Poisson kernel is bounded by
$$
p_t(x,y)\leq \frac{C}{t^m}\left(1+\frac{d(x,y)}{t}\right)^{-m-1}+\frac{C}{t^n|y|^{m-2}}\left(1+\frac{|y|}{t}\right)^{-n-1}.
$$
Then we have
\begin{align*}
II&\leq \int_{\R^m\backslash K} \frac{C}{t^m}\left(1+\frac{d(x,y)}{t}\right)^{-m-1}|f(y)-A_tf(y)| {\,} d\mu(y)\\
&\quad+
\int_{\R^m\backslash K} \frac{C}{t^n|y|^{m-2}}\left(1+\frac{|y|}{t}\right)^{-n-1}|f(y)-A_tf(y)| {\,} d\mu(y)\\
&=:II_1+II_2.
\end{align*}
Now by decomposing $\R^m\backslash K$ into annuli we have
\begin{align*}
II_1&\leq \int_{\R^m\backslash K} \frac{C}{t^m}\left(1+\frac{d(x,y)}{t}\right)^{-m-1}|f(y)-A_tf(y)| {\,} d\mu(y)\\
&\leq \int_{B(x,t)\cap(\R^m\backslash K)} \frac{C}{t^m}|f(y)-A_tf(y)| {\,} d\mu(y)\\
&\quad+\sum_{k\geq 1}\int_{B(x,t,k)\cap (\R^m\backslash K)} \frac{C}{t^{m}}\left(1+\frac{d(x,y)}{t}\right)^{-m-1}|f(y)-A_tf(y)| {\,} d\mu(y)\\
&=: II_{11}+II_{12}.
\end{align*}
For the first term $II_{11}$, if $B(x,t)\cap(\R^m\backslash K)=\emptyset$, then there is nothing to prove. If $B(x,t)\cap(\R^m\backslash K)\not=\emptyset$, then there are at most $2^m$ balls $B_\ell(t)$ centered at $\R^m\backslash K$ with radius $t$ whose union covers $B(x,t)\cap(\R^m\backslash K)$. Hence, $II_{11}$
 is bounded by
$$
\sum_\ell \frac{C}{B_{\ell}(t)}\int_{B_\ell(t)}|f(y)-A_tf(y)| {\,} d\mu(y)\leq C\|f\|_{\text{BMO}_L^\rho(M)},
$$
where the last inequality follows from \eqref{BMO-def1} since each $B_\ell(t)$ is in $\mathcal B_1^\rho$.

%
%
As for the second term, for each $k\geq 1$, we first note that for $y\in B(x,t,k)\cap (\R^m\backslash K)$,
$$
\frac{C}{t^{m}}\left(1+\frac{d(x,y)}{t}\right)^{-m-1}\leq  \frac{C}{t^{m}}2^{-k(m+1)}.
$$
Next, we can choose at most $2^{km}$ of balls $B_\ell(t)$ with radius $t$ to cover $B(x,t,k)\cap (\R^m\backslash K)$. Thus, the second term
\begin{align*}
II_{12}
&\leq \sum_{k\geq 1}\int_{B(x,t,k)\cap (\R^m\backslash K)} \frac{C}{t^m}2^{-k(m+1)}|f(y)-A_tf(y)| {\,} d\mu(y)\\
&\leq C \sum_{k\geq 1}2^{km}2^{-k(m+1)} \sup_\ell \frac{1}{\mu(B_\ell(t))}\int_{B_\ell(t)} |f(y)-A_tf(y)| {\,} d\mu(y)\\
&\leq C \|f\|_{\text{BMO}_L^\rho(M)},
\end{align*}
where the last inequality follows from \eqref{BMO-def1} since each $B_\ell(t)$ is in $\mathcal B_1^\rho$.

It remains to estimate $II_2$. Fixed a nature number $k_0$ such that $(2^{k_0}t)^{m-\rho}\sim t^{n-\rho}$.
\begin{align*}
II_2&\leq \int_{\R^m\backslash K} \frac{C}{t^n|y|^{m-2}}\left(1+\frac{|y|}{t}\right)^{-n-1}|f(y)-A_tf(y)| {\,} d\mu(y)\\
&\leq \sum_{k\geq 0}\int_{\{y\in\R^m\backslash K:\ 2^kt\leq |y|<2^{k+1}t\}} \frac{C}{t^n|y|^{m-2}}\left(1+\frac{|y|}{t}\right)^{-n-1}|f(y)-A_tf(y)| {\,} d\mu(y)\\
&\quad+\sum_{k_0\leq k<0}\int_{\{y\in\R^m\backslash K:\ 2^kt\leq |y|<2^{k+1}t\}} \frac{C}{t^n|y|^{m-2}}|f(y)-A_tf(y)| {\,} d\mu(y)\\
&\quad+\int_{\{y\in\R^m\backslash K:\ |y|\leq 2^{k_0}t\}} \frac{C}{t^n|y|^{m-2}}|f(y)-A_tf(y)| {\,} d\mu(y)\\
&=:II_{21}+II_{22}+II_{23}.
\end{align*}
For $k>0$, we can choose at most $2^{km}$ of balls $B_\ell(t)\subset \R^m$ with radius $t$ to cover the annuli $\{y\in \R^m\backslash K:\ 2^kt\leq |y|<2^{k+1}t\}$. Thus,
\begin{align*}
II_{21}
&\leq \sum_{k\geq 0}\int_{\{y\in\R^m\backslash K:\ 2^kt\leq |y|<2^{k+1}t\}} \frac{C}{t^n|y|^{m-2}}2^{-k(n+1)}|f(y)-A_tf(y)| {\,} d\mu(y)\\
&\leq C \sum_{k\geq 0}2^{km}2^{-k(n+1)}t^{m-n}(2^kt)^{2-m} \sup_\ell \frac{1}{\mu(B_\ell(t))}\int_{B_\ell(t)} |f(y)-A_tf(y)| {\,} d\mu(y)\\
&\leq C \sum_{k\geq 0}(2^kt)^{2-n}2^{-k} \|f\|_{\text{BMO}_L^\rho(M)}\\
&\leq C \|f\|_{\text{BMO}_L^\rho(M)},
\end{align*}
where the third inequality follows from \eqref{BMO-def1} since each $B_\ell(t)$ is in $\mathcal B_1^\rho$ and
the last inequality follows from the facts that $n\geq 3$ and that $t>1$.

To continue, we consider $II_2$. We first point out that for $ k_0\leq k<0$ and $\{y\in\R^m\backslash K:\ |y|<2^kt\}$, we can choose a ball $B(z_k,t)\subset M$ with radius $t$ such that
 $z_k\in \Rf^n$, $$\{y\in\R^m\backslash K:\ |y|<2^kt\}\subset B(z_k,t)$$ and that $\mu(B(z_k,t))\sim \max\{(2^kt)^m,t^n\}$. Then
\begin{align*}
II_{22}
&\leq \sum_{k_0\leq k<0}\int_{\{y\in \R^m\backslash K:\ 2^kt\leq |y|<2^{k+1}t\}} \frac{C}{t^n|y|^{m-2}}|f(y)-A_tf(y)| {\,} d\mu(y)\\
&\leq C \sum_{k_0\leq k<0}\max\{(2^kt)^m,t^n\}t^{-n}(2^kt)^{2-m}\log t \, \frac{1}{\mu(B(z_k,t))\log t}\int_{B(z_k,t)} |f(y)-A_tf(y)| {\,} d\mu(y)\\
&\leq C \sum_{k_0\leq k<0}(2^{2k}t^{2-n}\log t+(2^kt)^{2-m}\log t) \|f\|_{\text{BMO}_L^\rho(M)}\\
&\leq C \Big(\sum_{k_0\leq k<0} 2^{2k}+\sum_{(2^kt)\geq 1} (2^kt)^{2-n} (2^{k_0}t)^{n-m}\log t\Big)\|f\|_{\text{BMO}_L^\rho(M)}\\
&\leq C \Big(1+\sum_{(2^kt)\geq 1} (2^kt)^{2-n} t^{\frac{n-\rho}{m-\rho}(n-m)}\log t\Big)\|f\|_{\text{BMO}_L^\rho(M)}\\
&\leq C \|f\|_{\text{BMO}_L^\rho(M)},
\end{align*}
where the third inequality follows from \eqref{BMO-def2} since each $B(z_k,t)$ is in $\mathcal B_0^\rho$,
the fourth inequality follows from the fact that $t>1$ and the last inequality follows from the facts that $n<m$ and that $t>1$.

For $II_{23}$, we can choose a ball $B(z_0,t)\subset M$ with radius $t$ such that
 $z_0\in \Rf^n$, $$\{y\in\R^m\backslash K:\ |y|<2^{k_0}t\}\subset B(z_0,t)$$ and that $\mu(B(z_0,t))\sim t^n$. Note that, this time, $B(z_0,t)\notin \mathcal{B}_0$ and thus there is no $\log$ term in the definition of $\text{BMO}_L^\rho(M)$.
\begin{align*}
II_{23}
&\leq \int_{\{y\in\R^m\backslash K:\ |y|\leq 2^{k_0}t\}} \frac{C}{t^n|y|^{m-2}}|f(y)-A_tf(y)| {\,} d\mu(y)\\
&\leq C  \frac{1}{\mu(B(z_0,t))}\int_{B(z_0,t)} |f(y)-A_tf(y)| {\,} d\mu(y)\\
&\leq C \|f\|_{\text{BMO}_L^\rho(M)}.
\end{align*}

Combining all the estimates above, we see that \eqref{key step 4.1 pre} holds, and hence
we finish the proof of \eqref{BMO-def20}.
The proof of  Theorem \ref{main0} is complete.
\end{proof}

\section{John--Nirenberg inequality, proof of Theorem \ref{main1}}
\setcounter{equation}{0}

In this section, we will prove Theorem \ref{main1}, the John--Nirenberg inequality for $\bmo$.

\begin{proof}[Proof of Theorem \ref{main1}]
We point out that the inequalities \eqref{eq44.1} and \eqref{eq44.11} are scale invariant with respect to $f$, i.e., \eqref{eq44.1} and \eqref{eq44.11} do not change when we replace $f$ by $Cf$ where $C$ is an arbitrary constant. Thus, it is enough to prove that there exist two positive constants $c_1$ and $c_2$ such that  for  $f\in \bmo$ with $\|f\|_{\bmo}=1$ and for every ball $B\subset \mathcal B_1$,
\begin{eqnarray}\label{eq44.2}
\mu\big(\{x\in B:|f(x)-A_{r_B}f(x)|>\alpha\}\big)\leq c_1e^{-c_2\alpha}\mu(B).
\end{eqnarray}
Similarly, to prove \eqref{eq44.11}, it suffices to show that there exist two positive constants $c_1$ and $c_2$ such that  for  $f\in \bmo$ with $\|f\|_{\bmo}=1$
and for every ball $B\subset \mathcal B_0$
\begin{eqnarray}\label{eq44.223}
\mu\big(\{x\in B:|f(x)-A_{r_B}f(x)|>\alpha\}\big)\leq c_1e^{-c_2\alpha/\log r_B}\mu(B).
\end{eqnarray}
It is obvious that in the case $\alpha\leq 2$, the above inequality \eqref{eq44.2} and \eqref{eq44.223} are true for $c_1=e^2$ and $c_2=1$. Hence, it suffices to prove \eqref{eq44.2} and \eqref{eq44.223} for  $\alpha>2$.

For any fixed ball $B\subset M$, denote by $x_B$ and $r_B$
the center and the radius of the ball $B$.

\medskip
{\bf Case I:}\ $|x_B|<2r_B$ and $r_B>1$.


We assume that the ball $B(x_B,r_B)$ is divided  to two parts:
$$
B\cap \R^m  \ \mbox{and}\ B\cap \Rf^n,
$$
and some of the parts will be empty set if $B$ does not have intersection with the corresponding ends.

Recall that
$$
\mathcal{I}_1=\{\mbox{dyadic cubes $Q\subset \R^m$ or $Q\subset \Rf^n$ such that $r_Q\geq 2$ and dist$(Q,K)\leq r_Q$}\}
$$
and
$$
\mathcal{I}_2=\{\mbox{dyadic cubes $Q\subset \R^m$ or $Q\subset \Rf^n$ such that $r_Q<2$ or dist$(Q,K)> r_Q$}\}.
$$

If $x_B$ is in $\R^m$, then there exists a set of dyadic cubes $\{Q_i\}_{i=1}^{ 2^{m}}$ on $\R^m$ with side-length equivalent to $r_{B}$ that covers $B\cap \R^m$. Because $x_B\leq 2r_B$, we can choose all $Q_i\in \mathcal I_1$. In this case, if $B\cap \Rf^n$ is non-empty then there exists a set of dyadic cubes $\{Q'_j\}_{j=1}^{2^n}$ on $\Rf^n$ with side-length equivalent to $r_{B^n}$ that covers $(B\backslash K)\cap \Rf^n$ where we use $r_{B^n}$ to denote the radius  of  $B\cap \Rf^n$. Similarly, all $Q'_j\in \mathcal I_1$.

Similarly, If $x_B$ is in $\Rf^n$, then there exists a set of dyadic cubes $\{Q'_j\}_{j=1}^{ 2^{n}}$ on $\Rf^n$ with side-length equivalent to $r_{B}$ that covers $B\cap \Rf^n$. In this case, if $B\cap \R^m$ is non-empty then there exists a set of dyadic cubes $\{Q_i\}_{i=1}^{2^m}$ on $\R^m$ with side-length equivalent to $r_{B^m}$ that covers $B\cap \R^m$ where we use $r_{B^m}$ to denote the radius  of  $B\cap \R^m$.

As a consequence, we see that in any case we can find two sets of smallest dyadic cubes $\{Q_i\}\in \R^m$ and $\{Q'_j\}\in \Rf^n$ such that
\begin{align}\label{covering of B}
B\subset \big(\cup Q_i\big)\bigcup \big(\cup Q'_j\big).
\end{align}

Now to continue, it suffices to consider the estimate in one of the cubes in $\big\{ Q_i\big\} $ as well as one of the cubes in $\big\{\cup Q'_j\big\}$
since the side length of each $Q_i$ (each $Q'_j$) is the same.

To begin with, we pick one cube $Q$ from  $\big\{ Q_i\big\} $.
Set $Q=Q_0$ which is a dyadic cube such that one of the corner of $Q_0$ is the origin.

We shall prove that
\begin{eqnarray}\label{eq444.2}
\mu\big(\{x\in Q:|f(x)-A_{Q}f(x)|>\alpha\}\big)\leq c_1e^{-c_2\alpha}\mu(Q).
\end{eqnarray}
From here and to the end of the proof, we will use
$A_{Q}f$ to denote $A_{t_Q}$, where $t_Q =\ell(Q)$ and $\ell(Q)$ is the sidelength of $Q$.

Fix a constant $\beta>1$ to be chosen later. We apply Calder\'on--Zygmund decomposition to the function $f-A_{Q_0}f$ inside the cube $Q$.
We introduce the following selection criterion for a cube $R$:
\begin{eqnarray}\label{condition1}
\frac{1}{\mu(2R)}\int_{2R}|f(y)-A_Qf(y)| {\,} d\mu(y)>\beta,
\end{eqnarray}
where $2R$ denotes the ball with the same center as the center of cube $R$ and with radius of $2$ times of sidelength of $R$.

Set $Q_0=Q$ and subdivide $Q_0$ into its next level dyadic cubes. Denote $Q_1$ the only child of $Q_0$ such that $Q_1\in \mathcal{I}_1$. For any other child of $Q_0$, we denote it by $R_1$ and it is clear that it belongs to $\mathcal{I}_2$.
It follows from the definition of $\bmo$,
$$
\frac{1}{\mu(2R_1)}\int_{2R_1}|f(y)-A_Qf(y)| {\,} d\mu(y)= \frac{1}{\mu(2R_1)}\int_{2R_1}|f(y)-A_{2R_1}f(y)| {\,} d\mu(y)\leq \|f\|_{\bmo}=1<\beta,
$$
which means that the cube $R_1$ does not satisfy the selection criterion~\eqref{condition1}.
 Now divide all $R_1\in \mathcal{I}_2$. Select such a subcube $R$ if it satisfies the selection criterion~\eqref{condition1}.    Now subdivide all non-selected cubes into the next level dyadic cubes and select among these subcubes those that satisfy \eqref{condition1}. Continue this process indefinitely. We obtain a countable collection of cubes $\{Q_{j^{(1)'}}^{(1)'}\}_{j^{(1)'}}$. Select $Q_1$ no matter it satisfies~\eqref{condition1} or not. Then we obtain a collection of cubes $\{Q_j^{(1)}\}_j=\{Q_{j^{(1)'}}^{(1)'}\}_{j^{(1)'}}\cup Q_1$. Among all these selected cubes $\{Q_{j^{(1)'}}^{(1)}\}_{j^{(1)'}}$,
 $Q_1$ is the only one cube that belongs to $\mathcal{I}_1$. For this countable collection $\{Q_{j^{(1)}}^{(1)}\}_{j^{(1)}}$, we claim that:

\smallskip
(A-1) \ \ \ except $Q_1$, that is , for all $Q_{j^{(1)'}}^{(1)'}$, we have $\beta<\frac{1}{\mu(2Q_{j^{(1)'}}^{(1)'})}
\int_{2Q_{j^{(1)'}}^{(1)'}}|f(y)-A_Qf(y)| {\,} d\mu(y)\leq 2^m \beta$;

(B-1)  \ \ \  $|A_{Q_{j^{(1)}}^{(1)}}f(x)-A_{Q_0}f(x)|\leq C2^m\beta$\ for all $Q_{j^{(1)}}^{(1)}$ and $x\in Q_{j^{(1)}}^{(1)}$,
where the constant $C$ depends only on the dimensions $m$, $n$ and on the constant appeared in the upper bound of the Poisson
kernel as in Proposition~\ref{prop2.1};

(C-1) \ \ \   $\mu(Q_1)\leq \frac{1}{2^m}\mu(Q)$; $\sum_{j^{(1)}} \mu(Q_{j^{(1)'}}^{(1)'})\leq \frac{2^m-1}{2^m}\mu(Q)$;

(D-1)  \ \ \  $|f(x)-A_{Q_0}f(x)|\leq \beta$ for $x$ on the set $Q_0\backslash\cup_j Q_{j^{(1)}}^{(1)}$.

\smallskip
Proof of (A-1): The criterion criterion~\eqref{condition1} gives that $\beta<\frac{1}{\mu(2Q_{j^{(1)'}}^{(1)'})}\int_{2Q_{j^{(1)'}}^{(1)'}}|f(y)-A_Qf(y)| {\,} d\mu(y)$. To show the upper bound for this integration,
we denote by $\widetilde{Q_{j^{(1)'}}^{(1)'}}$  the father of $Q_{j^{(1)'}}^{(1)'}$. From our selection of cubes, we know that
$$
\frac{1}{\mu(2Q_{j^{(1)'}}^{(1)'})}\int_{2Q_{j^{(1)'}}^{(1)'}}|f(y)-A_Qf(y)| {\,} d\mu(y)\leq \frac{\mu(2\widetilde{Q_{j^{(1)'}}^{(1)'}})}{\mu(2Q_{j^{(1)'}}^{(1)'})}\frac{1}{\mu(2\widetilde{Q_{j^{(1)'}}^{(1)'}})}\int_{2\widetilde{Q_{j^{(1)'}}^{(1)'}}}|f(y)-A_Qf(y)| {\,} d\mu(y)
<2^m\beta.
$$

Proof of (B-1): For $Q_1$, by~\eqref{BMO-def20},
\begin{eqnarray*}
|A_{Q_1}f(x)-A_{Q_0}f(x)|\leq C\|f\|_{\bmo}\leq C2^m\beta.
\end{eqnarray*}
For every $Q_{j^{(1)}}^{(1)}$ which is not $Q_1$, by~\eqref{BMO-def20},
\begin{align*}
|A_{Q_{j^{(1)}}^{(1)}}f(x)-A_{Q_0}f(x)|&\leq |A_{Q_{j^{(1)}}^{(1)}}(f-A_{Q_0}f)(x)|+|A_{Q_{j^{(1)}}^{(1)}+Q_0}f(x)-A_{Q_0}f(x)|\\
&\leq |A_{Q_{j^{(1)}}^{(1)}}(f-A_{Q_0}f)(x)|+C\|f\|_{\bmo}.
\end{align*}
Then we estimate $|A_{Q_{j^{(1)}}^{(1)}}(f-A_{Q_0}f)(x)|$ with $x\in Q_{j^{(1)}}^{(1)}$. Denote by $\widetilde {Q_{j^{(1)}}^{(1)}}^{k}$  the $k$th ancestor of $Q_{j^{(1)}}^{(1)}$ and let $K_0$ be the number such that $\widetilde {Q_{j^{(1)}}^{(1)}}^{K_0}$ is one of the children of $Q_0$.
We write
\begin{align*}
&|A_{Q_{j^{(1)}}^{(1)}}(f-A_{Q_0}f)(x)|\\
&\leq \int_{2Q_{j^{(1)}}^{(1)}} p_{Q_{j^{(1)}}^{(1)}}(x,y)|f(y)-A_{Q_0}f(y)| {\,} d\mu(y)\\
&\quad+\int_{2\widetilde {Q_{j^{(1)}}^{(1)}}^{K_0}\backslash 2Q_{j^{(1)}}^{(1)}} p_{Q_j^{(1)}}(x,y)|f(y)-A_{Q_0}f(y)| {\,} d\mu(y)+\int_{M\backslash 2{\widetilde {Q_{j^{(1)}}^{(1)}}^{K_0}}} p_{Q_{j^{(1)}}^{(1)}}(x,y)|f(y)-A_{Q_0}f(y)| {\,} d\mu(y)\\
&=:B_{11}+B_{12}+B_{13}.
\end{align*}
We first note that $x\in Q_{j^{(1)}}^{(1)}$ implies that $|x|\geq\ell(Q_{j^{(1)}}^{(1)})$.

For $x,y\in \R^m$, from Proposition~\ref{prop2.1} we have the upper bound for the Poisson kernel as follows.
$$
p_{t}(x,y)\leq \frac{C}{t^m}(1+\frac{d(x,y)}{t})^{-m-1}+\frac{C}{t^n|x|^{m-2}|y|^{m-2}}(1+\frac{d(x,y)}{t})^{-n-1}.
$$
For the term $B_{11}$,  from the pointwise upper bound of the Poisson kernel $p_{Q_{j^{(1)}}^{(1)}}(x,y)$, the fact  $|x|>\ell(Q_{j^{(1)}}^{(1)})$ and from (A-1), we have
\begin{align*}
B_{11}&\leq \frac{1}{\ell(Q_{j^{(1)}}^{(1)})^m} \int_{2Q_{j^{(1)}}^{(1)}} |f(y)-A_{Q_0}f(y)| {\,} d\mu(y)\\
&\leq \frac{2^m}{\mu(2Q_{j^{(1)}}^{(1)})} \int_{2Q_{j^{(1)}}^{(1)}} |f(y)-A_{Q_0}f(y)| {\,} d\mu(y)\\
&\leq2^{2m} \beta.
\end{align*}
For the term $B_{12}$,  consider the chain of the dyadic cubes $\Big\{\widetilde {Q_{j^{(1)}}^{(1)}}^{k} \Big\}_{k=1}^{K_0}$ subject to the partial order via inclusion ``$\subset$'',  with the initial
dyadic cube $\widetilde {Q_{j^{(1)}}^{(1)}}^{1} $ which is the father of  $Q_{j^{(1)}}^{(1)}$  and the terminal one $\widetilde {Q_{j^{(1)}}^{(1)}}^{K_0} $ which is a child of $Q_0$ but different from $Q_1$, i.e., none of the corners of $\widetilde {Q_{j^{(1)}}^{(1)}}^{K_0} $ is the origin, which further gives that
$2\widetilde {Q_{j^{(1)}}^{(1)}}^{K_0}$ can not hit $K$. As a consequence, for any $x\in Q_{j^{(1)}}^{(1)} \subset \widetilde {Q_{j^{(1)}}^{(1)}}^{K_0}$ and for any $y\in \widetilde {Q_{j^{(1)}}^{(1)}}^{K_0} $, we have that
$$
|x|>\ell\Big(\widetilde {Q_{j^{(1)}}^{(1)}}^{K_0}\Big) \geq d(x,y).
$$

Then from the pointwise upper bound of the Poisson kernel $p_{Q_{j^{(1)}}^{(1)}}(x,y)$ we get
\begin{align*}
B_{12}&\leq \sum_{k=1}^{K_0}\frac{1}{\ell(Q_{j^{(1)}}^{(1)})^m}
2^{-k(m+1)}\int_{2\widetilde {Q_{j^{(1)}}^{(1)}}^{k}\backslash 2\widetilde {Q_{j^{(1)}}^{(1)}}^{k-1}} |f(y)-A_{Q_0}f(y)| {\,} d\mu(y)\\
&\quad+\sum_{k\geq 0, 2^k\ell(Q_{j^{(1)}}^{(1)})\leq 1}\frac{1}{\ell(Q_{j^{(1)}}^{(1)})^n }
2^{-k(n+1)}\int_{2\widetilde {Q_{j^{(1)}}^{(1)}}^{k}\backslash 2\widetilde {Q_{j^{(1)}}^{(1)}}^{k-1}}{1\over |x|^{m-2}} |f(y)-A_{Q_0}f(y)| {\,} d\mu(y)\\
&\quad+\sum_{k\geq 0, 2^k\ell(Q_{j^{(1)}}^{(1)})> 1}\frac{1}{\ell(Q_{j^{(1)}}^{(1)})^n }
2^{-k(n+1)}\int_{2\widetilde {Q_{j^{(1)}}^{(1)}}^{k}\backslash 2\widetilde {Q_{j^{(1)}}^{(1)}}^{k-1}}{1\over d(x,y)^{m-2}} |f(y)-A_{Q_0}f(y)| {\,} d\mu(y)\\
&\leq \sum_{k=1}^{K_0}\frac{\mu(2\widetilde {Q_{j^{(1)}}^{(1)}}^{k})}{\ell(Q_{j^{(1)}}^{(1)})^m}
2^{-k(m+1)}\frac{1}{\mu(2\widetilde {Q_{j^{(1)}}^{(1)}}^{k})}\int_{2\widetilde {Q_{j^{(1)}}^{(1)}}^{k}} |f(y)-A_{Q_0}f(y)| {\,} d\mu(y)\\
&\quad+\sum_{k\geq 0, 2^k\ell(Q_{j^{(1)}}^{(1)})\leq 1}\frac{\mu(2\widetilde {Q_{j^{(1)}}^{(1)}}^{k})}{\ell(Q_{j^{(1)}}^{(1)})^n }
2^{-k(n+1)}\frac{1}{\mu(2\widetilde {Q_{j^{(1)}}^{(1)}}^{k})}\int_{2\widetilde {Q_{j^{(1)}}^{(1)}}^{k}} |f(y)-A_{Q_0}f(y)| {\,} d\mu(y)\\
&\quad+\sum_{k\geq 0, 2^k\ell(Q_{j^{(1)}}^{(1)})>1}\frac{\mu(2\widetilde {Q_{j^{(1)}}^{(1)}}^{k})}{\ell(Q_{j^{(1)}}^{(1)})^n (2^k \ell(Q_{j^{(1)}}^{(1)}))^{m-2}}
2^{-k(n+1)}\frac{1}{\mu(2\widetilde {Q_{j^{(1)}}^{(1)}}^{k})}\int_{2\widetilde {Q_{j^{(1)}}^{(1)}}^{k}} |f(y)-A_{Q_0}f(y)| {\,} d\mu(y),
\end{align*}
where in the second inequality we use the fact that $d(x,y)\sim 2^k\ell(Q_{j^{(1)}}^{(1)})$ for any $x\in Q_{j^{(1)}}^{(1)} $ and for $y\in 2\widetilde {Q_{j^{(1)}}^{(1)}}^{k}\backslash 2\widetilde {Q_{j^{(1)}}^{(1)}}^{k-1}$ for the third term and the fact $|x|\geq1$ for the second term. Then from (A-1), we further have
\begin{align*}
B_{12}&\leq  2^m\beta\sum_{k=1}^{K_0}\frac{  \big(2^k\ell(Q_{j^{(1)}}^{(1)})\big)^m}{\ell(Q_{j^{(1)}}^{(1)})^m}
2^{-k(m+1)}  + 2^m\beta\sum_{k\geq 0, 2^k\ell(Q_{j^{(1)}}^{(1)})\leq 1}(2^k\ell(Q_{j^{(1)}}^{(1)}))^{m-n}
2^{-k}\\
&\quad +2^m\beta\sum_{k\geq 0, 2^k \ell(Q_{j^{(1)}}^{(1)})>1}(2^k\ell(Q_{j^{(1)}}^{(1)}))^{2-n}
2^{-k}\\
&\leq C\beta.
\end{align*}
We now consider the term $B_{13}$.
Note that in this case $d(x,y)\geq 2^{K_0}\ell(Q_{j^{(1)}}^{(1)})=\ell(Q_0)$
for any $x\in Q_{j^{(1)}}^{(1)} $ and for $y\in M\backslash 2\widetilde {Q_{j^{(1)}}^{(1)}}^{K_0}$.
Then we obtain from Proposition~\ref{prop2.1} that for any $x\in Q_{j^{(1)}}^{(1)} $:

1) for $y\in (\R^m\backslash K)\backslash 2\widetilde {Q_{j^{(1)}}^{(1)}}^{K_0}$,
\begin{align*}
p_{Q_{j^{(1)}}^{(1)}}(x,y)&\leq \frac{C}{\ell(Q_{j^{(1)}}^{(1)})^m}\Bigg(1+\frac{d(x,y)}{\ell(Q_{j^{(1)}}^{(1)})}\Bigg)^{-m-1}+\frac{C}{\ell(Q_{j^{(1)}}^{(1)})^n|x|^{m-2}|y|^{m-2}}\Bigg(1+\frac{d(x,y)}{\ell(Q_{j^{(1)}}^{(1)})}\Bigg)^{-n-1}\\
&\leq \frac{C}{\ell(Q_0)^m}\Bigg(1+\frac{d(x,y)}{\ell(Q_0)}\Bigg)^{-m-1}+\frac{C}{\ell(Q_0)^n|x|^{m-2}|y|^{m-2}}\Bigg(1+\frac{d(x,y)}{\ell(Q_0)}\Bigg)^{-n-1};
\end{align*}

2) for $y\in \Rf^n\backslash K$ 
\begin{align*}
p_{Q_{j^{(1)}}^{(1)}}(x,y)&\leq \frac{C}{\ell(Q_{j^{(1)}}^{(1)})^m}\Bigg(1+\frac{d(x,y)}{\ell(Q_{j^{(1)}}^{(1)})}\Bigg)^{-m-1}+\frac{C}{\ell(Q_{j^{(1)}}^{(1)})^n|x|^{m-2}}\Bigg(1+\frac{d(x,y)}{\ell(Q_{j^{(1)}}^{(1)})}\Bigg)^{-n-1}\\
&\quad+\frac{C}{\ell(Q_{j^{(1)}}^{(1)})^m|y|^{n-2}}\Bigg(1+\frac{d(x,y)}{\ell(Q_{j^{(1)}}^{(1)})}\Bigg)^{-m-1}\\
&\leq \frac{C}{\ell(Q_0)^m}\Bigg(1+\frac{d(x,y)}{\ell(Q_0)}\Bigg)^{-m-1}+\frac{C}{\ell(Q_0)^n|x|^{m-2}}\Bigg(1+\frac{d(x,y)}{\ell(Q_0)}\Bigg)^{-n-1}\\
&\quad+\frac{C}{\ell(Q_0)^m|y|^{n-2}}\Bigg(1+\frac{d(x,y)}{\ell(Q_0)}\Bigg)^{-m-1};
\end{align*}

3) for $y\in K$
\begin{align*}
p_{Q_{j^{(1)}}^{(1)}}(x,y)&\leq \frac{C}{\ell(Q_{j^{(1)}}^{(1)})^m}\Bigg(1+\frac{d(x,y)}{\ell(Q_{j^{(1)}}^{(1)})}\Bigg)^{-m-1}+\frac{C}{\ell(Q_{j^{(1)}}^{(1)})^n|x|^{m-2}}\Bigg(1+\frac{d(x,y)}{\ell(Q_{j^{(1)}}^{(1)})}\Bigg)^{-n-1}\\
&\leq \frac{C}{\ell(Q_0)^m}\Bigg(1+\frac{d(x,y)}{\ell(Q_0)}\Bigg)^{-m-1}+\frac{C}{\ell(Q_0)^n|x|^{m-2}}\Bigg(1+\frac{d(x,y)}{\ell(Q_0)}\Bigg)^{-n-1}.
\end{align*}

From the above estimates 1), 2) and 3) we see that for any $x\in Q_{j^{(1)}}^{(1)} $ and for any $y\in M\backslash 2\widetilde {Q_{j^{(1)}}^{(1)}}^{K_0}$,
 the upper bound of
$ p_{Q_{j^{(1)}}^{(1)}}(x,y)$ is controlled by the upper bound of
$ p_{Q_0}(x,y)$ pointwise. We use $\overline{p_{Q_0}(x,y)}$ to denote the
upper bound of  $ p_{Q_0}(x,y)$, which is as in the right-hand side of 1), 2) and 3) in each case.

Then, to estimate the term $B_{13}$, it suffices to estimate the term
\begin{align}\label{B13 upper bound}
 \int_{M}  \overline{p_{Q_0}(x,y)}|f(y)-A_{Q_0}f(y)| {\,} d\mu(y).
\end{align}
In fact, we point out that in the second half of the proof of Theorem \ref{main0}, we have already obtain such estimates. To be more specific, following the same estimates for the right-hand side of \eqref{AtAs integral} and the result in Theorem \ref{main0} about the equivalent norms of the two versions of BMO spaces, we obtain that
the integral in \eqref{B13 upper bound} is bounded by $C\|f\|_{\bmo}$, which gives that
$B_{13}\leq C \|f\|_{\bmo}$.

Combining the estimates of the terms $B_{11}$, $B_{12}$ and $B_{13}$, we obtain that
(B-1) holds.


Proof of (C-1): This is obvious.

Proof of (D-1): for $x\in Q_0\backslash \cup_{j^{(1)}} Q_{j^{(1)}}^{(1)}$, we have that
$$
|f(x)-A_{Q_0}f(x)|\leq {\varlimsup_{|R|\to 0,x\in 2R}}\frac{1}{\mu(2R)}\int_{2R}|f(y)-A_{Q_0}f(y)| {\,} d\mu(y)\leq \beta.
$$

For the next step, we divide it to two cases:

We now fix a selected first-generation cube $Q_{j^{(1)}}^{(1)}=Q_1$. Recall that $Q_1$ is the only one cube in the first generation that belongs to $\mathcal{I}_1$, i.e., one corner of $Q_1$ is origin.  In this case, just repeat the above argument, we get a sequence of
dyadic cubes $\{Q_{j^{(2)'}}^{(2)'}\}_{j^{(2)'}}\cup Q_2$, where  $Q_2$ is the only one child  of $Q_1$ such that one corner of $Q_2$ is origin. That is,
$Q_2$ belongs to $\mathcal{I}_1$. And we have $\{Q_{j^{(2)'}}^{(2)'}\}_{j^{(2)'}} \subset \mathcal I_2$.
Then we have the following results:

(A-2)  \ \ \  for all $Q_{j^{(2)'}}^{(2)'}$, we have $\beta<\frac{1}{\mu(2Q_{j^{(2)'}}^{(2)'})}
\int_{2Q_{j^{(2)'}}^{(2)'}}|f-A_{Q_1}f|(y)d\mu(y)\leq 2^m \beta$;

(B-2) \ \ \  $|A_{Q_{j^{(2)}}^{(2)}}f(x)-A_{Q_1}f(x)|\leq C2^m\beta$ for all $x\in Q_{j^{(2)}}^{(2)}\subset \{Q_{j^{(2)'}}^{(2)'}\}_{j^{(2)'}}\cup Q_2$, where the constant $C$ depends only on the dimensions $m$, $n$ and on the constant appeared in the upper bound of the Poisson kernel as in Proposition~\ref{prop2.1};

(C-2) \ \ \  $\mu(Q_2)\leq \frac{1}{2^m}\mu(Q_1)$;\quad $\sum_{j^{(2)'}} \mu(Q_{j^{(2)'}}^{(2)'})\leq \frac{2^m-1}{2^m}\mu(Q_1)$;

(D-2) \ \ \  $|f(x)-A_{Q_1}f(x)|\leq \beta$ for $x$ on the set $Q_1\backslash(\cup_{j^{(2)'}} Q_{j^{(2)'}}^{(2)'}\cup Q_2)$.

We now fix a selected first-generation cube $Q_{j^{(1)'}}^{(1)'}$, which is not $Q_1$. This means that $Q_{j^{(1)'}}^{(1)'}\in \mathcal{I}_2$
and from the construction of dyadic cubes, we have $\mbox{dist} \,(K, Q_{j^{(1)'}}^{(1)'})\sim \ell(Q_{j^{(1)'}}^{(1)'})$. Then
define $f_0=(f-A_{Q_{j^{(1)'}}^{(1)'}}f)\chi_{2Q_{j^{(1)'}}^{(1)'}}$.
We apply the Calder\'on--Zygmund decomposition to the function $f_0$ inside the cube $Q_{j^{(1)'}}^{(1)'}$.
We introduce the following selection criterion for a cube $R$:
\begin{eqnarray}\label{condition2}
\frac{1}{\mu(2R)}\int_{2R}|f_0(y)| {\,} d\mu(y)>\beta.
\end{eqnarray}

It follows from~\eqref{BMO-def10} and \eqref{BMO-def20} in the definition of $\bmo$,  that
\begin{align*}
&\frac{1}{\mu(2Q_{j^{(1)'}}^{(1)'})}\int_{2Q_{j^{(1)'}}^{(1)'}}|f(y)-A_{Q_{j^{(1)'}}^{(1)'}}f(y)| {\,} d\mu(y)\\
&\leq \frac{1}{\mu(2Q_{j^{(1)'}}^{(1)'})}\int_{2Q_{j^{(1)'}}^{(1)'}}|f(y)-A_{2Q_{j^{(1)'}}^{(1)'}}f(y)| {\,} d\mu(y)+ \sup_{y\in 2Q_{j^{(1)'}}^{(1)'}} |A_{2Q_{j^{(1)'}}^{(1)'}}f(y)-A_{Q_{j^{(1)'}}^{(1)'}}f(y)|\\
&\leq (1+C)\|f\|_{\bmo}\\
&<\beta.
\end{align*}
Thus, the cube $Q_{j^{(1)'}}^{(1)'}$ does not satisfy the selection criterion~\eqref{condition2}. Subdivide $Q_{j^{(1)'}}^{(1)'}$ into dyadic cubes in the next level. Select such a subcube $R$ if it satisfies the selection criterion~\eqref{condition2}. Now subdivide all non-selected cubes into their next level dyadic cubes and select among these subcubes those that satisfy \eqref{condition2}. Continue this process indefinitely. We obtain a countable collection of cubes $\{Q_{j^{(2)''}}^{(2)''}\}_{j^{(2)''}}$. Among these $\{Q_{j^{(2)''}}^{(2)''}\}_{j^{(2)''}}$, none of them are in
$ \mathcal{I}_1$. Then for $Q_{j^{(2)''}}^{(2)''}$ we have

\smallskip

(A-2') \ \ \  $\beta<\frac{1}{\mu(2Q_{j^{(2)''}}^{(2)''})}\int_{2Q_{j^{(2)''}}^{(2)''}}|f-A_{Q_{j^{(1)'}}^{(1)'}}f|(y)d\mu(y)\leq 2^m \beta$;

(B-2')  \ \ \  $|A_{Q_{j^{(2)''}}^{(2)''}}f(x)-A_{Q_{j^{(1)'}}^{(1)'}}f(x)|\leq C2^m\beta$\quad for all $x\in Q_j^{(2)''}$,
where the constant $C$ depends only on the dimensions $m$, $n$ and on the constant appeared in the upper bound of the Poisson
kernel as in Proposition~\ref{prop2.1};

(C-2')  \ \ \  $\sum_{j^{(2)''}} \mu(Q_{j^{(2)''}}^{(2)''})\leq C\mu(Q_{j^{(1)'}}^{(1)'})/\beta$;

(D-2') \ \ \  $|f(x)-A_{Q_{j^{(1)'}}^{(1)'}}f(x)|\leq \beta$ for $x$ on the set $Q_{Q_{j^{(1)'}}^{(1)'}}\backslash\cup_{j^{(2)''}} Q_{j^{(2)''}}^{(2)''}$.

Proof of (A-2'): From criterion~\eqref{condition2} , it is clear that $\beta<\frac{1}{\mu(2Q_{j^{(2)''}}^{(2)''})}
\int_{2Q_{j^{(2)''}}^{(2)''}}|f(y)-A_{Q_{j^{(1)'}}^{(1)'}}f(y)| {\,} d\mu(y)$.
By the similar argument in the proof of (A-1), we have
$\frac{1}{\mu(2Q_{j^{(2)''}}^{(2)''})}\int_{2Q_{j^{(2)''}}^{(2)''}}|f(y)-A_{Q_{j^{(1)'}}^{(1)'}}f(y)| {\,} d\mu(y)\leq 2^m \beta$.

Proof of (B-2'): The proof is similar to that of (B-1).

Proof of (C-2'): Note that for $x\in Q_{j^{(2)''}}^{(2)''}$,
$$
\beta<\frac{1}{\mu(2Q_{j^{(2)''}}^{(2)''})}\int_{2Q_{j^{(2)''}}^{(2)''}}|f(y)-A_{Q_{j^{(1)'}}^{(1)'}}f(y)| {\,} d\mu(y)\leq \mathcal M f_0(x)
$$
which together with the weak type $(1,1)$ of the Hardy--Littlewood maximal operator $\mathcal M$ and $2Q_{j^{(1)'}}^{(1)'}\subset \R^m$ implies
\begin{eqnarray*}
\sum_{j^{(2)''}} \mu(Q_{j^{(2)''}}^{(2)''})\leq \mu\big(\{x\in Q_{j^{(1)'}}^{(1)'}: \mathcal M f_0(x)>\beta\}\big)\leq C\frac{\|f_0\|_{L^1(\R^m)}}{\beta}\leq
C \frac{\mu(2Q_{j^{(1)'}}^{(1)'})}{\beta}\leq C\frac{\mu(Q_{j^{(1)'}}^{(1)'})}{\beta}.
\end{eqnarray*}

Proof of (D-2'): It is similar to the argument of (D-1).

Now we select the second-generation cubes $$\big\{Q_{j^{(2)}}^{(2)}\big\}_{j^{(2)}}:=\big\{Q_{j^{(2)''}}^{(2)''}\big\}_j\cup\{Q_{j^{(2)'}}^{(2)'}\}_j\cup Q_2. $$
It is clear that $\big\{Q_{j^{(2)}}^{(2)}\big\}_{j^{(2)}}$ includes three cases:

i)\ $Q_2$:\ this is the only one whose corner is at the origin;

ii) \ $Q_{j^{(2)'}}^{(2)'}$:\ arisen from decomposition of $Q_1$ but itself is not $Q_2$;

iii)\ $Q_{j^{(2)''}}^{(2)''}$:\ arisen  from decomposition of $Q_{j^{(1)'}}^{(1)'}$, which is not $Q_1$.

\smallskip
Then repeat the above argument we can get the third generation cubes
$$\{Q_{j^{(3)}}^{(3)}\}_{j^{(3)}}:=\{Q_{j^{(3)'''}}^{(3)'''}\}_{j^{(3)'''}}\cup\{Q_{j^{(3)''}}^{(3)''}\}_{j^{(3)''}}\cup\{Q_{j^{(3)'}}^{(3)'}\}_{j^{(3)'}}\cup Q_3,$$
in which

i) \ $Q_3$: this is the only one whose corner is at the origin;

ii) \ $Q_{j^{(3)'}}^{(3)'}$:\ arisen from the decomposition of $Q_2$ but itself is not $Q_3$;

iii) \ $Q_{j^{(3)''}}^{(3)''}$:\ arisen  from the decomposition of $Q_{j^{(2)'}}^{(2)'}$, which is not $Q_2$;

iv) \ $Q_{j^{(3)'''}}^{(3)'''}$:\ arisen  from the decomposition of $Q_{j^{(2)''}}^{(2)''}$.

\smallskip
\noindent
For $Q_{j^{(3)}}^{(3)}\in \{Q_{j^{(3)'}}^{(3)'}\}_{j^{(3)'}}\cup Q_3$, they will satisfy the following properties correspondingly:

(A-3)  \ \ \  except $Q_3$, that is , for all $Q_{j^{(3)'}}^{(3)'}$, we have $\beta<\frac{1}{\mu(2Q_{j^{(3)'}}^{(3)'})}
\int_{2Q_{j^{(3)'}}^{(3)'}}|f(y)-A_{Q_2}f(y)| {\,} d\mu(y)\leq 2^m \beta$;

(B-3)  \ \ \  $|A_{Q_{j^{(3)}}^{(3)}}f(x)-A_{Q_2}f(x)|\leq C2^m\beta$ for all $x\in Q_{j^{(3)}}^{(3)}\subset \{Q_{j^{(3)'}}^{(3)'}\}_j\cup Q_3$, where the constant $C$ depends only on the dimensions $m$, $n$ and on the constant appeared in the upper bound of the Poisson kernel as in Proposition~\ref{prop2.1};

(C-3)  \ \ \  $\mu(Q_3)\leq \frac{1}{2^m}\mu(Q_2)$; \quad $\sum_{j^{(3)'}} \mu(Q_{j^{(3)'}}^{(3)'})\leq \frac{2^m-1}{2^m}\mu(Q_2)$;

(D-3)  \ \ \  $|f(x)-A_{Q_2}f(x)|\leq \beta$ for $x$ in each of

\hskip1cm 1) on the set $Q_2\backslash(\cup_{j^{(3)'}} Q_{j^{(3)'}}^{(3)'}\cup Q_3)$,

\hskip1cm 2) on the set $Q_{j^{(3)}}^{(3)}\in \{Q_{j^{(3)'''}}^{(3)'''}\}_{j^{(3)'''}}\cup\{Q_{j^{(3)''}}^{(3)''}\}_{j^{(3)''}}$ and

\hskip1cm 3) on the set  $Q_{j^{(2)}}^{(2)}$ where $Q_{j^{(3)}}^{(3)}$ is located;

\noindent Moreover, we have

(A-3')  \ \ \  $\beta<\frac{1}{\mu(2Q_{j^{(3)}}^{(3)})}\int_{2Q_{j^{(3)}}^{(3)}}|f(y)-A_{Q_j^{(2)}}f(y)| {\,} d\mu(y)\leq 2^m \beta$;

(B-3')  \ \ \  $|A_{Q_{j^{(3)}}^{(3)}}f(x)-A_{Q_{j^{(2)}}^{(2)}}f(x)|\leq C2^m\beta$ for all $x\in Q_{j^{(3)}}^{(3)}$,
 where the constant $C$ depends only on the dimensions $m$, $n$ and on the constant appeared in the upper bound of the
 Poisson kernel as in Proposition~\ref{prop2.1};

(C-3')  \ \ \  $\sum_{j^{(3)}} \mu(Q_{j^{(3)}}^{(3)})\leq C\mu(Q_{j^{(2)}}^{(2)})/\beta$;

(D-3')  \ \ \  $|f(x)-A_{Q_{j^{(2)}}^{(2)}}f(x)|\leq \beta$ for $x$ on the set $Q_{j^{(2)}}^{(2)}\backslash\cup_{j^{(3)}} Q_{j^{(3)}}^{(3)}$.

\smallskip

For the measure relation, we show it in the following figure.
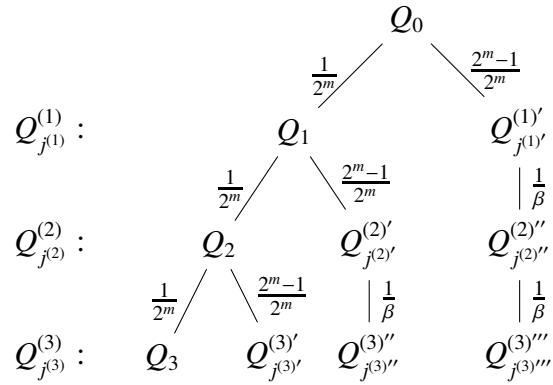
\begin{center}
\tikzstyle{level 1}=[level distance=1.5cm, sibling distance=3 cm]
\tikzstyle{level 2}=[level distance=1.5cm, sibling distance=2cm]
\tikzstyle{level 3}=[level distance=1.5cm, sibling distance=1.5cm]
\begin{tikzpicture} [level/.style={sibling distance=50mm/#1} ]
   \node   (z){$Q_0$}
  child  {node   (a) {$Q_1$}
          child {node (a1) {$Q_2$}
                  child {node (a3) {$Q_3$}
                       child [grow=left] {node (q) {$Q_{j^{(3)}}^{(3)}:$}
                             child [grow=up] {node (r) {$Q_{j^{(2)}}^{(2)}:$}
                                   child [grow=up] {node (s) {$Q_{j^{(1)}}^{(1)}:$} edge from parent[draw=none]}
                             edge from parent[draw=none]}
                        edge from parent[draw=none]}
                  edge from parent node [left] { $  \frac{1}{2^m}$} }
                  child {node (a4) {$Q_{j^{(3)'}}^{(3)'}$} edge from parent node[  right] { $  \frac{2^m-1}{2^m}$} }
          edge from parent node [  left] {$  \frac{1}{2^m}$} }
          child {node (a2) {$Q_{j^{(2)'}}^{(2)'}$}
                child {node (a5) {$Q_{j^{(3)''}}^{(3)''}$} edge from parent node [right] {$  \frac{1}{\beta}$} }
          edge from parent node[  right] {$  \frac{2^m-1}{2^m}$} }
     edge from parent node[ left] {$  \frac{1}{2^m}$}
  }
  child    {node   (b) {$Q_{j^{(1)'}}^{(1)'}$}
         child {node (a6) {$Q_{j^{(2)''}}^{(2)''}$}
              child {node (a7) {$Q_{j^{(3)'''}}^{(3)'''}$} edge from parent node [right] {$  \frac{1}{\beta}$} }
         edge from parent node [right] {$  \frac{1}{\beta}$} }
  edge from parent node[ right] {$ \frac{2^m-1}{2^m}$} };

\end{tikzpicture}
\end{center}
Choosing $\beta$ large enough and then
 summing all $\mu(Q_{j^{(3)}}^{(3)})$ gives
\begin{align*}
\sum_{j^{(3)}} \mu(Q_{j^{(3)}}^{(3)})&\leq \mu(Q_3) +\sum_{j^{(3)'}}\mu(Q_{j^{(3)'}}^{(3)'})+\sum_{j^{(3)''}}\mu(Q_{j^{(3)''}}^{(3)''})+
\sum_{j^{(3)'''}}\mu(Q_{j^{(3)'''}}^{(3)'''})\\
&\leq \mu(Q_2)+\sum_{j^{(2)'}}\frac{1}{\beta}\mu(Q_{j^{(2)'}}^{(2)'})+\sum_{j^{(2)''}}\frac{1}{\beta}\mu(Q_{j^{(2)''}}^{(2)''})\\
&\leq \frac{1}{2}\mu(Q_1)+\frac{1}{\beta}\mu(Q_1)+\sum_{j^{(1)'}}\frac{1}{\beta^2}\mu(Q_{j^{(1)'}}^{(1)'})\\
&\leq \frac{1}{2^2}\mu(Q_0)+\frac{1}{2\beta}\mu(Q_0)+\frac{1}{\beta^2}\mu(Q_0)\\
&\leq 3\mu(Q_0)/2^2.
\end{align*}

We iterate this procedure indefinitely to obtain a doubly indexed family of cubes $\{Q_{j^{(K)}}^{(K)}\}$
%
$$
\sum_{j^{(K)}} \mu(Q_{j^{(K)}}^{(K)})\leq K\mu(Q_0)/2^{K-1}
$$
and that
\begin{align}\label{f-AQf}
|f(x)-A_{Q_0}f(x)|\leq \bar{ C}K\beta, \,\mbox{for all }\, x\in Q_0\backslash \cup_{{j^{(K)}}}Q_{j^{(K)}}^{(K)},
\end{align}
where the constant $\bar C$ depends only on the dimensions $m$, $n$ and on the constant appeared in the upper bound of the Poisson kernel as in Proposition~\ref{prop2.1}.

In fact, to see the argument \eqref{f-AQf}, we point that for any fixed ${j^{(1)}}$, from (D-2) and (D-2'), we get
$|f(x) - A_{Q_{j^{(1)}}^{(1)}}f(x)|\leq \beta$ for every $x$ on $Q_{j^{(1)}}^{(1)}\backslash \cup_{j^{(2)}} Q_{j^{(2)}}^{(2)} $. Moreover, from (B-1), we get that $|A_{Q_{j^{(1)}} ^{(1)}}f(x) - A_{Q_0}f(x)| \leq C2^m\beta$ for every $x$ on $Q_{j^{(1)}}^{(1)}$.  This gives
\begin{align}\label{JN temp B-1}
|f(x)-A_{Q_0}f(x)|\leq \beta+C2^m\beta, \quad {\rm on \ }  Q_{j^{(1)}} ^{(1)}   \backslash (\cup_{j^{(2)}} Q_{j^{(2)}}^{(2)} ).
\end{align}
Combining the above inequality with (D-1), we obtain that
$$  |f(x)-A_{Q_0}f(x)|\leq   C2^m\cdot 2\beta \quad {\rm on \ }  Q_0  \backslash (\cup_{j^{(2)}} Q_{j^{(2)}}^{(2)} ).$$
Again, for any fixed $j^{(2)}$, from (D-3) and (D-3') we also have that
$|f(x) - A_{Q_{j^{(2)}}^{(2)}}f(x)|\leq \beta$ for every $x$ on $Q_{j^{(2)}}^{(2)}\backslash \cup_{j^{(3)}} Q_{j^{(3)}}^{(3)} $,
which combines with (B-3), (B-3') and (B-2), (B-2'), yields
$$  |f(x)-A_{Q_0}f(x)|\leq   C2^m\cdot 3\beta \quad {\rm on \ }  Q_{j^{(2)}}^{(2)}  \backslash (\cup_{j^{(3)}} Q_{j^{(3)}}^{(3)} ),
$$
which, combining (B-1) and \eqref{JN temp B-1}, gives that
$$  |f(x)-A_{Q_0}f(x)|\leq   C2^m\cdot 3\beta \quad {\rm on \ }  Q_0 \backslash (\cup_{j^{(3)}} Q_{j^{(3)}}^{(3)} ). $$
By induction, we obtain  that for all $K\geq 1$, \eqref{f-AQf} holds with $\bar C $ depending only on the dimensions $m$, $n$ and on the constant appeared in the upper bound of the Poisson kernel as in Proposition~\ref{prop2.1}.

Now for any fixed $\alpha>2$, if
$\bar C K\beta\leq \alpha< (K+1)\bar C\beta$ for some integer $K\geq1$,
then we have
\begin{align*}
\mu(\{x\in Q:|f(x)-A_{Q}f(x)|>\alpha\})&\leq\sum_{j^{(K)}} \mu(Q_{j^{(K)}}^{(K)})\leq K\mu(Q_0)/2^{K-1}= 4K\mu(Q_0)/2^{K+1}\\
&\leq 4K e^{-(K+1)\log 2}\mu(Q)\\
&\leq\frac{4\alpha}{\bar C\beta} e^{-\alpha \log 2/(\bar C\beta)}\mu(Q) \\
&=\frac{4\alpha}{\bar C\beta} e^{-\alpha \log 2/(2\bar C\beta)}\cdot e^{-\alpha \log 2/(2c\beta)}\mu(Q)\\
&\leq c_1 e^{-c_2\alpha}\mu(Q),
\end{align*}
where $$c_1 = \sup_{\alpha>0} \frac{4\alpha}{\bar C\beta} e^{-\alpha \log 2/(2\bar C\beta)}\quad{\rm\ and \ }\quad c_2 =\log 2/(2\bar C\beta) .$$
It is clear that $c_1$ and $c_2$ depend only the dimensions $m$ and $n$ and the constant $C$ in the upper bound of the Poisson kernel as in Proposition~\ref{prop2.1} and $c_2<1$.

For $Q'\in \Rf^n$, we also can prove that there exist two positive constants $c_1$ and $c_2$ (depending only the dimensions $m$ and $n$ and the constant $C$ in the upper bound of the Poisson kernel as in Proposition~\ref{prop2.1}, with $c_2<1$) such that
\begin{eqnarray}\label{eq4444.3}
\mu\big(\{x\in Q':|f(x)-A_{Q'}f(x)|>\alpha\}\big)\leq c_1e^{-c_2\alpha}\mu(Q'),
\end{eqnarray}
We point out that
the proof for \eqref{eq4444.3} is almost the same as that of \eqref{eq444.2} for $Q\in \R^m$ and we skip it here.

Combining all the estimates above, we obtain that both \eqref{eq444.2} and \eqref{eq4444.3} hold with $c_2<1$.

Based on these two auxiliary facts, we  now prove \eqref{eq44.2}.
To see this, for any ball $B$ in $\mathcal B_1$ with $|x_B|<2r_B$ and $r_B>1$,
we consider the covering of $B$ as in \eqref{covering of B}.
Then the estimates will be split  into the  following cases.

Case 1:  the center of $B$ is in $\R^m$.

It is clear that in this case,  we have $r_{B}\sim \ell(Q)$.

Subcase 1.1: $B\cap \Rf^n=\emptyset$.

In this subcase we have $B\subset Q$, which shows that $\mu(B)\approx \mu(Q)$.
Moreover, using the fact  $r_{B}\sim \ell(Q)$ and then
just repeating the proof for \eqref{eq444.2} by using $f-A_B f$ instead of $f-A_Q f$,
we have that
\begin{eqnarray}\label{eq444.2-1}
\mu\big(\{x\in Q:|f(x)-A_{B}f(x)|>\alpha\}\big)\leq c_1e^{-c_2\alpha}\mu(Q).
\end{eqnarray}

As a consequence, from \eqref{eq444.2-1} we further have
\begin{eqnarray*}
\mu\big(\{x\in B:|f(x)-A_{B}f(x)|>\alpha\}\big)\leq c_1e^{-c_2\alpha}\mu(B).
\end{eqnarray*}

Subcase 1.2:  $B\cap \Rf^n\neq\emptyset$.

In this subcase, we first note that  the part $B\cap \R^m$ is contained in $Q$, and then
following the argument in Subcase 1.1 above, we have that
\begin{eqnarray*}
\mu\big(\{x\in Q:|f(x)-A_{B}f(x)|>\alpha\}\big)\leq c_1e^{-c_2\alpha}\mu(B).
\end{eqnarray*}

Next, we choose larger $Q'$ such that $2B\cap \Rf^n\subset Q'$ and $\ell(Q')\sim r_B$. Thus in this case,
$\mu(Q')\leq C\mu(2B)\leq C\mu(B)$.  So just in the first step of the proof for $Q'$,
we use $f-A_B f$ instead of $f-A_{Q'} f$. It follows that
\begin{eqnarray*}
\mu\big(\{x\in Q':|f(x)-A_{B}f(x)|>\alpha\}\big)\leq c_1e^{-c_2\alpha}\mu(Q')\leq C_1e^{-c_2\alpha}\mu(B).
\end{eqnarray*}
Then combing these two subcases, we have
\begin{eqnarray*}
\mu\big(\{x\in B:|f(x)-A_{B}f(x)|>\alpha\}\big)\leq c_1e^{-c_2\alpha}\mu(B).
\end{eqnarray*}

Case 2:  the center of $B$ is in $\Rf^n$.

Note that in this case we have $\ell(Q')\sim r_B$. We first consider the corresponding cube $Q'$.
Again, just repeating the proof for \eqref{eq4444.3} by using $f-A_B f$ instead of $f-A_Q f$,
we have that
\begin{eqnarray}\label{eq Q'}
\mu\big(\{x\in Q':|f(x)-A_{B}f(x)|>\alpha\}\big)\leq c_1e^{-c_2\alpha}\mu(Q')\leq c_1e^{-c_2\alpha}\mu(B).
\end{eqnarray}

Subcase 2.1: $B\cap \R^m=\emptyset$.

In this subcase, we have that $B\subset Q'$. Hence,
\begin{eqnarray*}
\mu\big(\{x\in B:|f(x)-A_{B}f(x)|>\alpha\}\big)\leq c_1e^{-c_2\alpha}\mu(B).
\end{eqnarray*}

Subcase 2.2: $B\cap \R^m\not=\emptyset$.

From the definition of $\bmo$ and the assumption that $\|f\|_{\bmo}=1$, we have that
$$ |f(x)-A_{B}f(x)|\leq |f(x)-A_{Q}f(x)|+|A_Qf(x)-A_{B}f(x)|\leq |f(x)-A_{Q}f(x)|+\log (r_B/r_Q)+1, $$
which gives
\begin{align}\label{eq sub2.2}
\mu(\{x\in Q:|f(x)-A_{B}f(x)|>\alpha\})&\leq \mu(\{x\in Q:|f(x)-A_{Q}f(x)|>\alpha-\log (r_B/r_Q)-1\})\\
&\leq \mu(\{x\in Q:|f(x)-A_{Q}f(x)|>\alpha/2-\log (r_B/r_Q)\}),\nonumber
\end{align}
since at the very beginning we already point out that it suffices to consider $\alpha>2$.

If $\alpha\leq2\log (r_B/r_Q)$, then
\begin{align*}
\mu\big(\{x\in Q:|f(x)-A_{B}f(x)|>\alpha\}\big) &\leq \mu(Q) = Cr_Q r_Q^{m-1}\leq C{r_Q\over r_B}\mu(B)\\
&\leq Ce^{-\log(r_B/r_Q)}\mu(B)\\
&\leq Ce^{-\alpha/2}\mu(B),
\end{align*}
where the constant $C$ depends only on the dimensions $m$ and $n$.

If $\alpha>2\log (r_B/r_Q)$, then by applying \eqref{eq444.2} we obtain that
\begin{align*}
\mu\big(\{x\in Q:|f(x)-A_{Q}f(x)|>\alpha/2-\log (r_B/r_Q)\}\big)&\leq c_1 e^{ -c_2\alpha/2+c_2\log (r_B/r_Q)}\mu(Q)\\
&\leq c_1 e^{ -c_2\alpha/2+\log (r_B/r_Q)}\mu(Q)\\
&\leq C \Big({r_B\over r_Q}\Big) \, r_Q^m\, e^{ -c_2\alpha/2}\\
&\leq C r_B r_Q^{m-1}\, e^{ -c_2\alpha/2}\\
&\leq C \mu(B) e^{ -c_2\alpha/2},
\end{align*}
where in the second inequality we use the fact that $c_2<1$ and in the last inequality we use the fact that
$r_Q^{m-1}\leq r_B^{n-1}$.

Then combining this estimate and \eqref{eq Q'}, we have that
\begin{eqnarray*}
\mu\big(\{x\in B:|f(x)-A_{B}f(x)|>\alpha\}\big)\leq c_1e^{-c_2\alpha}\mu(B),
\end{eqnarray*}
which implies that \eqref{eq44.2} holds.

We now prove \eqref{eq44.223}.
For any $B\in \mathcal B_0$,
we now consider again the covering of $B$ as in \eqref{covering of B}.
Then it is clear that from \eqref{eq Q'} we first have
\begin{eqnarray*}
\mu\big(\{x\in Q':|f(x)-A_{B}f(x)|>\alpha\}\big)\leq c_1e^{-c_2\alpha}\mu(B).
\end{eqnarray*}

Next we consider $Q$, which is one part of the covering of $B$ as in \eqref{covering of B}.
Now recall that from the definition of  $ \mathcal B_0$, it is clear that  we have
$x_B\in \Rf^n, r_B\geq 2, K\subset B $ and $  r_B^{\frac{n-1}{m-1}}<r_{B^m}<r_B$.
This gives that
\begin{align}\label{rB over rQ}
 \log\Big( {r_B\over r_Q} \Big) <   \log\Big( {r_Q^{m-1\over n-1}\over r_Q} \Big) =
{m-1\over n-1} \log r_Q \leq {m-1\over n-1} \log r_B.
\end{align}

By using similar estimate as in
Subcase 2.2 above, we have that:

If $\alpha\leq 2{m-1\over n-1} \log r_B$, then we have
\begin{align*}
\mu  \big(\{x\in Q:|f(x)-A_{B}f(x)|>\alpha\}\big)
&\leq \mu(Q)\leq e^{ 2{m-1\over n-1} } e^{-2{m-1\over n-1}} \mu(Q)\\
&= e^{ 2{m-1\over n-1} } e^{ -2{m-1\over n-1}{\log r_B\over \log r_B} }\mu(Q)\\
&\leq C e^{ -{\alpha\over \log r_B} }\mu(B).
\end{align*}

If $\alpha> 2{m-1\over n-1} \log r_B$, then similar to \eqref{eq sub2.2}  we have
\begin{align}\label{eq sub2.2 for B0}
\mu \big(\{x\in Q:|f(x)-A_{B}f(x)|>\alpha\}\big)
\leq \mu \big(\{x\in Q:|f(x)-A_{Q}f(x)|>\alpha/2-\log (r_B/r_Q)\}\big).
\end{align}
Hence, from \eqref{rB over rQ} we further have
\begin{align*}
\mu \big(\{x\in Q:|f(x)-A_{B}f(x)|>\alpha\}\big)
&\leq \mu \Big(\Big\{x\in Q:|f(x)-A_{Q}f(x)|>{\alpha\over2}-{m-1\over n-1} \log r_B\Big\}\Big) \nonumber\\
&\leq \mu \Big(\Big\{x\in Q:|f(x)-A_{Q}f(x)|>{\alpha\over 2{m-1\over n-1} \cdot{\log r_B \over \log 2}}- \log2 \Big\}\Big) \nonumber\\
&\leq  c_1 \mu(Q) e^{-c_2\bigg({\alpha\over 2{m-1\over n-1}\cdot {\log r_B \over \log 2}}- \log2\bigg)}\\
&\leq c_12^{c_2}\mu(B) e^{-\tilde{c}_2 {\alpha\over \log r_B }}, \nonumber
\end{align*}
where $\tilde c_2 = {c_2 \over 2{m-1\over n-1}\cdot {1\over \log 2}}$,
the second inequality follows from the basic fact that
$$  {m-1\over n-1}\cdot {\log r_B \over \log 2}  \left({\alpha\over2}-{m-1\over n-1} \log r_B\right) >  {\alpha\over2}-{m-1\over n-1} \log r_B
$$ since
$$\alpha> 2{m-1\over n-1} \log r_B\ \ \ {\rm and}\ \ \ {m-1\over n-1}\cdot {\log r_B \over \log 2}>1,
$$
the third inequality follows from \eqref{eq444.2}.

Combining all the estimates above, we have showed \eqref{eq44.223}.

\medskip

We now turn to the second case.

{\bf Case II:}\ $|x_B|>2r_B$ or $r_B\leq1$.

\medskip

Actually, the estimates in this case are easier and direct, and are very similar to the proof of
the John--Nirenberg  inequality for classical BMO space (see for example the proof in \cite{Gra}).

The main reason is that there is no cube in this case that is non-doubling. Moreover,
for  every dyadic cube in this case, when we  decompose it into the next levels, all the subcubes are always similar to the
original one and belong to $\mathcal{I}_2$.  In other words, there is no dyadic cubes arisen from decompositions belong to
$\mathcal{I}_1$ and hence we do not have the chain of $Q_1,Q_2,\ldots$ as arisen from the estimates of \eqref{eq444.2}.

\medskip

Combining {\bf Cases I} and {\bf II}, we get that the proof of of  Theorem \ref{main1} is complete.

\end{proof}

\section{Interpolation between $L^p(M)$ and $\bmo$}
\setcounter{equation}{0}

Recall that the sharp maximal function $\mathcal M^\sharp $ and the non-tangential maximal function $\N_L $
are given \eqref{sharp} and \eqref{maximal}, respectively.
Let us prove the following result.

\begin{lemma}\label{lemma5.1}
There exits small enough $\gamma>0$ and large enough $K>0$ such that for all $\lambda>0$  and all locally integrable functions $f$, we have
\begin{eqnarray}\label{goodlambda3}
 \mu\big(\{  x\in M: |f(x)|>K\lambda,\ \mathcal M^\sharp f(x) \leq \gamma\lambda \}\big)
 \leq C\gamma\mu\big(\{x\in M: \N_L  f(x)>\lambda\}\big)\nonumber
\end{eqnarray}
\end{lemma}

\begin{proof}
To show this lemma,
it suffices to prove that there exit small enough $\gamma>0$ and large enough $N>0$ such that for all $\lambda>0$  and all locally integrable functions $f$, we have the following estimates:
\begin{eqnarray}\label{goodlambda1}
 \mu(\{  x\in \R^m: |f(x)|>N\lambda,\ \mathcal M^\sharp f(x) \leq \gamma\lambda \})
 \leq C\gamma\mu(\{x\in M: \mathcal N_L f(x)>\lambda\})
\end{eqnarray}
and
\begin{eqnarray}\label{goodlambda2}
 \mu(\{  x\in \Rf^n: |f(x)|>N\lambda,\ \mathcal M^\sharp f(x) \leq \gamma\lambda \})
 \leq C\gamma\mu(\{x\in \Rf^n: \mathcal N_L f(x)>\lambda\}).
\end{eqnarray}
Then combining \eqref{goodlambda1} and \eqref{goodlambda2} we can obtain \eqref{goodlambda3}, which finishes the proof.

We first sketch the proof for \eqref{goodlambda2} since it is similar to the proof for classical good-$\lambda$ inequality (see for example \cite{Gra}, see also \cite[Lemma 5.3]{DY1}).

To begin with, we assume that $\{x\in \Rf^n: \mathcal N_L f(x)>\lambda\}$ has finite measure, otherwise there is nothing to prove. Then we
set $E_{\lambda}:=\{  x\in \Rf^n: |f(x)|>N\lambda,\ \mathcal M^\sharp f(x) \leq \gamma\lambda \}$, $\Omega_\lambda:=\{x\in \Rf^n: \widetilde{\mathcal N}_L f(x)>\lambda\}$,
where
$$\widetilde{\mathcal N}_L f(x)=\sup_{\substack{Q: {\rm\ dyadic\ cubes\ in} \Rf^n\\ Q\ni x}} \sup_{y\in Q} \big|e^{-r_Q\sqrt{L}}f(y)\big|.$$
Note that for $x\in \Rf^n$,  $|f(x)|\leq \widetilde{\mathcal N}_L f(x)\leq \mathcal N_L f(x)$. Hence, to prove  \eqref{goodlambda2} it suffices to prove
\begin{align}\label{goodlambda2'}
\mu(E_\lambda) \leq C\gamma\mu(\Omega_\lambda).
\end{align}
To see this, we now decompose $\Omega_\lambda$ into pairwise disjoint dyadic cubes. Since $\Omega_\lambda$ has finite measure, for each $x\in \Omega_\lambda $
there is a maximal dyadic cube $Q_x$ in $\Rf^n$ such that
$$   \sup_{y\in Q_x} \big|e^{-r_{Q_x}\sqrt{L}}f(y)\big|>\lambda.   $$
We use $\widetilde {Q_x}$ to denote the father of $Q_x$ in the system of dyadic cubes in $\Rf^n$. Then, since $Q_x$ is maximal, we know that
$$   \sup_{y\in \widetilde {Q_x}} \big|e^{-r_{\widetilde {Q_x}}\sqrt{L}}f(y)\big|\leq \lambda.   $$
Let $Q_j$ be the collection of all such maximal dyadic cubes containing all $x$ in $\Omega_\lambda$, we have that
$$ \Omega_\lambda=\bigcup_jQ_j, $$
where all $Q_j$ are obviously pairwise disjoint. For each $Q_j$, following similar estimates as in \cite[Lemma 5.3]{DY1}
we can prove that
$$   \mu(\{  x\in Q_j: |f(x)|>N\lambda,\ \mathcal M^\sharp f(x) \leq \gamma\lambda \})
 \leq C\gamma\mu(Q_j).
$$
Then by using pairwise disjointness, summing over $j$, we obtain that \eqref{goodlambda2'} holds, and hence \eqref{goodlambda2} holds.

\medskip
We now prove \eqref{goodlambda1}. 
In comparison with \eqref{goodlambda2}, we point out that in the definition of $\mathcal N_L f(x)$, for the dyadic cubes  $Q\in \mathcal{I}_1$, we skip a corner of the cube. And so we need the cubes in (3) of the definition of $\mathcal{D}_1$. Hence the decomposition in this case is essentially different from the previous setting (\cite{Gra} or \cite{DY1}).

To begin with, define
$$
\overline{\mathcal N}_L f(x):=\max\left\{\sup_{\substack{Q: {\rm\ dyadic\ cubes\ in \ } \R^m\\ Q\ni x,Q\in \mathcal{I}_2}} \sup_{y\in Q} \big|e^{-r_Q\sqrt{L}}f(y)\big|,\sup_{\substack{Q: {\rm\ dyadic\ cubes\ in \  } \R^m\\ Q\ni x,Q\in \mathcal{I}_1}} \sup_{y\in Q,|y|>r_Q^{\frac{m-n}{m-2}}/2} \big|e^{-r_Q\sqrt{L}}f(y)\big|\right\}.
$$
We again set $E_{\lambda}=\{  x\in \R^m: |f(x)|>N\lambda,\ \mathcal M^\sharp f(x) \leq \gamma\lambda \}$ and
$\Omega_\lambda:=\{x\in \R^m\backslash K: \overline{\mathcal N}_L f(x)>\lambda\}$.
We assume that $\Omega_\lambda$ has finite measure since otherwise there is nothing to prove.

For $x\in \Omega_\lambda$, according to the definition of  $\overline{\mathcal N}_L f(x)$,  we have three different types of cubes $Q_x$ that contains $x$:

Case (1): $Q_x\in \mathcal I_2$, and $\widetilde {Q_x}\in \mathcal I_2$ \ \ (where $\widetilde {Q_x}$ is the father of $Q_x$), such that
$$ \sup_{y\in Q_x} |e^{-r_{Q_x} \sqrt L} f(y)| >\lambda\quad{\rm and}\quad  \sup_{y\in \widetilde{Q_x}} |e^{-r_{\widetilde{Q_x}} \sqrt L} f(y)| \leq\lambda;$$

Case (2): $Q_x\in \mathcal I_2$, and $\widetilde {Q_x}\in \mathcal I_1$, satisfying
$$  \sup_{y\in Q_x} |e^{-r_{Q_x} \sqrt L} f(y)| >\lambda    $$
and
$$  \sup_{y\in \widetilde{Q_x}, |y|\geq r_{\widetilde{Q_x}}^{\frac{m-n}{m-2}}/2 } |e^{-r_{\widetilde{Q_{x}}} \sqrt L} f(y)| \leq\lambda;  $$

Case (3): $Q_x\in \mathcal I_1$, and of course $\widetilde {Q_x}\in \mathcal I_1$, satisfying that
$$  \sup_{y\in Q_x, |y|\geq r_{Q_x}^{\frac{m-n}{m-2}}/2 } |e^{-r_{Q_x} \sqrt L} f(y)| >\lambda    $$
and
$$  \sup_{y\in \widetilde{Q_x}, |y|\geq r_{\widetilde{Q}_x}^{\frac{m-n}{m-2}}/2 } |e^{-r_{\widetilde{Q_{x}}} \sqrt L} f(y)| \leq\lambda;  $$

\medskip

Then, given $x\in \Omega_\lambda$, there must be a maximal dyadic cube $Q_x$ which is contained in one of the three cases above. Similarly for any $y\in \Omega_\lambda$, there must be a maximal dyadic cube $Q_y$ which is contained in one of the three cases above. Moreover, since $Q_x$ and $Q_y$ are dyadic and maximal, we have that
$$ Q_x=Q_y \quad{\rm or}\quad Q_x\cap Q_y =\emptyset.$$

We now split the cubes $\{Q_x\}_{x\in \Omega_\lambda}$ into three groups as follows:
$$\{Q_x\}_{x\in \Omega_\lambda} = \{Q_i^{(1)}\} \cup  \{Q_j^{(2)}\}\cup  \{Q_k^{(3)}\},$$
where $Q_i^{(1)}$ are in Case (1), $Q_j^{(2)}$ are in Case (2) and $Q_k^{(3)}$ are in Case (3). We point out that the number of cubes $Q_k^{(3)}$ is at most $2^m$, since $Q_k^{(3)}\in \mathcal I_1$ and $Q_k^{(3)}$ is maximal.

We now have
\begin{eqnarray}\label{inclusion}
 E_\lambda\subset \Omega_\lambda=\Big(\bigcup_i Q_i^{(1)}\Big) \bigcup \Big(\bigcup_j Q_j^{(2)}\Big) \bigcup
\Big(\bigcup_k  Q_k^{(3)} \Big).
\end{eqnarray}

We first prove that for all $Q_i^{(1)}$ in Case (1), we have the estimate
\begin{eqnarray}\label{estiamteQj}
\mu(\{  x\in Q_i^{(1)}: |f(x)|>N\lambda,\ \mathcal M^\sharp f(x) \leq \gamma\lambda \})
 \leq C\gamma\mu(Q_i^{(1)}).
\end{eqnarray}

\color{black}

\bigskip
To see this, now denote $Q_i^{(1)}=Q$ and $\widetilde{Q}$ is the father of $Q$.
For $x\in Q$,
$$
f(x)-A_{\widetilde{Q}}f(x)>N\lambda-\lambda>\lambda.
$$
Then we conclude that for any $\xi\in Q$,
\begin{align*}
\mu(\{  x\in Q: |f(x)|>N\lambda \})&\leq \mu(\{x\in Q: |f(x)-A_{\widetilde{Q}}f(x)|>\lambda\})\\
&\leq \frac{1}{\lambda}\int_{Q} |f(x)-A_{\widetilde{Q}}f(x)| {\,} d\mu(x)\\
&\leq \frac{2^m\mu(Q)}{\lambda} \frac{1}{\mu(\widetilde{Q})} \int_{\widetilde{Q}} |f(x)-A_{\widetilde{Q}}f(x)| {\,} d\mu(x)\\
&\leq \frac{2^m\mu(Q)}{\lambda}\mathcal M^\sharp f(\xi).
\end{align*}
When we go back to estimate~\eqref{estiamteQj}, we may assume that there exists $\xi\in Q$ and $M^\sharp f(\xi)\leq \gamma\lambda$.
Then
\begin{eqnarray*}
\mu(\{  x\in Q: |f(x)|>N\lambda,\ \mathcal M^\sharp f(x) \leq \gamma\lambda \})
 \leq 2^m\gamma\mu(Q),
\end{eqnarray*}
which implies that the estimate~\eqref{estiamteQj} holds.

Then we prove that for all $Q_j^{(2)}$ in Case (2), we have the estimate
\begin{eqnarray}\label{estiamteQj2}
\mu(\{  x\in Q_j^{(2)}: |f(x)|>N\lambda,\ \mathcal M^\sharp f(x) \leq \gamma\lambda \})
 \leq C\gamma\mu(Q_j^{(2)}).
\end{eqnarray}

\color{black}

\smallskip
To see this, now denote $Q_j^{(2)}=Q$, denote $\widetilde{Q}$  the father of $Q$.
Note that if $Q\in \mathcal{I}_2$ and $\widetilde{Q}\in \mathcal{I}_1$. Then for any $x\in Q$, we have
$|x|>r_{\widetilde{Q}}^{(m-n)/(m-2)}/2$. Then recall that from the way we choose $Q$,
$$
\sup_{y\in \widetilde{Q},|y|>r_{\widetilde{Q}}^{\frac{m-n}{m-2}}/2} |A_{\widetilde{Q}}f(y)|\leq\lambda.
$$
This implies that for every $x\in Q$, we have
$$
f(x)-A_{\widetilde{Q}}f(x)>N\lambda-\lambda>\lambda.
$$
Then for what is following to prove ~\eqref{estiamteQj2} is similar to the proof of~\eqref{estiamteQj}.

In the end, we prove that for all $Q_k^{(3)}$ in Case (3), we have the estimate
\begin{eqnarray*}
\mu\big(\big\{  x\in Q_k^{(3)} : |f(x)|>N\lambda,\ \mathcal M^\sharp f(x) \leq \gamma\lambda \big\}\big)
 \leq 2^m\gamma\mu(\{x\in M: \mathcal N_L f(x) >\lambda\}).
\end{eqnarray*}
To see this, we denote $Q_k^{(3)}=Q$ and denote $\widetilde{Q}$ the father of $Q$. We
recall that in this case,
$$  \sup_{y\in \widetilde{Q}, |y|\geq r_{\widetilde{Q}}^{\frac{m-n}{m-2}}/2 } |e^{-r_{\widetilde{Q}} \sqrt L} f(y)| \leq\lambda.  $$
This implies that for every $x\in Q$ and $|x|\geq r_{Q}^{\frac{m-n}{m-2}}$, we have $|x|\geq r_{Q}^{\frac{m-n}{m-2}}\geq r_{\widetilde{Q}}^{\frac{m-n}{m-2}}/2$ and thus
$$
f(x)-A_{\widetilde{Q}}f(x)>N\lambda-\lambda>\lambda.
$$

Then
\begin{align*}
&\mu(\{  x\in Q\cap \{|x|\geq r_{Q}^{\frac{m-n}{m-2}}\}: |f(x)|>N\lambda \})\\
&\leq \mu(\{x\in Q\cap \{|x|\geq r_{Q}^{\frac{m-n}{m-2}}: |f(x)-A_{\widetilde{Q}}f(x)|>\lambda\})\\
&\leq \frac{1}{\lambda}\int_{Q} |f(x)-A_{\widetilde{Q}}f(x)| {\,} d\mu(x)\\
&\leq \frac{2^m\mu(Q)}{\lambda} \frac{1}{ \mu(\widetilde{Q})} \int_{\widetilde{Q}} |f(x)-A_{\widetilde{Q}}f(x)| {\,} d\mu(x)\\
&\leq \frac{C2^m}{\lambda}\mathcal M^\sharp f(\xi)\,\mu(Q).
\end{align*}
Hence, we obtain that
\begin{eqnarray*}
\mu(\{  x\in Q \cap \{|x|\geq r_{Q}^{\frac{m-n}{m-2}}\}: |f(x)|>N\lambda,\ \mathcal M^\sharp f(x) \leq \gamma\lambda \})
\leq C2^m\gamma\mu(Q).
\end{eqnarray*}

Then it remains to prove that
\begin{eqnarray*}
\mu(\{  x\in Q \cap \{|x|< r_{Q}^{\frac{m-n}{m-2}}\}: |f(x)|>N\lambda,\ \mathcal M^\sharp f(x) \leq \gamma\lambda \})
\leq C2^m\gamma\mu(\{x\in M: \mathcal N_L f(x)>\lambda\}).
\end{eqnarray*}
Note that we use $\mu(\{x\in M: \mathcal N_L f(x)>\lambda\})$ instead of $\mu(Q)$.
Then we can find the smallest dyadic cube $Q_m\in \mathcal{I}_1$ such that $Q \cap \{|x|<r_{Q}^{\frac{m-n}{m-2}}\}\subset Q_m$ and $r_{Q_m}\sim r_{Q}^{\frac{m-n}{m-2}}$. And we find a dyadic cube $Q_n\subset \Rf^n$ and $Q_n\in \mathcal{I}_1$ such that $r_{Q_n}=r_Q^{m/n}$. Combining $Q_m$ and $Q_n$ we have a new cube $Q_b=Q_m\cup Q_n$ which belongs to (3) of definition of the maximal function $\mathcal{N}_L$. Note that $r_{ Q_b}=r_{Q_n}=r_Q^{m/n}$.
Now if we have
$$  \sup_{y\in Q_b} |e^{-r_{Q_b} \sqrt L} f(y)| \leq\lambda,  $$
then following the similar proof of~\eqref{estiamteQj}, we obtain that
$$
\mu(\{  x\in Q \cap \{|x|< r_{Q}^{\frac{m-n}{m-2}}/2\}: |f(x)|>N\lambda,\ \mathcal M^\sharp f(x) \leq \gamma\lambda \})
\leq C2^m\gamma\mu(Q_b)\leq C2^m\gamma\mu(Q).
$$
So we only need to consider the case that
$$  \sup_{y\in Q_b} |e^{-r_{Q_b} \sqrt L} f(y)| >\lambda.  $$
We now point out that there must be one $Q'_n$ from the ancestors of $Q_n$  such that
 $$
 \sup_{y\in Q_m\cup Q'_n} |e^{-r_{Q'_n} \sqrt L} f(y)| > \lambda
$$
and
$$
 \sup_{y\in Q_m\cup \widetilde{Q'_n}} |e^{-r_{\widetilde{Q'_n}} \sqrt L} f(y)| \leq \lambda.
$$
Otherwise $\mu(\{x\in M: \mathcal N_L f(x)>\lambda\})$ would be infinity. The key point here is that there exists a small enough $\varepsilon>0$ such that
$$
r_{Q_m}\sim r_{Q}^{\frac{m-n}{m-2}}<r_{Q}^{\frac{m}{n}\cdot\frac{n-\varepsilon}{m-\varepsilon}}\leq r_{Q_n'}^{\frac{n-\varepsilon}{m-\varepsilon}}.
$$
That means $Q'_b:=Q_m\cup Q'_n\in \mathcal{B}_1^{\varepsilon}$. Also note that $Q'_b\subset \{x\in M: \mathcal N_L f(x)>\lambda\}$. Then following the similar proof of~\eqref{estiamteQj}, we obtain that
\begin{eqnarray*}
&&\mu(\{  x\in Q \cap \{|x|< r_{Q}^{\frac{m-n}{m-2}}/2\}: |f(x)|>N\lambda,\ \mathcal M^\sharp f(x) \leq \gamma\lambda \})\\
&&\leq C2^m\gamma\mu(Q'_b)\leq C2^m\gamma\mu(\{x\in M: \mathcal N_L f(x)>\lambda\}).
\end{eqnarray*}
Note that when we add all the measure together,  the number of $Q_k^{(3)}$ in Case (3) are at most $2^m$. So in the above estimate we can  use $\mu(\{x\in M: \mathcal N_L f(x)>\lambda\})$ instead of $\mu(Q)$.

Thus, combining all the cases above, and the inclusion as in \eqref{inclusion}, we obtain that
Lemma~\ref{lemma5.1} holds.
\end{proof}

\begin{theorem}\label{th interpolation}
Let $1\leq s\leq q$. Assume that $T$  is a sublinear operator that is
bounded on $L^q (M)$, $1 \leq q < \infty$, and
$$ \|\mathcal M^{\sharp} (Tf)\|_{L^\infty(M)} \leq  \|f\|_{L^\infty(M)}. $$
Then $T$ is bounded on $L^p(M)$ for all $q<p<\infty$.
\end{theorem}

\begin{lemma}\label{lemma sharp}
Let $1<p<\infty$. For every $f \in L^p(M)$, there exists
a constant $c_{p}$ which is independent on $f$ such that
$$ \|f\|_{L^p(M)} \leq c_p \|\mathcal M^\sharp f\|_{L^p(M)}. $$
\end{lemma}

\begin{proof}

Note that
\begin{align*}
\|f\|^p_{L^p(M)} &= pN^p \int_0^\infty \lambda^{p-1} \mu(\{  x\in M: |f(x)|>N\lambda \})d\lambda \\
&\leq pN^p \int_0^\infty \lambda^{p-1} \mu(\{  x\in M: |f(x)|>N\lambda,\ \mathcal M^\sharp f(x) \leq \gamma\lambda \})d\lambda \\
&\quad+ pN^p \int_0^\infty \lambda^{p-1} \mu(\{  x\in M: \mathcal M^\sharp f(x) > \gamma\lambda \})d\lambda\\
&\leq pN^p \int_0^\infty \lambda^{p-1} \mu(\{  x\in M: |f(x)|>N\lambda,\ \mathcal M^\sharp f(x) \leq \gamma\lambda \})d\lambda \\
&\quad+\frac{pN^p}{\gamma^p}\| \mathcal M^\sharp f \|^p_{L^p(M)}.
\end{align*}
From good Lambda inequality,
\begin{align*}
\|f\|_{L^p(M)}^p
&\leq C pN^p \gamma\int_0^\infty \lambda^{p-1} \mu(\{  x\in M: \mathcal N_L f(x)>\lambda \})d\lambda \\
&+C \frac{pN^p}{\gamma^p}\| \mathcal M^\sharp f \|^p_{L^p(M)}\\
&\leq C\gamma\|\mathcal M_\Delta f\|^p_{L^p(M)}+\frac{pN^p}{\gamma^p}\| \mathcal M^\sharp f \|^p_{L^p(M)}\\
&\leq C\gamma\|f\|^p_{L^p(M)}+\frac{pN^p}{\gamma^p}\| \mathcal M^\sharp f \|^p_{L^p(M)}.
\end{align*}
Choose $\gamma=1/(2C)$, and we have
$$
\|f\|_{L^p(M)}^p\leq C\| \mathcal M^\sharp f \|^p_{L^p(M)}.
$$

The proof of Lemma \ref{lemma sharp} is complete.
\end{proof}

\begin{proof}[Proof of Theorem~\ref{th interpolation}]

We define a new sublinear operator
$$
T^\sharp f(x)=\mathcal M^\sharp Tf(x).
$$
According to the assumption, $T^\sharp$ is bounded from $L^\infty(M)$ to $L^\infty(M)$.
Then we show that $\mathcal  M^\sharp$ is bounded on $L^p$ for all $1<p<\infty$.
\begin{align*}
\mathcal M^\sharp T g(x)
&\leq \sup_{B\ni x}\frac{1}{\mu(B)}\int_B |g(y)-A_Bg(y)| {\,} d\mu(y)\\
&\leq \sup_{B\ni x} \frac{1}{\mu(B)}\int_B |g(y)| d\mu(y)+\sup_{B\ni x}\frac{1}{\mu(B)}\int_B |A_Bg(y)| {\,} d\mu(y)\\
&\leq \mathcal M g(x)+\sup_{B\ni x}\frac{1}{\mu(B)}\int_B \sup_{t>0}|A_tg(y)| {\,} d\mu(y)\\
&\leq \mathcal M g(x) +\mathcal M (\mathcal N_L g)(x).
\end{align*}
And from Lemma~\ref{max} and Theorem~\ref{th nontangential 2}, both $\mathcal M$ and $\mathcal N_L$ are bounded on $L^p$ for all $1<p<\infty$.
Also note that
$T$ is bounded on $L^q$ and then we have
\begin{eqnarray*}
\|T^\sharp f(x)\|_{L^q(M)}=\|\mathcal M^\sharp T f(x)\|_{L^q(M)}
\leq C\|T f\|_{L^q(M)}\leq C\|f\|_{L^q(M)}.
\end{eqnarray*}
Then by interpolation we have for all $q\leq p<\infty$
\begin{eqnarray*}
\|T^\sharp f\|_{L^p(M)}\leq C\|f\|_{L^p(M)}.
\end{eqnarray*}
Then from Lemma~\ref{lemma sharp}, we have
\begin{eqnarray*}
\|Tf\|_{L^p(M)}\leq \|\mathcal M^\sharp Tf\|_{L^p(M)}+\|\mathcal M^\sharp Tf\|_{L^p(M)}^2=\|T^\sharp f\|_{L^p(M)}+\|T^\sharp f\|_{L^p(M)}^2\leq C\|f\|_{L^p(M)}+C\|f\|_{L^p(M)}^2.
\end{eqnarray*}

The proof of Theorem~\ref{th interpolation} is complete.
\end{proof}

We now apply the sharp maximal function to obtain an interpolation
theorem for an analytic family of linear operators. Our assumptions are as
follows.

(a) Let $S$ denote the closed strip $0\leq Re z\leq 1 $ in the complex $z$-plane. There
exists some $1 < q < \infty$ such that $T_z$ is a family of uniformly bounded linear operator
on $L^q(M)$ i.e., there
is a $C$ such that
$$
\|T_z\|_{L^q(M)\to L^q(M)}\leq C \quad\quad \forall z\in S.
$$

(b) $T_z$ is a holomorphic function of $z$ in the sense that
$$
z\rightarrow \int T_z( f )(x)g(x)d(x)
$$
is continuous in $S$ and analytic in the interior of $S$ whenever $f\in L^q(M)$ and $g \in L^{q'}(M)$ with $1/q + 1/q' = 1$.

(c) There exists a constant $N>0$ so that
\begin{eqnarray}
\|T_{it}(f)\|_{L^q(M)}\leq N\|f\|_{L^q(M)}, \quad f\in  L^q(M) \cap L^\infty(M), \, -\infty<t<\infty
\end{eqnarray}
and
\begin{eqnarray}
\|T_{1+it}(f)\|_{\bmo}\leq N\|f\|_{L^\infty(M)}, \quad f\in L^q(M)\cap L^\infty(M), \, -\infty<t<\infty.
\end{eqnarray}

\begin{theorem}\label{Theorem-Interpolation}
Under the above assumptions (a) to (c), we can conclude that
$$
\|T_\theta(f)\|_{L^p(M)}\leq N_\theta\|f\|_{L^p(M)}, \quad f\in L^p(M)\cap L^q(M)
$$
whenever $0\leq \theta=1-q/p<1$ and $N_\theta$ depends only on $N$ and $\theta$, and not on $C$.
\end{theorem}

\begin{proof}
The proof of this theorem follows closely the standard proof of Theorem
4 in [Stein, chap. 4] and Theorem 5.7 in [DY]. For completeness, we modify
the proof and sketch it here.

We fix a measurable function $x \mapsto B_x $ from points in $M$ to balls belonging to $\mathcal B_1$  with $x\in B_x$ and
a measurable function $x \mapsto B'_x $ from points in $M$ to balls belonging to $\mathcal B_0$ with $x\in B'_x$.
 We also fix two measurable function $\eta^1_x(y)$  and $\eta^2_x(y)$ with $|\eta^1_x(y)|,|\eta^2_x(y)|\leq 1$
 for $(x, y)\in M\times M$. Starting with an $f \in L^q$, we set $F_z = T_z( f )$ for $z$ in the
strip $S$ and write
$$
U^z(f)(x):=\frac{1}{\mu(B_x)}\int_{B_x} [F_z(y)-A_{B_x}F_z(y)]\eta^1_x(y)d\mu(y)+\frac{1}{\log r_{B'_x}\mu(B'_x)}\int_{B'_x} [F_z(y)-A_{B'_x}F_z(y)]\eta^2_x(y)d\mu(y).
$$
It is easy to check that
\begin{eqnarray}
|U^z(f)(x)|\leq 2M^\sharp (T_zf)(x) \quad\quad\mbox{and}\quad\quad 2\sup|U^z(f)(x)|\geq M^\sharp (T_zf)(x),
\end{eqnarray}
where the supremum is taken over all possible $B_x$ and $B'_x$ and functions $\eta^1_x$ and $\eta^2_x$ described
above.

The left proof  is quite similar to that of Theorem
4 in \cite[Chapter 4]{St2}   and \cite[Theorem 5.7]{DY1}. We omit it.
\end{proof}

\section{Boundedness of singular integrals from $L^\infty(M)$ to $\bmo$}
\setcounter{equation}{0}

While the regularised BMO spaces (introduced by Tolsa) can be defined for general non-homogeneous spaces which include the non-doubling manifolds with ends  $\mathbb R^m \sharp \Rf^n$, they are not the natural setting to study the end-point estimates for singular integrals with rough kernels such as singular integrals associated to an operator $L$ with generalised Poisson bounds when there are no further assumptions on the regularity of the kernels of
$e^{-tL}$. In this section, we will show that holomorphic functional calculus of $\sqrt{L}$ which includes the purely imaginary powers
$(\sqrt {L})^{it}$ with $t$ real, is bounded from $L^{\infty}(M)$ into our BMO space $\bmo$.

Concerning the definition of holomorphic functional calculus of operators, we refer the reader to \cite{Mc}. We will now prove Theorem \ref{main3}.

\begin{proof}[\bf Proof of Theorem \ref{main3}]
To begin with, we note that from Theorem \ref{main0}, we have that
$\bmo$ coincides with $\text{BMO}_L^\rho(M)$ for $0<\rho< n$. We now take $\rho=2$ since in our setting, $n\geq 3$.
Suppose $f\in L^\infty(M)$.

To verify that $\tilde  m(\sqrt{L})$ maps $L^\infty(M)$ to $\bmo$,
by definition, we only need to prove that there exists a positive constant $C$ such that
for all $B\in \mathcal{B}_0^2$,
\begin{align}\label{B0}
{1\over \log r_B\ \mu(B)} \int_B \big| \big( I-e^{-r_B\sqrt{L}}   \big)\tilde m(\sqrt{L})f(x) \big|\, d\mu(x)\leq C\|f\|_{L^\infty(M)}
\end{align}
and for all $B\in \mathcal{B}_1^2$,
\begin{align}\label{B1}
{1\over \mu(B)} \int_B \big| \big( I-e^{-r_B\sqrt{L}}   \big)\tilde m(\sqrt{L})f(x) \big|\, d\mu(x)\leq C\|f\|_{L^\infty(M)},
\end{align}
where $\mathcal{B}_0^2$ and $\mathcal{B}_1^2$ are defined as in \ref{b1} and \eqref{b2}, respectively.

We first consider \eqref{B1}. For every $B\in \mathcal{B}_1^2$ we consider the following two cases.

{\bf Case 1}: the center of $B$ is in $\R^m$.

To continue, we set
\begin{align}\label{f1f2}
f_1(x) = f(x)\cdot \chi_{\R^m\backslash K}, \quad f_2(x) = f(x)\cdot \chi_{\Rf^n}.
\end{align}
Then, to estimate \eqref{B1}, based on the upper bounds of the Poisson kernel, we need to estimate
\begin{align}\label{B11}
E:={1\over \mu(B)} \int_B \big| \big( I-e^{-r_B\sqrt{L}}   \big)\tilde m(\sqrt{L})f_1(x) \big|\, d\mu(x)
\end{align}
\noindent and
\begin{align}\label{B12}
F:= {1\over \mu(B)} \int_B \big| \big( I-e^{-r_B\sqrt{L}}   \big)\tilde m(\sqrt{L})f_2(x) \big|\, d\mu(x).
\end{align}

We first consider the term $E$. We write $f_1(x)= f_{11}(x)+f_{12}(x)$, where $f_{11}(x)=f_1(x)\cdot \chi_{4B}(x)$,
and $f_{12}(x)=f_1(x)\cdot \chi_{(4B)^c}(x)$. Note that $\mu(4B)\leq C\mu(B)$ in this case.
\begin{align*}
E&\leq {1\over \mu(B)} \int_B \big| \big( I-e^{-r_B\sqrt{L}}   \big)\tilde m(\sqrt{L})f_{11}(x) \big|\, d\mu(x)\\
&\quad+ {1\over \mu(B)} \int_B \big| \big( I-e^{-r_B\sqrt{L}}   \big)\tilde m(\sqrt{L})f_{12}(x) \big|\, d\mu(x)\\
&=: E_{11}+E_{12}.
\end{align*}
For the term  $E_{11}$, from the $L^2(M)$ boundedness of $\big( I-e^{-r_B\sqrt{L}}   \big)$ and $\tilde m(\sqrt{L})$ it is direct that
\begin{align*}
E_{11}
&\leq {1\over \mu(B)} \mu(B)^{1\over2}  \bigg( \int_B \big| \big( I-e^{-r_B\sqrt{L}}   \big)\tilde m(\sqrt{L})f_{11}(x) \big|^2\, d\mu(x)\bigg)^{1\over2}\\
&\leq C\mu(B)^{-{1\over2}} \|f_{11}\|_{L^2(M)}\\
&\leq C\|f\|_{L^\infty(M)}\mu(B)^{-{1\over2}} \mu(4B)^{1\over2} \\
&\leq C\|f\|_{L^\infty(M)},
\end{align*}
where the last inequality follows from the fact that $\mu(4B)\leq 4^m\mu(B)$, since the the center of $B$ is in $\R^m$.

We now consider the term $E_{12}$.

Note that
$$   I-e^{-r_B\sqrt{L}}  = \int_0^{r_B} -{d\over ds} e^{-s\sqrt{L}} {ds}  =\int_0^{r_B} s\sqrt{L} e^{-s\sqrt{L}}\, {ds\over s}. $$
Then we have
\begin{align}\label{eeee12}
&\big( I-e^{-r_B\sqrt{L}}   \big)\tilde m(\sqrt{L})f_{12}(x)\\
&= \int_0^{r_B} s\sqrt{L} e^{-s\sqrt{L}} {ds\over s} \int_0^\infty t\sqrt L \exp(-t\sqrt L) f_{12}(x) m(t)  {dt\over t}\nonumber\\
&=  \int_0^{r_B}\int_0^\infty {st \over (s+t)^2} \big( (s+t)\sqrt{L}\big)^2\exp(-(s+t)\sqrt L) f_{12}(x) m(t) \, {dt\over t}{ds\over s}\nonumber\\
&=  \int_0^{r_B}\int_0^\infty {st \over (s+t)^2} \int_M p_{t+s,2}(x,y) f_{12}(y) d\mu(y)\, m(t) \, {dt\over t}{ds\over s}\nonumber\\
&=  \int_0^{r_B}\int_0^\infty {st \over (s+t)^2} \int_{\R^m\cap (4B)^c} p_{t+s,2}(x,y) f(y) d\mu(y)\, m(t) \, {dt\over t}{ds\over s},\nonumber
\end{align}
where$p_{t+s,2}(x,y)$ is the   kernel  of the operator $\big( (s+t)\sqrt{L}\big)^2\exp(-(s+t)\sqrt L)$.

Next, we denote by $x_B$ and $r_B$ the center and the radius of $B$, respectively. Then, to estimate $E_{12}$, we consider the following cases: $d(x_B,K)\leq 2r_B$ and $d(x_B,K)>2r_B$.

Case (i): $d(x_B,K)\leq 2r_B$.

Note that in this case, when $y\in \R^m\cap (4B)^c$ and $x\in B$ is, from estimates 2, 4 and 5 in Proposition~\ref{prop2.1}, we obtain that the kernel $p_{t+s,2}(x,y)$  satisfies
\begin{align*}
|p_{t+s,2}(x,y)|\leq   C{(t+s)^2\over (t+s+d(x,y))^{m+2}}+ C{(t+s)^2\over |y|^{m-2} (t+s+|y|)^{n+2}}.
\end{align*}
Thus, from \eqref{eeee12} and the upper bound of $p_{t+s,2}(x,y)$ as above, we have
\begin{align*}
E_{12}
&\leq  {C\over \mu(B)} \int_B \int_0^{r_B}\int_0^\infty {st \over (s+t)^2} \int_{\R^m\cap (4B)^c}  {(t+s)^2\over (t+s+d(x,y))^{m+2}}\ |f(y)|\,d\mu(y)\, {dt\over t}{ds\over s} \, d\mu(x)\\
&\quad+ {C\over \mu(B)} \int_B \int_0^{r_B}\int_0^\infty {st \over (s+t)^2} \int_{\R^m\cap (4B)^c} {(t+s)^2\over |y|^{m-2} (t+s+|y|)^{n+2}}\ |f(y)|\,d\mu(y)\, {dt\over t}{ds\over s} \, d\mu(x)\\
&=E_{121}+E_{122}.
\end{align*}
We first estimate $E_{121}$. Note that
\begin{align*}
E_{121}&\leq  {C\|f\|_{L^\infty(M)}\over \mu(B)} \int_B \int_0^{r_B}\int_0^{r_B} {st \over (s+t)^2} \int_{\R^m\cap (4B)^c}  {(t+s)^2\over (t+s+d(x,y))^{m+2}}\,d\mu(y)\, {dt\over t}{ds\over s} \, d\mu(x)\\
&\quad+  {C\|f\|_{L^\infty(M)}\over \mu(B)} \int_B \int_0^{r_B}\int_{r_B}^\infty {st \over (s+t)^2} \int_{\R^m\cap (4B)^c}  {(t+s)^2\over (t+s+d(x,y))^{m+2}}\,d\mu(y)\, {dt\over t}{ds\over s} \, d\mu(x)\\
&=E_{1211}+E_{1212}.
\end{align*}
For the term $E_{1212}$, it is clear that from the fact that
\begin{align*}
\int_{\R^m\cap (4B)^c}  {(t+s)^2\over (t+s+d(x,y))^{m+2}}\,d\mu(y)
\leq C,
\end{align*}
we  have
\begin{align*}
E_{1212}
&\leq C{\|f\|_{L^\infty(M)}\over \mu(B)} \int_B \int_0^{r_B}\int_{r_B}^\infty {st \over (s+t)^2} \, {dt\over t}{ds\over s} \, d\mu(x)\leq C\|f\|_{L^\infty(M)}.
\end{align*}
For the term $E_{1211}$, we have that
\begin{align*}
E_{1211}&\leq  {C\|f\|_{L^\infty(M)}\over \mu(B)} \int_B \int_0^{r_B}\int_0^{r_B}  \int_{\R^m\cap (4B)^c}  {1\over d(x,y)^{m+2}}\,d\mu(y)\, {dt}{ds} \, d\mu(x)\\
&\leq  {C\|f\|_{L^\infty(M)}\over \mu(B)} \cdot r_B^2\cdot \int_B  \sum_{j=2}^\infty \int_{y\in \R^m\cap (4B)^c: 2^jr_B<d(x,y)\leq 2^{j+1}r_B}  {1\over d(x,y)^{m+2}}\,d\mu(y) \, d\mu(x)\\
&\leq  {C\|f\|_{L^\infty(M)}\over \mu(B)} \mu(B)  \cdot r_B^2\cdot\sum_{j=2}^\infty (2^{j+1}r_B)^m  {1\over (2^jr_B)^{m+2}}\\
&\leq C\|f\|_{L^\infty(M)}.
\end{align*}

Next we consider $E_{122}$. Again we write
\begin{align*}
E_{122}&\leq {C\|f\|_{L^\infty(M)}\over \mu(B)} \int_B \int_0^{r_B}\int_0^{r_B} {st \over (s+t)^2} \int_{\R^m\cap (4B)^c} {(t+s)^2\over |y|^{m-2} (t+s+|y|)^{n+2}}\,d\mu(y)\, {dt\over t}{ds\over s} \, d\mu(x)\\
&\quad+ {C\|f\|_{L^\infty(M)}\over \mu(B)} \int_B \int_0^{r_B}\int_{r_B}^\infty {st \over (s+t)^2} \int_{\R^m\cap (4B)^c} {(t+s)^2\over |y|^{m-2} (t+s+|y|)^{n+2}}\,d\mu(y)\, {dt\over t}{ds\over s} \, d\mu(x)\\
&=E_{1221}+E_{1222}.
\end{align*}
We first consider the term $E_{1222}$.
We claim that there exists a positive constant $C$ such that for all $s,t\in (0,\infty)$,
\begin{align}\label{eeee1222}
 \int_{\R^m\cap (4B)^c} {(t+s)^2\over |y|^{m-2} (t+s+|y|)^{n+2}}\,d\mu(y)\leq C.
\end{align}
In fact, note that
\begin{align*}
 &\int_{\R^m\cap (4B)^c} {(t+s)^2\over |y|^{m-2} (t+s+|y|)^{n+2}}\,d\mu(y)\\
 &\leq\int_{K} {(t+s)^2\over |y|^{m-2} (t+s+|y|)^{n+2}}\,d\mu(y) + \int_{\R^m\cap K} {(t+s)^2\over |y|^{m-2} (t+s+|y|)^{n+2}}\,d\mu(y)\\
 &\leq\int_{K} {1\over |y|^{m-2} |y|^{n}}\,d\mu(y) + \int_{\R^m\cap K} {1\over |y|^{m-2} |y|^{n}}\,d\mu(y)\\
 &\leq \mu(K) + C\int_1^\infty {r^{m-1}\over  r^{m+n-2} } dr\\
 &\leq C,
 \end{align*}
where the third inequality follows from the fact that $|y|\geq1$ and from the changing of the integration on $\R^m\cap K$
 into polar coordinates. Thus, we obtain that \eqref{eeee1222} holds.
Then we further have
\begin{align*}
E_{1222}&\leq C {\|f\|_{L^\infty(M)}\over \mu(B)} \int_B \int_0^{r_B}\int_{r_B}^\infty {st \over (s+t)^2} \, {dt\over t}{ds\over s} \, d\mu(x)\leq C\|f\|_{L^\infty(M)}.
\end{align*}

We now estimate the term $E_{1221}$.
We claim that there exists a positive constant $C$ such that,
\begin{align}\label{eeee1221}
 \int_{\R^m\cap (4B)^c} {1\over |y|^{m+n} }\,d\mu(y)\leq {C\over r_B^2}.
\end{align}
In fact, note that from the condition $d(x_B,K)\leq 2r_B$ in Case (i), for every $y\in \R^m\backslash (4B)^c$, we have $|y|\geq r_B$.
\begin{align*}
 \int_{\R^m\cap (4B)^c} {1\over |y|^{m+n}}\,d\mu(y)
 &\leq \int_{\{y\in \R^m: |y|\geq r_B \} }  {1\over |y|^{m+2} }\,d\mu(y)\\
 &\leq C\int_{r_B}^\infty {r^{m-1}\over  r^{m+2} } dr\\
 &\leq {C\over r_B^2},
 \end{align*}
where the first inequality follows from the fact that $|y|\geq1$. As a consequence, we obtain that \eqref{eeee1221} holds.
Then from \eqref{eeee1221} we have
\begin{align*}
E_{1221}&\leq {\|f\|_{L^\infty(M)}\over \mu(B)} \int_B \int_0^{r_B}\int_0^{r_B}  \int_{\R^m\cap (4B)^c} {1\over |y|^{m-2} |y|^{n+2}}\,d\mu(y)\, {dt}{ds} \, d\mu(x)\\
&\leq {C\|f\|_{L^\infty(M)}\over \mu(B)} \mu(B) \int_0^{r_B}\int_0^{r_B}  \int_{\R^m\cap (4B)^c} {1\over  |y|^{m+n}}\,d\mu(y)\, {dt}{ds} \\
&\leq C\|f\|_{L^\infty(M)}.
\end{align*}

Case (ii): $d(x_B,K)> 2r_B$.

Note that in this case, the ball $B$ is contained in $\R^m\backslash K$. Hence, from estimate 5 in Proposition~\ref{prop2.1}, we obtain that the kernel $p_{t+s,2}(x,y)$  satisfies
\begin{align*}
|p_{t+s,2}(x,y)|\leq   C{(t+s)^2\over (t+s+d(x,y))^{m+2}}+ C{(t+s)^2\over |x|^{m-2} |y|^{m-2} (t+s+|x|+|y|)^{n+2}}.
\end{align*}
Thus, from \eqref{eeee12} and the upper bound of $p_{t+s,2}(x,y)$ as above, we have
\begin{align*}
E_{12}
&\leq  {C\over \mu(B)} \int_B \int_0^{r_B}\int_0^\infty {st \over (s+t)^2} \int_{\R^m\cap (4B)^c}  {(t+s)^2\over (t+s+d(x,y))^{m+2}}\ |f(y)|\,d\mu(y)\, {dt\over t}{ds\over s} \, d\mu(x)\\
&\quad+ {C\over \mu(B)} \int_B \int_0^{r_B}\int_0^\infty {st \over (s+t)^2} \int_{\R^m\cap (4B)^c} {(t+s)^2\over |x|^{m-2} |y|^{m-2} (t+s+|x|+|y|)^{n+2}}\ |f(y)|\,d\mu(y)\, {dt\over t}{ds\over s} \, d\mu(x)\\
&=E_{121}+\widetilde E_{122},
\end{align*}
where the term $E_{121}$ is exactly the same as that in Case (i).  As a consequence, we only need to verify the term
$\widetilde E_{122}$ in this case.

We now write $\widetilde E_{122}$ as
\begin{align*}
&\widetilde E_{122}\\
&\leq {C\|f\|_{L^\infty(M)}\over \mu(B)} \int_B \int_0^{r_B}\int_0^{r_B} {st \over (s+t)^2} \int_{\R^m\cap (4B)^c} {(t+s)^2\over |x|^{m-2} |y|^{m-2} (t+s+|x|+|y|)^{n+2}}\,d\mu(y)\, {dt\over t}{ds\over s} \, d\mu(x)\\
&\quad+ {C\|f\|_{L^\infty(M)}\over \mu(B)} \int_B \int_0^{r_B}\int_{r_B}^\infty {st \over (s+t)^2} \int_{\R^m\cap (4B)^c} {(t+s)^2\over |x|^{m-2} |y|^{m-2} (t+s+|x|+|y|)^{n+2}}\,d\mu(y)\, {dt\over t}{ds\over s} \, d\mu(x)\\
&=\widetilde E_{1221}+ \widetilde E_{1222}.
\end{align*}
For the term $\widetilde E_{1222}$, from the fact that $|x|\geq1$, it is direct to see that
$\widetilde E_{1222}$ is controlled by the term $ E_{1222}$ as in Case (i) above, and hence it is bounded by
$C\|f\|_{L^\infty(M)}$.

Now it suffices to verify the term $\widetilde E_{1221}$. To continue, we now claim that
\begin{align}\label{eeee1221tilde}
 \int_{\R^m\cap (4B)^c} {1\over |x|^{m-2}|y|^{m-2} } {1\over (|x|+|y|)^{n+2}}\,d\mu(y)\leq {C\over r_B^2}.
\end{align}
To see this, we first point out that since $d(x_B,K)> 2r_B$, it is direct  that for every $x\in B$,
we have $|x|\geq r_B$. Then, if $r_B\geq1$, we have
\begin{align*}
&\int_{\R^m\cap (4B)^c} {1\over |x|^{m-2}|y|^{m-2} } {1\over (|x|+|y|)^{n+2}}\,d\mu(y)\\
 &\leq \int_{\{ y\in\R^m: |y|\leq 8r_B\}} {1\over r_B^{m-2}|y|^{m-2} } {1\over (r_B+|y|)^{n+2}}\,d\mu(y)\\
 &\quad +\int_{\{ y\in\R^m: |y|> 8r_B\}} {1\over r_B^{m-2}|y|^{m-2} } {1\over (r_B+|y|)^{n+2}}\,d\mu(y)\\
 &\leq \int_{\{ y\in\R^m: |y|\leq 8r_B\}} {1\over r_B^{m+n}} \,d\mu(y)+ {1\over r_B^{m-2}} \int_{\{ y\in\R^m: |y|> 8r_B\}}{1\over |y|^{m+n}} \,d\mu(y)\\
 &\leq  C{r_B^m\over r_B^{m+n}} + {C\over r_B^{m-2}} \int_{ 8r_B}^\infty {r^{m-1}\over r^{m+n}} \,dr\\
 &\leq {C\over r_B^2}.
 \end{align*}
 If $r_B<1$, then from the fact that $|x|\geq1$, we have
\begin{align*}
\int_{\R^m\cap (4B)^c} {1\over |x|^{m-2}|y|^{m-2} } {1\over (|x|+|y|)^{n+2}}\,d\mu(y)
 &\leq \int_{\{ y\in\R^m: |y|\geq 1\}} {1\over |y|^{m-2} } {1\over |y|^{n+2}}\,d\mu(y)\leq C\leq {C\over r_B^2}.
 \end{align*}
 Combining these estimates we see that the claim \eqref{eeee1221tilde} holds. As a consequence, we obtain that
 \begin{align*}
\widetilde E_{1221}
&\leq {C\|f\|_{L^\infty(M)}\over \mu(B)} \int_B \int_0^{r_B}\int_0^{r_B}  \int_{\R^m\cap (4B)^c} {1\over |x|^{m-2} |y|^{m-2} (|x|+|y|)^{n+2}}\,d\mu(y)\, {dt}{ds} \, d\mu(x)\\
&\leq {C\|f\|_{L^\infty(M)}\over \mu(B)} \int_B \int_0^{r_B}\int_0^{r_B} {1\over r_B^2} \, {dt}{ds} \, d\mu(x)\\
&\leq C\|f\|_{L^\infty(M)}.
\end{align*}

As a consequence of all the subcases for the term $E$ above, we obtain that  $$E \leq C\|f\|_{L^\infty(M)}.$$ We now turn to the term $F$.

We write $f_2(x)= f_{21}(x)+f_{22}(x)$, where $f_{21}(x)=f_2(x)\cdot \chi_{4B}(x)$,
and $f_{22}(x)=f_2(x)\cdot \chi_{(4B)^c}(x)$. Here we point out that if $4B\cap \Rf^n=\emptyset$ then we have $f_{21}=0$ and hence $f_{22}=f_2$. Note that $\mu(4B)\leq C\mu(B)$ in this case.
\begin{align*}
F&\leq {1\over \mu(B)} \int_B \big| \big( I-e^{-r_B\sqrt{L}}   \big)\tilde m(\sqrt{L})f_{21}(x) \big|\, d\mu(x)\\
&\quad+ {1\over \mu(B)} \int_B \big| \big( I-e^{-r_B\sqrt{L}}   \big)\tilde m(\sqrt{L})f_{22}(x) \big|\, d\mu(x)\\
&=: F_{21}+F_{22}.
\end{align*}
For the term  $F_{21}$, from the $L^2(M)$ boundedness of $\big( I-e^{-r_B\sqrt{L}}   \big)$ and $\tilde m(\sqrt{L})$ it is direct that
\begin{align*}
F_{21}
&\leq {1\over \mu(B)} \mu(B)^{1\over2}  \bigg( \int_B \big| \big( I-e^{-r_B\sqrt{L}}   \big)\tilde m(\sqrt{L})f_{21}(x) \big|^2\, d\mu(x)\bigg)^{1\over2}\\
&\leq C\mu(B)^{-{1\over2}} \|f_{21}\|_{L^2(M)}\\
&\leq C\|f\|_{L^\infty(M)}\mu(B)^{-{1\over2}} \mu(4B)^{1\over2} \\
&\leq C\|f\|_{L^\infty(M)},
\end{align*}
where the last inequality follows from the fact that $\mu(4B)\leq 4^m\mu(B)$, since the the center of $B$ is in $\R^m$.

We now consider the term $F_{22}$.
Again, we have
\begin{align*}
\big( I-e^{-r_B\sqrt{L}}   \big)\tilde m(\sqrt{L})f_{22}(x)=  \int_0^{r_B}\int_0^\infty {st \over (s+t)^2} \big( (s+t)\sqrt{L}\big)^2\exp(-(s+t)\sqrt L) f_{22}(x) m(t) \, {dt\over t}{ds\over s}.
\end{align*}

Next, note that when $y\in \Rf^n\cap (4B)^c$ no mater where $x$ is, the  kernel $p_{t+s,2}(x,y)$ of the operator $\big( (s+t)\sqrt{L}\big)^2\exp(-(s+t)\sqrt L)$ satisfies
\begin{align*}
|p_{t+s,2}(x,y)|\leq   {(t+s)^2\over (t+s+d(x,y))^{m+2}}+ {(t+s)^2\over  (t+s+d(x,y))^{n+2}}.
\end{align*}
Thus we have
\begin{align*}
F_{22}&\leq {1\over \mu(B)} \int_B \int_0^{r_B}\int_0^\infty {st \over (s+t)^2} \int_M |p_{t+s,2}(x,y)|\ |f_{22}(y)|\,d\mu(y)\, {dt\over t}{ds\over s} \, d\mu(x)\\
&\leq  {1\over \mu(B)} \int_B \int_0^{r_B}\int_0^\infty {st \over (s+t)^2} \int_{\Rf^n\cap (4B)^c}  {(t+s)^2\over (t+s+d(x,y))^{m+2}}\ |f(y)|\,d\mu(y)\, {dt\over t}{ds\over s} \, d\mu(x)\\
&\quad+ {1\over \mu(B)} \int_B \int_0^{r_B}\int_0^\infty {st \over (s+t)^2} \int_{\Rf^n\cap (4B)^c} {(t+s)^2\over  (t+s+d(x,y))^{n+2}}\ |f(y)|\,d\mu(y)\, {dt\over t}{ds\over s} \, d\mu(x)\\
&=F_{221}+F_{222}.
\end{align*}
We first consider the term $F_{221}$.  One can write
\begin{align*}
F_{221}
&\leq  {\|f\|_{L^\infty(M)}\over \mu(B)} \int_B \int_0^{r_B}\int_0^{r_B} {st \over (s+t)^2} \int_{\Rf^n\cap (4B)^c}  {(t+s)^2\over (t+s+d(x,y))^{m+2}}\,d\mu(y)\, {dt\over t}{ds\over s} \, d\mu(x)\\
&\quad+  {\|f\|_{L^\infty(M)}\over \mu(B)} \int_B \int_0^{r_B}\int_{r_B}^\infty {st \over (s+t)^2} \int_{\Rf^n\cap (4B)^c}  {(t+s)^2\over (t+s+d(x,y))^{m+2}}\,d\mu(y)\, {dt\over t}{ds\over s} \, d\mu(x)\\
&=F_{2211}+F_{2212}.
\end{align*}

We first consider $F_{2212}$.
We note that
\begin{align}\label{claim eeee}
 &\int_{\Rf^n\cap (4B)^c}  {(t+s)^2\over (t+s+d(x,y))^{m+2}}\,d\mu(y)\\
 &\leq \Bigg\{\int_{\{y\in\Rf^n\cap (4B)^c:\ d(x,y)\leq (t+s)\}}+\sum_{j=0}^\infty\int_{\{y\in\Rf^n\cap (4B)^c:\ 2^j(t+s)< d(x,y)\leq 2^{j+1}(t+s)\}}\Bigg\}  {(t+s)^2\over (t+s+d(x,y))^{m+2}}\,d\mu(y).\nonumber
\end{align}
 Then we have that
\begin{align*}
 &\int_{\Rf^n\cap (4B)^c}  {(t+s)^2\over (t+s+d(x,y))^{m+2}}\,d\mu(y)\\
 &\leq {\mu(B(x,t+s))\over  (t+s)^m}+ \sum_{j=0}^\infty    \mu(B(x,2^{j+1}(t+s))) {(t+s)^2\over ( 2^j(t+s))^{m+2}}\\
  &\leq C,
\end{align*}
where  in the second inequality above we have used the fact that  $\mu(B(x,r))\leq C r^m$ for any $r>0$.
Therefore, we obtain that
\begin{align*}
F_{2212}&\leq  {C\|f\|_{L^\infty(M)}\over \mu(B)} \int_B \int_0^{r_B}\int_{r_B}^\infty {st \over (s+t)^2} \, {dt\over t}{ds\over s} \, d\mu(x)\leq C\|f\|_{L^\infty(M)}.
\end{align*}

We now consider $F_{2211}$.
\begin{align*}
F_{2211}
&\leq  {\|f\|_{L^\infty(M)}\over \mu(B)} \int_B \int_0^{r_B}\int_0^{r_B}  \int_{\Rf^n\cap (4B)^c}  {1\over d(x,y)^{m+2}}\,d\mu(y)\, {dt}{ds} \, d\mu(x).
\end{align*}
By using the same approach as in the estimate of \eqref{claim eeee}, we obtain that
$$\int_{\Rf^n\cap (4B)^c}  {1\over d(x,y)^{m+2}}\,d\mu(y)\leq {C\over r_B^2}.$$
As a consequence we obtain that
\begin{align*}
F_{2211}
&\leq  C{\|f\|_{L^\infty(M)}\over \mu(B)} \cdot{1\over r_B^2}\cdot \int_B \int_0^{r_B}\int_0^{r_B}   {dt}{ds} \, d\mu(x)\leq C\|f\|_{L^\infty(M)}.
\end{align*}

As for the term $F_{222}$, by repeating the estimates for the term $E_{121}$, we obtain that
\begin{align*}
F_{222}
\leq  C\|f\|_{L^\infty(M)}.
\end{align*}

Combining the estimates of $E$ and $F$ we obtain that \eqref{B1} holds in {\bf Case 1}.

\medskip

{\bf Case 2:}  the center of $B$ is in $\Rf^n$.

\smallskip

If $B\cap (\R^m\backslash K) = \emptyset$, then by using the same approach and similar estimates as in
 {\bf Case 1}, we can obtain that \eqref{B1} holds. Thus, it suffices to consider $B\in\mathcal B_1^2$ and
$B\cap (\R^m\backslash K) \not= \emptyset$.

In this case, we decompose $B$ as follows. Let
\begin{align}\label{B decom}
B:=\bigg(\bigcup_{K_0<k\leq 0} T_k\bigg)\, \bigcup B^m,
\end{align}
where
\begin{align*}
&T_k:=\{x\in \Rf^n\cap B: 2^{k}r_B<|x|\leq 2^{k+1}r_B\},\quad\quad k\leq -1\\
&T_0:=B\cap\Rf^n\backslash \bigcup_{K_0<k< 0} T_k\\
&B^m:=B\cap \R^m
\end{align*}
and $K_0$ is the number such that $2^{K_0}r_B\approx1$.

Let $E$ be the same as the term in \eqref{B11}.
We write
\begin{align*}
E&\leq{1\over \mu(B)} \int_{B^m} \big| \big( I-e^{-r_B\sqrt{L}}   \big)\tilde m(\sqrt{L})f_1(x) \big|\, d\mu(x)\\
&\quad+ \sum_{K_0<k\leq 0}{1\over \mu(B)} \int_{T_k} \big| \big( I-e^{-r_B\sqrt{L}}   \big)\tilde m(\sqrt{L})f_1(x) \big|\, d\mu(x)\\
&=E_1+E_2.
\end{align*}

We first estimate  $E_1$. Denote the radius of $B^m$ by $r_{B^m}$.  Similar to {\bf Case 1},
we  write the function $f_1 = f_{11}+f_{12}$, where
$f_{11}(x)=f_1(x)\cdot \chi_{4B^m}(x)$,
and $f_{12}(x)=f_1(x)\cdot \chi_{(4B^m)^c}(x)$. Note that $\mu(4B^m)\leq C\mu(B^m) \leq C\mu(B)$ in this case.

we can split the term $E_{1}$ into $ E_{11}$ and $ E_{12}$. Then the term $E_{11}$ can be handled
following the same approach  for the term $E_{11}$ in {\bf Case 1}. Hence we get that
$E_{11}\leq C\|f\|_{L^\infty(M)}$.

Now it suffices to consider $E_{12}$. We first point out that for the ball $B^m$, we can consider that its center $x_{B^m}$ is
in $K$. That is, we have $d(x_{B^m},K)\leq 2r_B$.
Then we write
\begin{align*}
E_{12}
&={1\over \mu(B)} \int_{B^m} \big| \big( I-e^{-r_{B^m}\sqrt{L}}   \big)\tilde m(\sqrt{L})f_{12}(x) \big|\, d\mu(x)
\\
&\quad\quad+{1\over \mu(B)} \int_{B^m} \big| \big(  e^{-r_{B^m}\sqrt{L}} -e^{-r_B\sqrt{L}}   \big)\tilde m(\sqrt{L})f_{12}(x) \big|\, d\mu(x)\nonumber\\
&=:\overline E_{12}+\overline{\overline E}_{12}.\nonumber
\end{align*}
For the term $\overline E_{12}$, note that the scale in the semigroup in the integrand is exactly the radius of the ball $B^m$, hence, by
using the same approach as that in the estimate of Case (i) for term $E_{12}$ in {\bf Case 1} and by   the fact that $\mu(B^m)\leq \mu(B)$, we obtain that
$$
\overline E_{12}\leq C\|f\|_{L^\infty(M)}.
$$

We now consider the term $\overline {\overline E}_{12}$.

We note that
$$   e^{-r_{B^m}\sqrt{L}}-e^{-r_B\sqrt{L}}  = \int_{r^{B^m}}^{r_B} -{d\over ds} e^{-s\sqrt{L}} {ds}  =\int_{r^{B^m}}^{r_B} s\sqrt{L} e^{-s\sqrt{L}} {ds\over s}. $$
Then we have
\begin{align*}
&\big( e^{-r_{B^m}\sqrt{L}}-e^{-r_B\sqrt{L}}   \big)\tilde m(\sqrt{L})f_{12}(x)\\
&= \int_{r_{B^m}}^{r_B} s\sqrt{L} e^{-s\sqrt{L}} {ds\over s} \int_0^\infty t\sqrt L \exp(-t\sqrt L) f_{12}(x) m(t) \, {dt\over t}\\
&=  \int_{r_{B^m}}^{r_B}\int_0^\infty {st \over (s+t)^2} \big( (s+t)\sqrt{L}\big)^2\exp(-(s+t)\sqrt L) f_{12}(x) m(t) \, {dt\over t}{ds\over s}.
\end{align*}
Hence, we further have
\begin{align*}
\overline {\overline E}_{12}
&\leq  {C\over \mu(B)} \int_{B^m} \int_{r_{B^m}}^{r_B}\int_0^\infty {st \over (s+t)^2} \int_{\R^m\cap (4B^m)^c}  {(t+s)^2\over (t+s+d(x,y))^{m+2}}\ |f(y)|\,d\mu(y)\, {dt\over t}{ds\over s} \, d\mu(x)\\
&\quad+ {C\over \mu(B)} \int_{B^m} \int_{r_{B^m}}^{r_B}\int_0^\infty {st \over (s+t)^2} \int_{\R^m\cap (4B^m)^c} {(t+s)^2\over |y|^{m-2} (t+s+|y|)^{n+2}}\ |f(y)|\,d\mu(y)\, {dt\over t}{ds\over s} \, d\mu(x)\\
&=\overline {\overline E}_{121}+\overline {\overline E}_{122}.
\end{align*}
We now further split $\overline {\overline E}_{121}$ as
\begin{align*}
\overline {\overline E}_{121}
&\leq  {C\over \mu(B)} \int_{B^m} \int_{r_{B^m}}^{r_B}\int_0^{r_B} {st \over (s+t)^2} \int_{\R^m\cap (4B^m)^c}  {(t+s)^2\over (t+s+d(x,y))^{m+2}}\ |f(y)|\,d\mu(y)\, {dt\over t}{ds\over s} \, d\mu(x)\\
&\quad+ {C\over \mu(B)} \int_{B^m} \int_{r_{B^m}}^{r_B}\int_{r_B}^\infty {st \over (s+t)^2} \int_{\R^m\cap (4B^m)^c} {(t+s)^2\over (t+s+d(x,y))^{m+2}}\ |f(y)|\,d\mu(y)\, {dt\over t}{ds\over s} \, d\mu(x)\\
&=\overline {\overline E}_{1211}+\overline {\overline E}_{1212}.
\end{align*}
For the term $\overline {\overline E}_{1212}$, following the same approach as that  in Case (i)
for the term $E_{1212}$ in {\bf Case 1}, we obtain that
\begin{align*}
\overline {\overline E}_{1212}
&\leq   {C\|f\|_{L^\infty(M)}\over \mu(B)} \int_{B^m} \int_{r_{B^m}}^{r_B}\int_{r_B}^\infty {st \over (s+t)^2} \, {dt\over t}{ds\over s} \, d\mu(x)
\leq C\|f\|_{L^\infty(M)}{\mu(B^m)\over \mu(B)} {r_B-r_{B^m}\over r_B} \leq C\|f\|_{L^\infty(M)}.
\end{align*}
For the term $\overline {\overline E}_{1211}$, following the same approach as that  in Case (i)
for the term $E_{1211}$ in {\bf Case 1}, we obtain that
\begin{align*}
\overline {\overline E}_{1211}
&\leq     {C\|f\|_{L^\infty(M)}\over \mu(B)} \cdot \int_{r_{B^m}}^{r_B}\int_0^{r_B} dtds\cdot \int_{B^m}  \sum_{j=2}^\infty \int_{y\in \R^m\cap (4B^m)^c: 2^jr_{B^m}<d(x,y)\leq 2^{j+1}r_{B^m}}  {1\over d(x,y)^{m+2}}\,d\mu(y) \, d\mu(x)\\
&\leq     {C\|f\|_{L^\infty(M)}\over \mu(B)} \mu(B^m) {r^2_B\over r_{B^m}^2} \leq  {C\|f\|_{L^\infty(M)}\over \mu(B)} r^{m-2}_{B^m} r_B^2 \leq {C\|f\|_{L^\infty(M)}\over \mu(B)} r_B^n\\
& \leq C\|f\|_{L^\infty(M)},
\end{align*}
where the fourth inequality follows from the condition that $B\in\mathcal B_1^2$.

Now, following the same approach as that in  the estimate of Case (i) for term $E_{122}$ in {\bf Case 1} and using
the condition that $B\in\mathcal B_1^2$, we can also obtain that
\begin{align*}
\overline {\overline E}_{122}
&\leq   {C\over \mu(B)} \int_{B^m} \int_{r_{B^m}}^{r_B}\int_0^{r_B} {st \over (s+t)^2} \int_{\R^m\cap (4B^m)^c} {(t+s)^2\over |y|^{m-2} (t+s+|y|)^{n+2}}\ |f(y)|\,d\mu(y)\, {dt\over t}{ds\over s} \, d\mu(x)\\
&\quad +   {C\over \mu(B)} \int_{B^m} \int_{r_{B^m}}^{r_B}\int_{r_B}^\infty {st \over (s+t)^2} \int_{\R^m\cap (4B^m)^c} {(t+s)^2\over |y|^{m-2} (t+s+|y|)^{n+2}}\ |f(y)|\,d\mu(y)\, {dt\over t}{ds\over s} \, d\mu(x)\\
&=\overline {\overline E}_{1221}+\overline {\overline E}_{1222}.
\end{align*}
And it is clear that following similar decompositions as in $\overline {\overline E}_{1211}$ and $\overline {\overline E}_{1212}$ above respectively, we have that
\begin{align*}
\overline {\overline E}_{1222}
&\leq   {C\|f\|_{L^\infty(M)}\over \mu(B)} \int_{B^m} \int_{r_{B^m}}^{r_B}\int_{r_B}^\infty {st \over (s+t)^2} \, {dt\over t}{ds\over s} \, d\mu(x)
\leq C\|f\|_{L^\infty(M)}{\mu(B^m)\over \mu(B)} {r_B-r_{B^m}\over r_B} \leq C\|f\|_{L^\infty(M)}
\end{align*}
and that
\begin{align*}
\overline {\overline E}_{1221}
&\leq     {C\|f\|_{L^\infty(M)}\over \mu(B)} \mu(B^m) {r^2_B\over r_{B^m}^2} \leq  {C\|f\|_{L^\infty(M)}\over \mu(B)} r^{m-2}_{B^m} r_B^2 \leq {C\|f\|_{L^\infty(M)}\over \mu(B)} r_B^n \leq C\|f\|_{L^\infty(M)}.
\end{align*}

\begin{align*}
\overline {\overline E}_{122}\leq C\|f\|_{L^\infty(M)}.
 \end{align*}

As a consequence, we obtain that
\begin{align*}
\overline {\overline E}_{12}
\leq C\|f\|_{L^\infty(M)}.
 \end{align*}

Then we consider $E_2$.  Note that
\begin{align*}
E_2&\leq {\|f\|_{L^\infty(M)}\over \mu(B)}\sum_{K_0<k\leq 0} \int_{T_k}  \, \int_{0}^{r_B} \int_0^\infty {st \over (s+t)^2} \int_{\R^m\backslash K} |p_{t+s,2}(x,y)|\, d\mu(y)\, {dt\over t}{ds\over s} \, d\mu(x).
\end{align*}
Recall that for $x\in T_k$ and $y\in\R^m\backslash K$, we have that
\begin{align*}
|p_{t+s,2}(x,y)|&\leq C\frac{(t+s)^2}{(t+s+d(x,y))^{m+2}}
+C\frac{1}{t^n|y|^{m-2}}\frac{(t+s)^2}{(t+s+d(x,y))^{n+2}},
\end{align*}
which gives
\begin{align*}
E_2&\leq {C\|f\|_{L^\infty(M)}\over \mu(B)}\sum_{K_0<k\leq 0} \int_{T_k}  \, \int_{0}^{r_B} \int_0^\infty {st \over (s+t)^2} \int_{\R^m\backslash K} \frac{(t+s)^2}{(t+s+d(x,y))^{m+2}}\, d\mu(y)\, {dt\over t}{ds\over s} \, d\mu(x)\\
&\quad+{C\|f\|_{L^\infty(M)}\over \mu(B)}\sum_{K_0<k\leq 0} \int_{T_k}  \, \int_{0}^{r_B} \int_0^\infty {st \over (s+t)^2} \int_{\R^m\backslash K} \frac{1}{|y|^{m-2}}\frac{(t+s)^2}{(t+s+d(x,y))^{n+2}}\, d\mu(y)\, {dt\over t}{ds\over s} \, d\mu(x)\\
&=:E_{21}+E_{22}.
\end{align*}
We first consider $E_{21}$. We write
\begin{align*}
E_{21}&\leq {C\|f\|_{L^\infty(M)}\over \mu(B)}\sum_{K_0<k\leq 0} \int_{T_k}  \, \int_{0}^{r_B} \int_0^{r_B} {st \over (s+t)^2} \int_{\R^m\backslash K} \frac{(t+s)^2}{(t+s+d(x,y))^{m+2}}\, d\mu(y)\, {dt\over t}{ds\over s} \, d\mu(x)\\
&\quad+{C\|f\|_{L^\infty(M)}\over \mu(B)}\sum_{K_0<k\leq 0} \int_{T_k}  \, \int_{0}^{r_B} \int_{r_B}^\infty {st \over (s+t)^2} \int_{\R^m\backslash K} \frac{(t+s)^2}{(t+s+d(x,y))^{m+2}}\, d\mu(y)\, {dt\over t}{ds\over s} \, d\mu(x)\\
&=:E_{211}+E_{212}.
\end{align*}
For the term $E_{212}$,
since
$$ \int_{\R^m\backslash K} \frac{(t+s)^2}{(t+s+d(x,y))^{m+2}}\, d\mu(y) \leq C,$$ we obtain that
\begin{align*}
E_{212}&\leq {C\|f\|_{L^\infty(M)}\over \mu(B)}\sum_{K_0<k\leq 0} \int_{T_k}  \, \int_{0}^{r_B} \int_{r_B}^\infty {st \over (s+t)^2} \, {dt\over t}{ds\over s} \, d\mu(x)\\
&\leq {C\|f\|_{L^\infty(M)}\over \mu(B)}\sum_{K_0<k\leq 0} \mu(T_k) \\
&\leq {C\|f\|_{L^\infty(M)}\over \mu(B)}\sum_{K_0<k\leq 0} 2^{kn}\mu(B) \\
&\leq C\|f\|_{L^\infty(M)}.
\end{align*}
For the term $E_{211}$,   we obtain that
\begin{align*}
E_{211}&\leq {C\|f\|_{L^\infty(M)}\over \mu(B)}\sum_{K_0<k\leq 0} \int_{T_k}  \, \int_{0}^{r_B} \int_0^{r_B} \int_{\R^m\backslash K} \frac{1}{d(x,y)^{m+2}}\, d\mu(y)\, {dt}{ds} \, d\mu(x)\\
&\leq {C\|f\|_{L^\infty(M)}\over \mu(B)}\cdot \sum_{K_0<k\leq 0}  {1\over (2^kr_B)^2} \int_{T_k}  \, \int_{0}^{r_B} \int_0^{r_B}   {dt}{ds} \, d\mu(x)\\
&\leq {C\|f\|_{L^\infty(M)}\over \mu(B)}\sum_{K_0<k\leq 0} {1\over 2^{2k}} \mu(T_k) \\
&\leq {C\|f\|_{L^\infty(M)}\over \mu(B)}\sum_{K_0<k\leq 0} 2^{k(n-2)}\mu(B) \\
&\leq C\|f\|_{L^\infty(M)},
\end{align*}
where the second inequality follows from the fact that $x\in T_k$ and from decomposing $\R^m\backslash K$
into annuli according to the scale of $2^kr_B$.

The term $E_{22}$ can be handled by using similar approach and hence combing all the cases above we obtain that
$E_2\leq C\|f\|_{L^\infty(M)}$.  Combining with the estimate for $E_1$, we obtain that
$E\leq C\|f\|_{L^\infty(M)}$.

We now consider the term $F$ as defined in \eqref{B12}. We again write
$f_2(x)=f_{21}(x)+f_{22}(x)$ with $f_{21}(x)= f_2(x)\cdot \chi_{4B}(x)$ and $f_{22}(x)= f_2(x)\cdot \chi_{(4B)^c}(x)$.
\begin{align*}
F&\leq {1\over \mu(B)} \int_B \big| \big( I-e^{-r_B\sqrt{L}}   \big)\tilde m(\sqrt{L})f_{21}(x) \big|\, d\mu(x)\\
&\quad+{1\over \mu(B)} \int_B \big| \big( I-e^{-r_B\sqrt{L}}   \big)\tilde m(\sqrt{L})f_{22}(x) \big|\, d\mu(x).
\end{align*}
We point out that the estimate for $F$ can be handled by following the same approach and techniques of those for the term $E$ in {\bf Case 1}. For the detail we omit here.

Combining all the estimates of $E$ and $F$ we obtain that \eqref{B1} holds in {\bf Case 2}. Hence  \eqref{B1} holds.

\medskip

We now consider \eqref{B0}. Recall that for every $B\in \mathcal{B}_0^2$, we have that $B$ is centered in  $\Rf^n$
with $r_B\geq 2, K\subset B,  r_B^{\frac{n-1}{m-1}}<r_{B^m}<r_B \}.
$

We now decompose $B$ following the same way as in {\bf Case 2} of \eqref{B1} above,

Then we have that the left-hand side of \eqref{B0} is bounded by
\begin{align*}
&{1\over \log r_B\ \mu(B)} \int_B \big| \big( I-e^{-r_B\sqrt{L}}   \big)\tilde m(\sqrt{L})f_1(x) \big|\, d\mu(x)+{1\over \log r_B\ \mu(B)} \int_B \big| \big( I-e^{-r_B\sqrt{L}}   \big)\tilde m(\sqrt{L})f_2(x) \big|\, d\mu(x)\\
&=:\tilde E+\tilde F.
\end{align*}
with $f_1$ and $f_2$ the same as in \eqref{f1f2}. It is clear that for the term $\tilde F$, by using the same estimate as that for $F$ in {\bf Case 2} above and using the fact that $r_B\geq2$, we obtain that $\tilde F\leq C\|f\|_{L^\infty(M)}$.

For the term $\tilde E$, we further decompose it as
\begin{align*}
\tilde E&\leq{1\over \log r_B\ \mu(B)} \int_{B^m} \big| \big( I-e^{-r_B\sqrt{L}}   \big)\tilde m(\sqrt{L})f_1(x) \big|\, d\mu(x)\\
&\quad+ \sum_{K_0<k\leq 0}{1\over \log r_B\ \mu(B)} \int_{T_k} \big| \big( I-e^{-r_B\sqrt{L}}   \big)\tilde m(\sqrt{L})f_1(x) \big|\, d\mu(x)\\
&=\tilde E_1+\tilde E_2.
\end{align*}
Again, for the term $\tilde E_2$, by using the same estimate as that for $E_2$ in {\bf Case 2} above and using the fact that $r_B\geq2$, we obtain that $\tilde E_2\leq C\|f\|_{L^\infty(M)}$.

For the term $\tilde E_1$, following the estimate for $E_1$ in {\bf Case 2}, we denote the radius of $B^m$ by $r_{B^m}$, then we further control it by $\tilde E_{11}+\tilde E_{12}$. Again, the term $\tilde E_{11}$ can be estimate by using the same approach as for $E_{11}$ in  {\bf Case 2} of \eqref{B1} above.

For the term  $\tilde E_{12}$, we further control it as
\begin{align*}
\overline {\overline {\tilde E}}_{12}
&\leq  {C\|f\|_{L^\infty(M)}\over \log r_B\mu(B)} \int_{B^m} \int_{r_{B^m}}^{r_B}\int_0^\infty {st \over (s+t)^2} \int_{\R^m\cap (4B^m)^c}  {(t+s)^2\over (t+s+d(x,y))^{m+2}}\ \,d\mu(y)\, {dt\over t}{ds\over s} \, d\mu(x)\\
&\quad+ {C\|f\|_{L^\infty(M)}\over \log r_B\mu(B)} \int_{B^m} \int_{r_{B^m}}^{r_B}\int_0^\infty {st \over (s+t)^2} \int_{\R^m\cap (4B^m)^c} {(t+s)^2\over |y|^{m-2} (t+s+|y|)^{n+2}}\ \,d\mu(y)\, {dt\over t}{ds\over s} \, d\mu(x)\\
&=\overline {\overline {\tilde E}}_{121}+\overline {\overline{\tilde  E}}_{122}.
\end{align*}
We first consider $\overline {\overline {\tilde E}}_{121}$.
By noting that there exists a positive constant $C$ such that for every $s,t\in (0,\infty)$,
$$\int_{\R^m\cap (4B^m)^c}  {(t+s)^2\over (t+s+d(x,y))^{m+2}}\,d\mu(y) \leq C,$$
we have
\begin{align*}
\overline {\overline {\tilde E}}_{121}
&\leq  {C\|f\|_{L^\infty(M)}\over \log r_B\mu(B)} \int_{B^m} \int_{r_{B^m}}^{r_B}\int_0^\infty {st \over (s+t)^2} \, {dt\over t}{ds\over s} \, d\mu(x)\\
&\leq  {C\|f\|_{L^\infty(M)}\over\log r_B \mu(B)} \int_{B^m} \int_{r_{B^m}}^{r_B} {ds\over s} \, d\mu(x)\\
&\leq C\|f\|_{L^\infty(M)} { \mu(B^m)\over \log r_B\mu(B)   } \log\bigg( {r_B\over r_{B^m}} \bigg)\\
&\leq C\|f\|_{L^\infty(M)},
\end{align*}
where the second inequality follows from the direct calculation via splitting the integration $\int_0^\infty = \int_0^s+\int_s^\infty$ and the last inequality follows from the condition that  $r_{B^m}$ has a positive lower bound $r_B^{\frac{n-1}{m-1}}$.

For the term $\overline {\overline {\tilde E}}_{122}$, again,
by noting that there exists a positive constant $C$ such that for every $s,t\in (0,\infty)$,
$$\int_{\R^m\cap (4B^m)^c}  {(t+s)^2\over |y|^{m-2} (t+s+|y|)^{n+2}}\,d\mu(y) \leq C,$$
and following the same estimate as that for $\overline {\overline {\tilde E}}_{121}$ above, we have that
\begin{align*}
\overline {\overline {\tilde E}}_{122}\leq C\|f\|_{L^\infty(M)}.
\end{align*}

As a consequence, we obtain that
\begin{align*}
\tilde E_{12}
\leq C\|f\|_{L^\infty(M)}.
 \end{align*}

Combining all the estimates of $\tilde E$ and $\tilde F$ we obtain that \eqref{B0} holds.
The proof of Theorem \ref{main3} is complete.
\end{proof}

\bigskip
\noindent
{\bf Acknowledgements}: P. Chen was supported by NNSF of China 11501583, Guangdong Natural Science Foundation
 2016A030313351 and the Fundamental Research Funds for the Central Universities 161gpy45.
Duong and Li were supported by the Australian Research Council (ARC) through the
research grant DP160100153 and by Macquarie University Research Seeding Grant. L. Song is supported in part by
 Guangdong Natural Science Funds for Distinguished Young Scholar
(No. 2016A030306040) and  NNSF of China (Nos 11471338 and 11622113).
 L.X. Yan was supported by the NNSF
of China, Grant No.  ~11521101, and Guangdong Special Support Program.

\end{document}